\RequirePackage{fix-cm}
\documentclass[smallextended]{svjour3}       
\smartqed  

\usepackage[a4paper, total={7.0in, 10in}]{geometry}
\usepackage{graphicx}

\usepackage{amssymb}

\usepackage{amssymb}
\usepackage{isomath}
\usepackage{mathtools}
\usepackage{tikz}
\usepackage{pgfplots}
\usetikzlibrary{arrows.meta}
  \pgfplotsset{compat=newest}
  \usetikzlibrary{plotmarks}
  \usepackage{grffile}

\pgfplotsset{every axis/.append style={
        scaled ticks = false, 
        tick label style={/pgf/number format/fixed}
    }
}

\usetikzlibrary{external}
\newlength\figureheight 
\newlength\figurewidth 
\usetikzlibrary{patterns}
\usetikzlibrary{plotmarks}
\usetikzlibrary{patterns}
\usepackage{tikz-3dplot}
\usepackage{multirow}
\usepackage{subfigure}
\usepackage{esint} 
\usepackage{multirow}
\usepackage{booktabs}
\usepackage{ulem}
\usepackage{cancel}
\usepackage[autostyle]{csquotes} 
\usepackage{stackengine} 
\usepackage{float}
\usepackage{setspace}
\usepackage{placeins}
\usepackage{lipsum}
\usepackage{amsfonts}
\usepackage{stmaryrd} 
\usepackage{algorithm2e}
\usepackage{adjustbox}
\usepackage{hyperref}
\usepackage{pdflscape}

\begin{document}





\title{The Hermite-Taylor Correction Function Method for Embedded Boundary and Maxwell's Interface Problems}

\titlerunning{The HT-CF Method for Embedded Boundary and Maxwell's Interface Problems}        

\author{Yann-Meing Law \and Daniel Appel\"{o} \and Thomas Hagstrom
}


\institute{Y.-M. Law  \at
              Department of Mathematics and Industrial Engineering, Polytechnique Montr\'eal, C.P. 6079, succ. Centre-ville, Montr\'eal, QC H3C 3A7, Canada,
              \email{yann-meing.law@polymtl.ca}   \\
              D. Appel\"{o} \at 
              Department of Mathematics, Virginia Tech, Blacksburg, 24060, VA, USA,
              \email{appelo@vt.edu} \\
              T. Hagstrom \at 
              Department of Mathematics, Southern Methodist University, Dallas, 75275, TX, USA,
              \email{thagstrom@smu.edu}      
}

\date{Received: date / Accepted: date}

\maketitle



\begin{abstract}
We propose a novel Hermite-Taylor correction function method to handle embedded boundary and 
	interface conditions for Maxwell's equations. 
The Hermite-Taylor method evolves the electromagnetic fields and their derivatives through order $m$ in each
Cartesian coordinate.
This makes the development of a systematic approach to enforce boundary and interface conditions difficult.
Here we use the correction function method to update the numerical solution where the Hermite-Taylor method 
	cannot be applied directly. 
Time derivatives of boundary and interface conditions,
    converted into spatial derivatives, 
    are enforced to obtain a stable method and relax the time-step size restriction of the Hermite-Taylor correction function method.
The proposed high-order method offers a flexible systematic approach to handle embedded boundary and interface problems, 
	including problems with discontinuous solutions at the interface.
This method is also easily adaptable to other first order hyperbolic systems.


\end{abstract}

%


\section{Introduction}

Interface and boundary problems are of great importance in electromagnetics. 
The former type of problem involves interfaces between different media 
	and is found in many applications in electromechanics, 
	biophotonics and
	magneto-hydrodynamics, 
	to name a few.
The latter type of problem focuses on the interaction between 
	the electromagnetic fields and a surface,
	and is found in waveguide applications. 
	
In computational electromagnetics, 
	many challenges arise from those types of problems.
The development of efficient high-order methods are important to diminish the phase error for 
	long time simulations \cite{Hesthaven2003}.
This is particularly difficult for interface problems where the solution can be discontinuous at the interface. 
In addition to high-order accuracy, 
	a numerical method should also be able to handle complex geometries.
Many high-order numerical methods have been developed to handle Maxwell's interface and boundary problems, 
	such as finite-difference time-domain (FDTD) methods 
	\cite{Ditkowski2001,Cai2003,Zhao2004,Zhang2016,Banks2020,LawNave2022}, 
	discontinuous Galerkin (DG) methods \cite{cockburn2001runge,Hesthaven2002}
	and pseudo-spectral methods \cite{Yang1997,Driscoll1999,Fan2002,Galagusz2016}. 
	
In this work, 
	we focus on the Hermite-Taylor method, 
	introduced by Goodrich, Hagstrom and Lorenz in 2005 \cite{Goodrich2005}.
This high-order method is particularly well-suited for linear hyperbolic problems with periodic boundary conditions.
 By evolving in time the variables and their derivatives through order $m$ in each coordinate,
	the Hermite-Taylor method achieves a $(2\,m+1)$ rate of convergence.
The stability condition of this method depends only on the largest wave speed of the problem and is 
	 independent of the order of accuracy, 
	 making the Hermite-Taylor method appealing for large scale problems.

The difficulty in designing a systematic approach to handle the boundary conditions has prevented the use of the Hermite-Taylor method
	for many engineering and real-world scientific problems. 
Hybrid DG-Hermite methods were proposed to circumvent this issue by taking advantage of a DG method to 
	handle the boundary conditions on complex geometries \cite{Chen2010,OversetHermiteDG}. 	
A DG solver is used on an unstructured or boundary fitted curvilinear mesh 
	which encloses a Cartesian mesh where the Hermite method is applied. 
A local time-stepping procedure is needed to retain the large time-step sizes of the Hermite method. 
	
Another method based on compatibility boundary conditions was developed for the wave equation on 
	Cartesian and curvilinear meshes \cite{compat_wave_hermite_AAL_DEAA_WDH}. 
In $d$ dimensions, 
	this method computes the $(m+1)^d$ degrees of freedom on the boundary by enforcing the physical boundary 
	condition as well as the compatibility boundary conditions. 
However, 
	the extension of this method to first order hyperbolic systems is not straightforward {\color{black} due to the need for a characteristic decomposition and the presence of characteristic variables for the Maxwell system.} 

In \cite{LawAppelo2023},
	the Hermite-Taylor correction function method was proposed to handle general boundary conditions
	for Maxwell's equations.
This method relies on the correction function method (CFM) to update the numerical solution and its derivatives 
where the Hermite-Taylor method cannot be applied. 
The correction function method seeks space-time polynomial approximations of the solution in small domains,
    called local patches, 
    near the boundary or the interface
	by minimizing a functional.
The functional is based on a square measure of the residual of the governing equations, 
	boundary or interface conditions and the numerical solution from the original method (here the Hermite-Taylor method).
 The CFM minimization procedure provides polynomial approximations,
	also called correction functions, 
	that are used to compute the $(m+1)^d$ degrees of freedom at the nodes where the Hermite-Taylor method cannot be applied.
The CFM was first developed to enhance finite-difference methods for Poisson's equation 
	with interface conditions \cite{Marques2011,Marques2017}.
It has been extended to the wave equation \cite{Abraham2018} and 
	Maxwell's equations \cite{LawMarquesNave2020,LawNave2021,LawNave2022}.
 
 From a CFM point of view, the Hermite-Taylor setting provides several advantages compared to finite-difference methods. 
Indeed, 
    the Hermite method uses a stencil that remains the same regardless of its order and naturally provides space-time polynomials approximating the solution that are required in 
    the CFM functional.
The Hermite-Taylor correction function method presented in \cite{LawAppelo2023} achieved up to a seventh-order rate of convergence with a loose CFL constant 
	but was limited to boundaries aligned with the nodes.
 In this paper, 
    we extend the Hermite-Taylor correction function method to embedded boundary and interface problems.
 We consider the situation where the mesh resolution allows the numerical solution from the original method to be available around the interface and boundary. 
 In other words, 
    interfaces and boundaries are sufficiently far away from each other to construct CFM local patches, 
    leaving close contact interface problems for future work.
 
The paper is organized as follows. 
We introduce Maxwell's equations with boundary and interface conditions in Section~\ref{sec:problem_definition}.
The Hermite-Taylor method in two space dimensions is presented in detail in Section~\ref{sec:HermiteTaylor}. 
The correction function method in the Hermite-Taylor setting for embedded boundary and interface problems is 
	described in Section~\ref{sec:correction_function_method}. 
Finally, 
	in Section~\ref{sec:numerical_examples},
	numerical examples in one and two space dimensions,
	including problems with discontinuous electromagnetic fields, 
	are performed to verify the properties of the Hermite-Taylor correction function method.
	

\section{Problem Definition} \label{sec:problem_definition}

In this work, 
	we are seeking numerical solutions to Maxwell's equations
\begin{equation} \label{eq:Maxwell}
\begin{aligned}
	\mu(\mathbold{x})\,\partial_t \mathbold{H} + \nabla\times \mathbold{E} =&\,\, 0, \\
	\epsilon(\mathbold{x})\,\partial_t \mathbold{E} - \nabla\times\mathbold{H} =&\,\, - \mathbold{J}, \\
	\nabla\cdot(\epsilon(\mathbold{x})\,\mathbold{E}) =&\,\, \rho , \\
	\nabla\cdot(\mu(\mathbold{x})\,\mathbold{H})=&\,\, 0, 
\end{aligned}
\end{equation}
 	in a domain $\Omega \subset \mathbb{R}^d$ ($d=1,2$) and a time interval $I=[t_0,t_f]$.
Here $\mathbold{H}$ is the magnetic field, 
	$\mathbold{E}$ is the electric field, 
	$\mu(\mathbold{x})>0$ is the magnetic permeability and $\epsilon(\mathbold{x})>0$ is the electric permittivity.
The initial conditions are given by 
\begin{equation} \label{eq:initial_cdns}
\begin{aligned}
	\mathbold{H}(\mathbold{x},0) =&\,\, \mathbold{H}_0(\mathbold{x})	\quad \text{in } \Omega, \\
	\mathbold{E}(\mathbold{x},0) =&\,\, \mathbold{E}_0(\mathbold{x})	\quad \text{in } \Omega,
\end{aligned}
\end{equation}
	and the boundary condition is given by 
\begin{equation} \label{eq:boundary_condition}
	\mathbold{n}\times\mathbold{E} = \mathbold{g}(\mathbold{x},t) \quad \text{on } \partial\Omega\times I.
\end{equation}
Here $\partial\Omega$ is the smooth boundary of the domain $\Omega$, 
	$\mathbold{n}$ is the outward unit normal to $\partial\Omega$
	and $\mathbold{g}(\mathbold{x},t)$ is a known function.
Note that the situation where $\mathbold{g} = 0$ corresponds to perfect electric conductor (PEC) boundary conditions. For results on the well-posedness of Maxwell's equations with PEC boundary conditions, we refer the reader to \cite{Assous2018}.
	
We also consider Maxwell's interface problems. 
In this situation, 
	the domain $\Omega$ is subdivided into two subdomains $\Omega^+$ and $\Omega^-$,
	and is such that $\Omega = \Omega^+ \cup \Omega^-$ and 
	$\Omega^+\cap\Omega^- = \Gamma$.
Here $\Gamma$ is the smooth interface between the subdomains. 
Fig.~\ref{fig:typical_geo} illustrates a typical geometry of a domain $\Omega$.
\begin{figure}
 	\centering
	\includegraphics[width=2.0in]{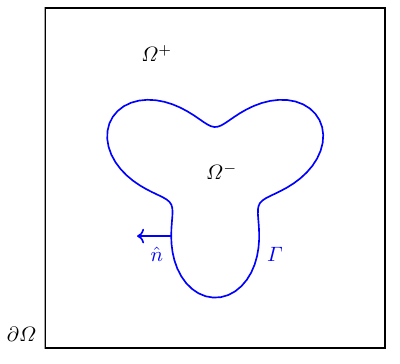}
       \caption{Geometry of a domain $\Omega$. The domain consists of two materials.}
	\label{fig:typical_geo}
\end{figure}
The physical parameters are assumed to be piecewise constant and are given by
\begin{equation}
	\mu(\mathbold{x}) = \left\{ 
  \begin{array}{l}
    \mu^+ \quad \mbox{for} \quad \mathbold{x} \in \Omega^+, \\[1pt]
    \mu^- \quad \mbox{for} \quad \mathbold{x} \in \Omega^-,
  \end{array} \right. 
\end{equation}
and
\begin{equation}
	\epsilon(\mathbold{x}) = \left\{ 
  \begin{array}{l}
    \epsilon^+ \quad \mbox{for} \quad  \mathbold{x} \in\Omega^+, \\[1pt]
    \epsilon^- \quad \mbox{for} \quad  \mathbold{x} \in\Omega^-.
  \end{array} \right. 
\end{equation}
To complete Maxwell's equations, 
	we impose the interface conditions
\begin{equation} \label{eq:interface_cdns}
	\begin{aligned}
	\hat{\mathbold{n}}\times\llbracket \mathbold{E} \rrbracket =&\,\, 0 \quad \text{on } \Gamma \times I ,\\
	\hat{\mathbold{n}}\times\llbracket \mathbold{H} \rrbracket =&\,\, 0  \quad \text{on } \Gamma \times I ,\\
	\hat{\mathbold{n}}\cdot\llbracket \epsilon\,\mathbold{E} \rrbracket =&\,\, 0 \quad \text{on } \Gamma \times I ,\\
	\hat{\mathbold{n}}\cdot\llbracket \mu\,\mathbold{H} \rrbracket =&\,\, 0 \quad \text{on } \Gamma \times I.
	\end{aligned}
\end{equation}
Here $\llbracket \mathbold{H} \rrbracket = \mathbold{H}^+ - \mathbold{H}^-$ is the jump of the variable $\mathbold{H}$ on the interface, 
	$\mathbold{H}^+$ is the solution in $\Omega^+$, 
	$\mathbold{H}^-$ is the solution in $\Omega^-$ and 
	$\mathbold{\hat{n}}$ is the unit normal to the interface $\Gamma$. 

In what follows we will assume for simplicity that $\mathbold{J}=0$, $\rho=0$. The construction of the boundary and interface conditions are unchanged so long as the currents are supported in the volume.

\section{Hermite-Taylor Method} \label{sec:HermiteTaylor}
{\color{black} 
In this section, for completeness,
	we briefly describe the Hermite-Taylor method,
	introduced by Goodrich et al. \cite{Goodrich2005}, 
	to handle linear hyperbolic problems.} 
For simplicity, 
	the Hermite-Taylor method is presented in 2-D using the transverse magnetic (TM$_z$) mode.
In this situation, 
	Maxwell's equations are simplified to
\begin{equation} \label{eq:TMz_syst}
	\begin{aligned}
		\mu\,\partial_t H_x + \partial_y E_z =&\,\, 0,\\
		\mu\,\partial_t H_y - \partial_x E_z  =&\,\, 0,\\
		\epsilon\,\partial_t E_z - \partial_x H_y + \partial_y H_x =&\,\, 0,\\
		\partial_x H_x + \partial_y H_y =&\,\, 0, 
	\end{aligned}
\end{equation}
	in a domain $\Omega =[x_\ell,x_r]\times[y_b,y_t]$ and a time interval $I = [t_0,t_f]$.
Here we assume the physical parameters $\mu$ and $\epsilon$ to be constant.
We consider initial conditions on $H_x$, 
	$H_y$ and $E_z$, 
	and periodic boundary conditions. 

The Hermite-Taylor method uses a mesh that is staggered in space and time. 
The primal mesh is defined as
\begin{equation}
	(x_i,y_j)  = (x_\ell+i\,\Delta x, y_b+j\,\Delta y), \quad i=0,\dots,N_x, \quad j=0,\dots,N_y, 
\end{equation}
	with 
\begin{equation}
	\Delta x = \frac{x_r-x_\ell}{N_x}, \quad \Delta y = \frac{y_t-y_b}{N_y}.
\end{equation}
Here $N_x$ and $N_y$ are respectively the number of cells in the $x$ and $y$ directions.
The numerical solution on the primal mesh is centered at times
\begin{equation}
	t_n = t_0 + n\,\Delta t, \quad n = 0, \dots, N_t, \quad \Delta t = \frac{t_f-t_0}{N_t}.
\end{equation}
Here $N_t$ is the required number of time steps to reach $t_f$.
The nodes of the dual mesh are located at the cell centers of the primal mesh 
\begin{equation}
	(x_{i+1/2},y_{j+1/2})  = (x_\ell+(i+1/2)\,\Delta x, y_b+(j+1/2)\,\Delta y), 
 \end{equation}
 for $i=0,\dots,N_x-1$, $j=0,\dots,N_y-1$,
	and at times 
\begin{equation}
	t_{n+1/2} = t_0 + (n+1/2)\,\Delta t, \quad n=0,\dots,N_t-1.
\end{equation}

{\color{black} The Hermite-Taylor method requires three processes to evolve the electromagnetic fields and their spatial derivatives through order $m$ from the primal mesh at $t_n$ to the dual mesh at $t_{n+1/2}$ :}
\begin{itemize}
	\item[] \underline{Hermite interpolation}
	
		Let us consider a sufficiently accurate approximation of each electromagnetic field component, 
  for example $E_z$, and its derivatives $\frac {\partial^{k+\ell} E_z}{\partial x^k \partial y^{\ell}}$, $k,\ell = 0, \ldots ,m$,
  on the primal mesh at time $t_n$.
{\color{black}		
For each cell $[x_i,x_{i+1}]\times[y_j,y_{j+1}]$ of the primal mesh and for each electromagnetic field component, 
			we compute the unique degree $(2\,m+1)^2$ tensor-product polynomial satisfying the value and 
			the given derivatives of the electromagnetic field components at the corners of the cell.}
		The resulting polynomial is known as the Hermite interpolant. 
		
	\item[] \underline{Recursion relation}
		For each cell of the primal mesh and for each electromagnetic field, 
			we identify the derivatives of the Hermite interpolant at the cell center as scaled coefficients,
	{\color{black} denoted by $c_{k,\ell}^{H_x}(t)|_{t_n}$, 
        $c_{k,\ell}^{H_y}(t)|_{t_n}$ and $c_{k,\ell}^{E_z}(t)|_{t_n}$.
    Expanding each scaled coefficient of the Hermite interpolant in time using 
        \begin{equation}
        c_{k,\ell}(t) = \sum_{s=0}^{q} c_{k,\ell,s}\bigg(\frac{t-t_n}{\Delta t}\bigg)^s
        \end{equation}
        } gives us a space-time polynomial,
			referred to as the Hermite-Taylor polynomial, 
			approximating each electromagnetic field. 
{\color{black}		We then enforce Maxwell's equations and its derivatives at the cell center to obtain a recursion relation 
			for the scaled coefficients of the Hermite-Taylor polynomials 
\begin{equation} \label{eq:recursion_relations}
	\begin{aligned}
	c_{k,\ell,s}^{H_x} =&\,\, -\frac{(\ell+1)\,\Delta t}{\mu\,s\Delta y} c^{E_z}_{k,\ell+1,s-1}, \\
	c_{k,\ell,s}^{H_y} =&\,\, \frac{(k+1)\,\Delta t}{\mu\,s\Delta x} c^{E_z}_{k+1,\ell,s-1},\\
	c_{k,\ell,s}^{E_z} =&\,\, \frac{\Delta t}{\epsilon\,s}\bigg(\frac{(k+1)}{\Delta x} c^{H_y}_{k+1,\ell,s-1} - 
				\frac{(\ell+1)}{\Delta y} c^{H_x}_{k,\ell+1,s-1}\bigg),
	\end{aligned}
\end{equation}
	for $k,\ell = 0,\dots,2\,m+1$ and $s = 1,\dots,q$.
	Here $q = \nu\,(2\,m+1)$ in $\mathbb{R}^\nu$ so the Taylor expansion in time of the coefficients of Hermite interpolants is done exactly.}
	\item[] \underline{Time evolution}
		
		For each cell and for each electromagnetic field, 
			we update the numerical solution on the dual mesh by evaluating the Hermite-Taylor polynomials
			at the cell center $(x_{i+1/2},y_{j+1/2})$ and time $t_{n+1/2}$.
\end{itemize}
{\color{black} To complete the time step, 
	we repeat a similar procedure for each cell $[x_{i-1/2},x_{i+1/2}]\times[y_{j-1/2},y_{j+1/2}]$ of the dual mesh at time $t_{n+1/2}$ and 
	update the data on the primal mesh at $(x_i,y_j)$ at $t_{n+1}$.
The whole procedure is then repeated until the final time is reached.}

{\color{black} 
	
In the situation where the derivatives cannot be easily computed,
	we project the initial solution onto a polynomial space of degree at least $2\,m+1$ to maintain accuracy.
To do so, 
	for each electromagnetic field and for each primal node $(x_i,y_j)$, 
	we define the spatial domain $[x_{i-1/2},x_{i+1/2}]\times[y_{j-1/2},y_{j+1/2}]$ and project the initial 
	solution on the space of degree $(2\,m+2)^2$ tensor-product Legendre polynomials. 
We then approximate the required derivatives of each electromagnetic field at $(x_i,y_j)$ using the derivatives of 
	the Legendre polynomials approximating the electromagnetic fields.

The enforcement of boundary conditions is challenging for the Hermite-Taylor method since all values of the electromagnetic fields 
	and their derivatives through order $m$ in normal and tangential directions must also be known on the boundary, 
	which is not the case in general. 
For an embedded boundary,
	the boundary can be located between the nodes of the mesh which further complicates the enforcement of boundary conditions.
The imposition of interface conditions shares the same issue with the additional difficulty that 
	the electromagnetic fields could be discontinuous at the interface. 
In the next section,
	we present a new avenue to handle embedded boundary and interface conditions based on the 
	correction function method.
}

\section{Correction Function Method} \label{sec:correction_function_method}

The correction function method seeks a polynomial approximating each electromagnetic field component in the vicinity of the nodes
	where the Hermite-Taylor method cannot be directly applied.
We refer to such a node as a CF node. 
The node where the numerical solution can be updated using the Hermite-Taylor method is referred to as a Hermite node. 
The correction function method relies on the minimization of a functional describing the electromagnetic fields in the vicinity of 
	a CF node.
The approximations of the electromagnetic fields are sought in a polynomial space and a careful definition of the space-time domain 
	of the polynomials approximating the electromagnetic fields is required for accuracy. 
In this section,
	we describe in detail the correction function method in 1-D for embedded boundary and interface problems. 
We then extend the method to the two-dimensional case. 

\subsection{Embedded Boundary in One Dimension}

Let us consider the following 1-D simplification of Maxwell's equations
\begin{equation} \label{eq:Maxwell_1D}
	\begin{aligned}
		\mu\,\partial_t H + \partial_x E =&\,\, 0, \\
		\epsilon\,\partial_t E + \partial_x H =&\,\, 0,
	\end{aligned}
\end{equation}
	in the domain $\Omega=[x_\ell,x_r]$ and the time interval $[t_0,t_f]$.
Here $\mu$ and $\epsilon$ are constant. 
We enforce the boundary conditions
\begin{equation}
	E(x_\ell,t) = g_\ell(t) \qquad \mbox{and} \qquad E(x_r,t) = g_r(t).
\end{equation}
We consider the physical domain $\Omega$ to be embedded in a computational domain $\Omega_c = [x_0,x_N]$.
We then have two CF nodes, 
	one for the left boundary and one for the right boundary.
For simplicity, 
	we assume that both CF nodes belong to the primal mesh.
 
For the $q^{\text{th}}$ CF node at time $t_n$, 
	we define a functional 
\begin{equation} \label{eq:functional_boundary}
	J_q^n = \mathcal{G}_q^n + \mathcal{B}_q^n + \mathcal{H}_q^n.
\end{equation} 
The first part $\mathcal{G}_q^n$ ensures that the governing equations are approximately fulfilled.
The second part $\mathcal{B}_q^n$ weakly enforces the boundary conditions.
The third part $\mathcal{H}_q^n$ weakly enforces the correction functions to match the Hermite solution.
{\color{black} Note that the scaling of each part is determined by a dimensional analysis and is detailed in Remark~\ref{rem:dim_analysis}.}

Each part of the functional $J_q^n$ is computed over different domains. 
The electromagnetic fields are required to approximately satisfy Maxwell's equations in the local patch of $J_q^n$.  
The space-time domain of the governing equations functional $\mathcal{G}_q^n$ then encloses 
	the $q^{\text{th}}$ CF node, 
	the domains of the boundary functional $\mathcal{B}_q^n$ and the Hermite functional $\mathcal{H}_q^n$.
The domain of $\mathcal{B}_q^n$ encloses the part of the boundary close to the $q^{\text{th}}$ CF node. 
Finally,
	the domain of $\mathcal{H}_q^n$ encloses the space-time domains of the closest primal Hermite and dual Hermite nodes
	to the $q^{\text{th}}$ CF node. 

{\color{black} As an example, 
	let us consider that the left boundary is located at $x_{\ell}$ between the dual node $x_{1/2}$ 
	and the primal node $x_1$, 
    so the zeroth CF node ($q=0$) is associated with $x_1$.}
In this situation, 
	we use the correction function method to update 
	the numerical approximation to the electromagnetic fields 
	and their $m$ first derivatives located at the primal CF node $x_1$ at a given time $t_n$.  
    
The governing equations functional $\mathcal{G}_0^n$ contains the residual of Maxwell's equations and is integrated over the space-time domain
	$S_0\times[t_{n-1},t_n]$.
Here the space interval is $S_0 = [x_\ell, x_{5/2}]$.
The governing equations functional is then given by 
\begin{equation}
	\mathcal{G}_0^n(H^n_{h,0},E^n_{h,0}) =  \frac{\ell_0}{2} \,\int\limits_{t_{n-1}}^{t_n} \int\limits_{S_0} (\mu\,\partial_t H^n_{h,0}+\partial_x E^n_{h,0})^2 + Z^2 (\epsilon\,\partial_t E^n_{h,0} + \partial_x H^n_{h,0})^2\,\mathrm{d}x\,\mathrm{d}t.
\end{equation}
Here $Z = \sqrt{\mu/\epsilon}$ is the impedance and $\ell_0 = x_{5/2}-x_\ell$ is the characteristic length of the local patch,
	and $H^n_{h,0}$ and $E^n_{h,0}$ are the polynomials approximating the electromagnetic fields in the local patch 
	that we seek. 
$H^n_{h,0}$ and $E^n_{h,0}$ are referred to as the correction functions.
The integration domain of $\mathcal{G}^n_0$ is illustrated in Fig.~\ref{fig:local_patch_1D_x0}.
\begin{figure}
 	\centering
	\includegraphics[width=3.0in]{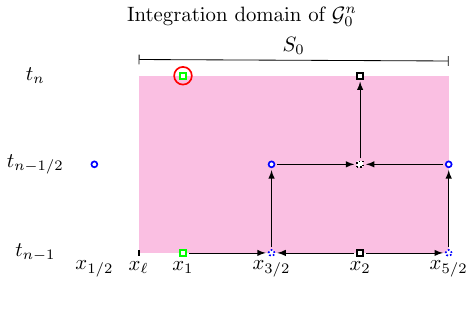}
	\caption{Illustration of the domain of integration $S_0\times [t_{n-1},t_n]$ of $\mathcal{G}_0^n$. 
		   The primal CF and Hermite nodes are respectively represented by green squares and black squares 
		   while the dual Hermite nodes are represented by blue circles.
		   The CFM seeks the information located at $(x_1,t_{n})$ which is enclosed by the red circle.
		   The space-time local patch $S_0\times [t_{n-1},t_n]$ is denoted by a magenta box.}
	\label{fig:local_patch_1D_x0}
\end{figure}

The boundary functional $\mathcal{B}_0^n$ contains the residual of the boundary condition at $x_\ell$ and 
	is integrated over the time interval $[t_{n-1},t_{n}]$.
We then have 
\begin{equation}
	\mathcal{B}_0^n(E^n_{h,0}) = \frac{1}{2} \, \int\limits_{t_{n-1}}^{t_n} (E^n_{h,0}(x_\ell,t)-g_\ell(t))^2\,\mathrm{d}t.
\end{equation}
The integration domain of $\mathcal{B}_0^n$ is illustrated in Fig.~\ref{fig:boundary_1D_x0}.
\begin{figure}
 	\centering
	\includegraphics[width=3.0in]{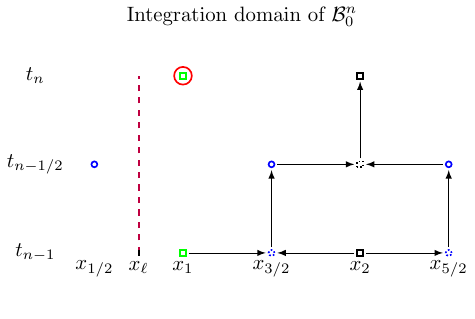}
	\caption{Illustration of the domain of integration $[t_{n-1},t_n]$ at $x_\ell$ of $\mathcal{B}_0^n$. 
		   The primal CF and Hermite nodes are respectively represented by green squares and black squares 
		   while the dual Hermite nodes are represented by blue circles.
		   The CFM seeks the information located at $(x_1,t_{n})$ which is enclosed by the red circle.
		   The intersection between the boundary and the local patch,
			that is the line connecting $(x_\ell,t_{n-1})$ to $(x_\ell,t_n)$,  
			is denoted by a dashed purple line.}	
	\label{fig:boundary_1D_x0}
\end{figure}
{\color{black} Unfortunately, 
    as shown in the numerical results of Section \ref{sec:numerical_examples}, 
the enforcement of only the boundary condition \eqref{eq:boundary_condition}
    leads either to an unstable method or a severe restriction on the CFL constant.
Although the CFM, 
    in general, 
    impacts the stability of the original method, 
    the numerical results reported in \cite{LawNave2021,LawNave2022} for finite-difference methods and 
    \cite{LawAppelo2023} for Hermite methods (limited to boundaries aligned with the mesh nodes) did not suffer from a significant CFL number reduction.

For embedded boundary and interface problems, 
    the main difference between the method proposed here and previous works on finite-difference methods is that 
    the original method is coupled to the CFM with not only the values of the correction functions but also their spatial derivatives through order $m$. 
{\color{black} This motivates us to add additional constraints on the spatial derivatives of the correction functions at the boundary using compatibility conditions for improving the stability of the overall method.
 Compatibility conditions were used to enforce boundary and interface conditions in high-order finite-difference methods for Navier-Stokes equations \cite{Henshaw1994}, 
conservation laws \cite{Tan2010}, 
wave equation \cite{Appelo2012} 
and
Maxwell's equations \cite{Zhao2004,Henshaw2006,Angel2019,Banks2020}, 
to name a few.
Recently, 
    this approach was also used to enforce boundary conditions when  dissipative and conservative Hermite methods are used for the scalar wave equation.
For the CFM, 
    we propose to use the time derivatives of boundary conditions 
    and convert them into spatial derivatives,
    similar to what is done for inverse Lax-Wendroff methods \cite{Tan2010} for conservation laws.
Note that, 
    in this situation, 
    we do not use the compatibility conditions to obtain solvable linear systems that enforce the boundary (or interface) conditions (see Propositions \ref{prop:global_minimizer_boundary} and \ref{prop:global_minimizer_interface}),
    but rather to improve the stability of the proposed Hermite-Taylor correction function method. 
} 

Taking time derivatives of the left boundary condition lead to 
\begin{equation}
    \partial_t^{j} E(x_\ell,t) = \partial_t^{j} g_\ell(t). 
\end{equation}
Using now the 1-D simplification of Maxwell's equations \eqref{eq:Maxwell_1D} to convert the time derivatives of the electric field into spatial derivatives, 
    we obtain 
\begin{equation} \label{eq:convert_time_derivatives_into_spatial_1D}
    \partial_t^{j} E = \left.\Bigg\{ 
    \begin{aligned}
        -\partial_x^{j}H/(\epsilon^\theta\,\mu^{\theta -1}), \qquad &\mbox{if } j \mbox{ odd}, \\
        \partial_x^{j}E/(\epsilon\,\mu)^{j/2}, \qquad &\mbox{otherwise}.
    \end{aligned}\right.
\end{equation}
Here $\theta = (j+1)/2$.
Thus, 
    the boundary functional becomes 
\begin{equation} \label{eq:constraint_derivatives_1D}
    \mathcal{B}_0^n(E^n_{h,0}) = \frac{1}{2} \, \sum_{j=0}^{N_d}\, \bigg(\frac{\ell_0}{c}\bigg)^{2j}\int\limits_{t_{n-1}}^{t_n} (\partial_t^{j}E^n_{h,0}(x_\ell,t)-\partial_t^{j}g_\ell(t))^2\,\mathrm{d}t.
\end{equation}
Here $c = 1/\sqrt{\epsilon\,\mu}$ is the wave speed, 
    and $N_d$ is the maximum order of the derivatives that are considered. 
}

The Hermite functional 
\begin{equation}
	\mathcal{H}_0^n = \mathcal{H}_{p,0}^n + \mathcal{H}_{d,0}^n,
\end{equation}
	weakly enforces the correction function to match the Hermite solution.
The first part $\mathcal{H}_{p,0}^n$ weakly enforces the correction function to match the Hermite-Taylor 
	polynomial associated 
	with the primal Hermite node $x_2$ in the space-time domain
	$S_{0,p}^\mathcal{H} \times[t_{n-1/2},t_n]$.
Here the space interval $S_{0,p}^\mathcal{H} = [x_{3/2},x_{5/2}]$.
We obtain 
\begin{equation}
	\mathcal{H}_{p,0}^n(H^n_{h,0},E^n_{h,0}) = \frac {1}{2} \frac{c_H}{\Delta x}\int\limits_{t_{n-\frac {1}{2}}}^{t_n}\int\limits_{S_{0,p}^{\mathcal{H}}} Z^2 (H^n_{h,0}-H^*)^2 + (E^n_{h,0}-E^*)^2\,\mathrm{d}x\,\mathrm{d}t.
\end{equation}
Here $c_H >0$ is a given penalization parameter, 
	and $H^*$ and $E^*$ are Hermite-Taylor polynomials.
The second term $\mathcal{H}_{d,0}^n$ weakly enforces the Hermite-Taylor polynomial associated with the dual Hermite node 
	$x_{3/2}$ in the space-time domain $S_{0,d}^\mathcal{H} \times [t_{n-1},t_{n-1/2}]$.
Here the space interval is $S_{0,d}^\mathcal{H} = [x_{1},x_{2}]$.
We then have  
\begin{equation}
\mathcal{H}_{d,0}^n(H^n_{h,0},E^n_{h,0}) = \frac {1}{2} \frac{c_H}{\Delta x}\int\limits_{t_{n-1}}^{t_{n-\frac {1}{2}}}\int\limits_{S_{0,d}^{\mathcal{H}}} Z^2 (H^n_{h,0}-H^*)^2 + (E^n_{h,0}-E^*)^2\,\mathrm{d}x\,\mathrm{d}t.
\end{equation}
The integration domain of $\mathcal{H}_0^n$ is illustrated in Fig.~\ref{fig:HermiteTaylor_patch_1D_x0}.
\begin{figure}
 	\centering
	\includegraphics[width=3.0in]{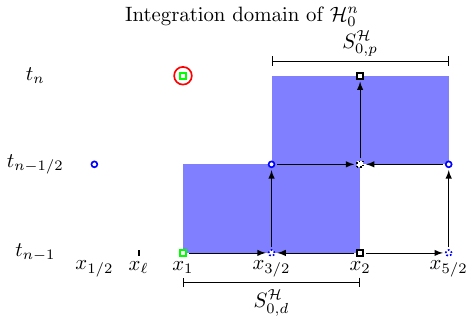}
       \caption{Illustration of the domains of integration $S_{0,p}^\mathcal{H} \times[t_{n-1/2},t_n]$ and $S_{0,d}^\mathcal{H} \times[t_{n-1},t_{n-1/2}]$ of $\mathcal{H}_0^n$. 
		   The primal CF and Hermite nodes are respectively represented by green squares and black squares 
		   while the dual Hermite nodes are represented by blue circles.
		   The CFM seeks the information located at $(x_1,t_{n})$ which is enclosed by the red circle.
		   The domains $S_{0,p}^\mathcal{H} \times [t_{n-1/2},t_n]$ and $S_{0,d}^\mathcal{H} \times [t_{n-1},t_{n-1/2}]$,
			where we enforce the correction functions to match the Hermite-Taylor polynomials,
			are denoted by blue boxes.} 
	\label{fig:HermiteTaylor_patch_1D_x0}
\end{figure}
The procedure described above can be easily adapted to weakly enforce the boundary condition at $x_r$.

{\color{black}
\begin{remark} \label{rem:dim_analysis}
The choice of the scaling of the different parts of the functional $J_q^n$ is dictated by a dimensional analysis. 
In other words, 
    the different terms share all the same dimensional unit. 
Consider  
\begin{equation}
    \mathbold{x} = L_0 \hat{\mathbold{x}}, \quad t = \frac{L_0}{c_0}\hat{t}, \quad \mathbold{E} = H_0 Z_0 \hat{\mathbold{E}}, \quad \mathbold{H} = H_0 \hat{\mathbold{H}}, \quad \mu^{\pm} = \mu_0 \hat{\mu}^{\pm}, \quad \epsilon^{\pm} = \epsilon_0 \hat{\epsilon}^{\pm},
\end{equation}
where $\hat{\mathbold{x}}$ represents a non-dimensional variable associated with $\mathbold{x}$,
    $L_0$ is the reference length, 
    $H_0$ is the reference magnetic strength, 
    $c_0 = 1/\sqrt{\epsilon_0\mu_0}$ is the dimensional speed of light in free space, 
    and $Z_0 = \sqrt{\mu_0/\epsilon_0}$. 
Substituting these in the different terms of the functional $J_q^n$ 
    shows that they share the same dimensional unit, that is $Z_0^2H_0^2L_0^d/c_0$ in $d$ space dimensions.
\end{remark}
}
\subsubsection{The Linear System of Equations that Solves the Optimization Problem}

For each CF node, 
	we solve the following minimization problem 
	\begin{equation} \label{eq:minPblm_1D}
		\begin{aligned}
			&\text{Find } (H^n_{h,q},E^n_{h,q}) \in V \times V \text{ such that }\\
 			&\qquad (H^n_{h,q},E^n_{h,q}) =  \underset{v,w\in V}{\arg\min}\, J_q^n(v,w).
		\end{aligned}
	\end{equation}
Here $V = \mathbb{Q}^k\big(S_q\times [t_{n-1},t_n] \big)$ is the space of tensor-product polynomials of degree at most $k$ in each variable,
	$n=1,\dots,N_t$ and $q=0,1$ in our 1-D example. 
We use space-time Legendre polynomials as basis functions of $V$. 
To solve the minimization problem \eqref{eq:minPblm_1D}, 
	we compute the gradient of $J_q^n$ with respect to the coefficients of the polynomials $H^n_{h,q}$ and $E^n_{h,q}$
	and use that it vanishes at a minimum.
We then obtain the linear system of equations
\begin{equation}
	M_q^n\,\mathbold{c}_q^n = \mathbold{b}_q^n.
\end{equation}
Here $M_q^n$ is a square matrix of dimension $2\,(k+1)^2$ and 
	$\mathbold{c}_q^n$ is a vector containing the polynomials coefficients.
Note that the dimension of the matrices $M_q^n$ is independent of the mesh size.

Since the boundary is invariant in time, 
	we have $M_q = M_q^n$ for all $n$ and therefore obtain one matrix per CF node.
The matrices, $M_q$,
	their scaling and their LU factorization are all precomputed.
{\color{black} Note that the matrices $M_q$ are scaled using row and column scalings. The scaling matrices are computed using only the block diagonal of $M_q$, where the diagonal blocks are $(k+1)^{d+1}\times(k+1)^{d+1}$ matrices in $d$ space dimensions.}
For each time step,
	we then have to compute the right-hand side $\mathbold{b}_q^n$, 
	perform forward and backward substitutions to find $\mathbold{c}_q^n$,
	and update the numerical solution of the electromagnetic fields and their
	first $m$ derivatives at the $q^{\text{th}}$ CF node using $H^n_{h,q}$ and $E^n_{h,q}$.
This can be done independently for each $q$.

\begin{remark}
	A similar procedure can be done to update the data located at dual CF nodes. 
    However,
        the electromagnetic fields and 
		their first $m$ space derivatives at time $t_{-1/2}$ 
        are required.
    In this work, 
        we consider that these data are provided.
    Note that it would also be possible to use Hermite-Taylor polynomials obtained with the initial data at $t_0$ to estimate the required data at $t_{-1/2}$.
\end{remark}

Assume that there are CF nodes on the primal and dual meshes.  
Given the numerical solution on the primal mesh at $t_{n-1}$ and on the dual mesh at $t_{n-3/2}$,
	the algorithm to evolve the numerical solution at $t_n$ is 
\begin{itemize}
	\item[1.] Update the numerical solution on the dual Hermite node at $t_{n-1/2}$ using the Hermite-Taylor method and 
		      store the Hermite-Taylor polynomials needed for the CFM;
	\item[2.] Update the numerical solution on the dual CF nodes at $t_{n-1/2}$ using the CFM by computing $\mathbold{b}_q^{n-1/2}$ and solving for $\mathbold{c}_q^{n-1/2}$;
	\item[3.] Update the numerical solution on the primal Hermite node at $t_{n}$ using the Hermite-Taylor method and 
		      store the Hermite-Taylor polynomials needed for the CFM;
	\item[4.] Update the numerical solution on the primal CF nodes at $t_n$ using the CFM by computing $\mathbold{b}_q^{n}$ and solving for $\mathbold{c}_q^{n}$.
\end{itemize}

\subsection{Interface in One Dimension} \label{sec:interface_1D}

Let us now consider 1-D interface problems. 
In addition to Maxwell's equations \eqref{eq:Maxwell_1D}, 
	we consider the interface conditions
\begin{equation}
	\llbracket H \rrbracket = 0 \qquad \mbox{and} \qquad \llbracket E \rrbracket = 0,
\end{equation}
	on the interface $\Gamma$ located at $x_\Gamma$.
We define the subdomains $\Omega^+ = [x_\ell, x_\Gamma]$ and $\Omega^- = [x_\Gamma,x_r]$.
The physical parameters $\mu$ and $\epsilon$ are assumed to be piecewise constant. 
In this situation, 
	we seek approximations of the electromagnetic fields in each subdomain for a given CF node.
For each CF node, 
	we then compute $H^{+,n}_{h}$ and $E^{+,n}_{h}$ approximating the electromagnetic fields in $\Omega^+$, 
	and $H^{-,n}_{h}$ and $E^{-,n}_{h}$ approximating the electromagnetic fields in $\Omega^-$.
	
For the $q^{\text{th}}$ CF node at time $t_n$, 
	we define a functional 
\begin{equation} \label{eq:functional_interface_1D}
	J_q^n = \mathcal{G}_q^{+,n} + \mathcal{G}_q^{-,n} + \mathcal{I}_q^n + \mathcal{H}_q^{+,n}  + \mathcal{H}_q^{-,n}.
\end{equation} 
The governing equations functionals $\mathcal{G}_q^{+,n}$ and $\mathcal{G}_q^{-,n}$ ensure that Maxwell's equations 
	in respectively $\Omega^+$ and $\Omega^-$ are approximately fulfilled.
The interface functional $\mathcal{I}_q^n$ weakly enforces the interface conditions. 
The Hermite functional $\mathcal{H}_q^{+,n}$ weakly enforces the correction functions $H^{+,n}_{h,q}$ and $E^{+,n}_{h,q}$ 
	to match the Hermite solution in $\Omega^+$ while $\mathcal{H}_q^{-,n}$ weakly enforces the correction functions 
	$H^{-,n}_{h,q}$ and $E^{-,n}_{h,q}$ to match the Hermite solution in $\Omega^-$.

As for embedded boundary problems,
	each part of the functional $J_q^n$ is computed in different domains. 
The domains of $\mathcal{G}_q^{+,n}$ and $\mathcal{G}_q^{-,n}$ enclose the $q^{\text{th}}$ CF node, 
	the domains of the interface functional and the Hermite functionals.
This then defines the local patch of the $q^{\text{th}}$ CF node. 
The interface functional $\mathcal{I}_q^n$ encloses the part of the interface close to the $q^{\text{th}}$ CF node. 
The Hermite functional $\mathcal{H}_q^{+,n}$ encloses the space-time domains 
	of the closest primal Hermite and dual Hermite nodes in  $\Omega^+$ to the $q^{\text{th}}$ CF node.
Finally, 
	the Hermite functional $\mathcal{H}_q^{-,n}$ encloses the domains 
	of the closest primal Hermite and dual Hermite nodes in  $\Omega^-$ to the $q^{\text{th}}$ CF node.

{\color{black}
As an example, 
	we assume an interface $\Gamma$ located at $x_\Gamma$ between the dual node $x_{i+1/2}$ and 
	the primal node $x_{i+1}$.
In this situation, 
	there are two CF nodes, 
	one primal CF node located at $x_{i+1}$ and one dual CF node located at $x_{i+1/2}$.
We now focus on the primal CF node which we assume corresponds to the zeroth CF node ($q=0$). 

The governing equations functional $\mathcal{G}_0^{+,n}$ contains the residual of Maxwell's equations 
with the parameters from $\Omega^+$ and is integrated over the domain $S_0\times[t_{n-1},t_n]$.
Here the space interval $S_0 = [x_{i-1},x_{i+5/2}]$.
We have 
\begin{equation}
	\mathcal{G}_0^{+,n}(H^{+,n}_{h,0},E^{+,n}_{h,0}) =  \frac{\ell_0}{2} \,\int\limits_{t_{n-1}}^{t_n} \int\limits_{S_0} (\mu^+\,\partial_t H^{+,n}_{h,0}+\partial_x E^{+,n}_{h,0})^2 + (Z^{+})^2 (\epsilon^+\,\partial_t E^{+,n}_{h,0} + \partial_x H^{+,n}_{h,0})^2\,\mathrm{d}x\,\mathrm{d}t.
\end{equation}
Here the characteristic length of the local patch is $\ell_0 = x_{i+5/2}- x_{i-1}$.
The governing equations functional $\mathcal{G}_0^{-,n}$ is defined on the same domain as the functional $\mathcal{G}_0^{+,n}$
	but contains the residual of Maxwell's equations with the parameters from $\Omega^-$.
We then have 
\begin{equation}
	\mathcal{G}_0^{-,n}(H^{-,n}_{h,0},E^{-,n}_{h,0}) = \frac{\ell_0}{2} \,\int\limits_{t_{n-1}}^{t_n} \int\limits_{S_0} (\mu^-\,\partial_t H^{-,n}_{h,0}+\partial_x E^{-,n}_{h,0})^2 
 + (Z^{-})^2 (\epsilon^-\,\partial_t E^{-,n}_{h,0} + \partial_x H^{-,n}_{h,0})^2\,\mathrm{d}x\,\mathrm{d}t.
\end{equation}
The domain of integration of the governing equations functionals $\mathcal{G}_0^{+,n}$ and $\mathcal{G}_0^{-,n}$ is shown 
	in Fig.~\ref{fig:local_patch_1D_interface}.
\begin{figure}
 	\centering
	\includegraphics[width=4.75in]{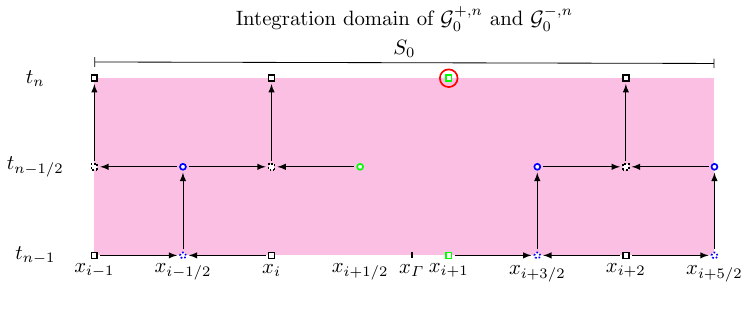}
	\caption{Illustration of the domain of integration $S_0\times [t_{n-1},t_n]$ of $\mathcal{G}_0^{+,n}$ and $\mathcal{G}_0^{-,n}$. 
		   The primal CF and Hermite nodes are respectively represented by green squares and black squares 
		   while the dual CF and Hermite nodes are represented by green circles and blue circles.
		   The CFM seeks the information located at $(x_{i+1},t_{n})$ which is enclosed by the red circle.
		   The space-time local patch $S_0\times [t_{n-1},t_n]$ is denoted by a magenta box.}
	\label{fig:local_patch_1D_interface}
\end{figure}

The interface functional $\mathcal{I}_0^n$ contains the residual of the interface conditions, which in this case we take to be continuity for
$E$ and $H$, and is integrated over 
	the time interval $[t_{n-1},t_n]$.
{\color{black} We then have 
\begin{equation}
\mathcal{I}_0^n(H^{+,n}_{h,0},E^{+,n}_{h,0},H^{-,n}_{h,0},E^{-,n}_{h,0}) = \frac{1}{2} \, \sum_{j=0}^{N_d}\, \bigg(\frac{\ell_0}{c}\bigg)^{2j} \int\limits_{t_{n-1}}^{t_n} \bar{Z}^2 \llbracket \partial_t^j H^{n}_{h,0}(x_\Gamma,t) \rrbracket^2 + \llbracket \partial_t^j E^{n}_{h,0}(x_\Gamma,t) \rrbracket^2\,\mathrm{d}t.
\end{equation}
Here 
\begin{equation} \label{eq:convert_time_derivatives_into_spatial_1D_H}
    \partial_t^{j} H = \left.\Bigg\{ 
    \begin{aligned}
        -\partial_x^{j}E/(\epsilon^{\theta -1}\,\mu^\theta), \qquad &\mbox{if } j \mbox{ odd}, \\
        \partial_x^{j}H/(\epsilon\,\mu)^{j/2}, \qquad &\mbox{otherwise}.
    \end{aligned}\right.
\end{equation}
Note that the interface functional couples the electromagnetic fields from the different subdomains at the interface and that we convert the time derivatives of the electromagnetic fields into spatial derivatives using \eqref{eq:convert_time_derivatives_into_spatial_1D} and \eqref{eq:convert_time_derivatives_into_spatial_1D_H} to improve the stability of the Hermite-Taylor correction function method, 
as shown in the numerical results of Section \ref{sec:numerical_examples}.} We can take $\bar{Z}=(Z^{+}+Z^{-})/2$ or the values from the left or right as convenient. 
The integration domain of $\mathcal{I}_0^n$ is illustrated in Fig.~\ref{fig:boundary_1D_interface}.
\begin{figure}
 	\centering
	\includegraphics[width=4.75in]{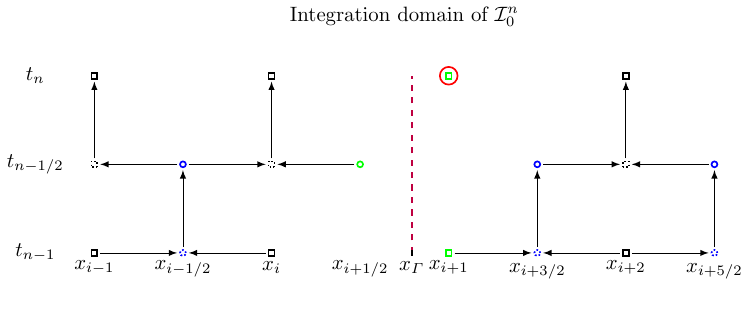}
	\caption{Illustration of the domain of integration $[t_{n-1},t_n]$ at $x_\Gamma$ of $\mathcal{I}_0^n$. 
		   The primal CF and Hermite nodes are respectively represented by green squares and black squares 
		   while the dual CF and Hermite nodes are represented by green circles and blue circles.
		   The CFM seeks the information located at $(x_{i+1},t_{n})$ which is enclosed by the red circle.
		   The intersection between the interface and the local patch,
			that is the line connecting $(x_{\Gamma},t_{n-1})$ to $(x_\Gamma,t_n)$,  
			is denoted by a dashed purple line.}	
	\label{fig:boundary_1D_interface}
\end{figure}

The Hermite functional $\mathcal{H}_0^{+,n}$ weakly enforces the correction functions 
	$H^{+,n}_{h,0}$ and $E^{+,n}_{h,0}$ to match the Hermite solution in $\Omega^+$ 
	over the domains $S_{p,0}^{\mathcal{H}^+}\times [t_{n-1/2},t_n]$ and $S_{d,0}^{\mathcal{H}^+}\times [t_{n-1},t_{n-1/2}]$.
Here $S_{p,0}^{\mathcal{H}^+} = [x_{i+3/2},x_{i+5/2}]$ is the space interval of the primal Hermite node $x_{i+2}$
	and  $S_{d,0}^{\mathcal{H}^+} = [x_{i+1},x_{i+2}]$ is the space interval of the dual Hermite node $x_{i+3/2}$.
We then have 
\begin{equation}
	\mathcal{H}_0^{+,n} = \mathcal{H}_{p,0}^{+,n}+\mathcal{H}_{d,0}^{+,n},
\end{equation}
	with
\begin{equation}
\begin{aligned}
	\mathcal{H}_{p,0}^{+,n}(H^{+,n}_{h,0},E^{+,n}_{h,0}) =&\,\, \frac {1}{2} \frac{c_H}{\Delta x}\int\limits_{t_{n-\frac {1}{2}}}^{t_n}\int\limits_{S_{0,p}^{\mathcal{H}^+}} (Z^{+})^2 (H^{+,n}_{h,0}-H^*)^2 + (E^{+,n}_{h,0}-E^*)^2\,\mathrm{d}x\,\mathrm{d}t, \\
	\mathcal{H}_{d,0}^{+,n}(H^{+,n}_{h,0},E^{+,n}_{h,0}) =&\,\, \frac {1}{2} \frac{c_H}{\Delta x}\int\limits_{t_{n-1}}^{t_{n-\frac {1}{2}}}\int\limits_{S_{0,d}^{\mathcal{H}^+}} (Z^{+})^2 (H^{+,n}_{h,0}-H^*)^2 + (E^{+,n}_{h,0}-E^*)^2\,\mathrm{d}x\,\mathrm{d}t.
\end{aligned}
\end{equation}
Finally, 
	the Hermite functional $\mathcal{H}_0^{-,n}$ weakly enforces the correction functions $H^{-,n}_{h,0}$
	and $E^{-,n}_{h,0}$ to match the Hermite solution in $\Omega^-$ over the domains 
	$S_{p,0}^{\mathcal{H}^-}\times [t_{n-1/2},t_n]$ and $S_{d,0}^{\mathcal{H}^-}\times [t_{n-1},t_{n-1/2}]$.
Here the space intervals $S_{p,0}^{\mathcal{H}^-} = [x_{i-1/2},x_{i+1/2}]$ and $S_{d,0}^{\mathcal{H}^-} = [x_{i-1},x_i]$.
We obtain 
\begin{equation}
	\mathcal{H}_0^{-,n} = \mathcal{H}_{p,0}^{-,n}+\mathcal{H}_{d,0}^{-,n},
\end{equation}
	with
\begin{equation}
\begin{aligned}
	\mathcal{H}_{p,0}^{-,n}(H^{-,n}_{h,0},E^{-,n}_{h,0}) =&\,\, \frac {1}{2} \frac{c_H}{\Delta x}\int\limits_{t_{n-\frac {1}{2}}}^{t_n}\int\limits_{S_{0,p}^{\mathcal{H}^-}} (Z^{-})^2 (H^{-,n}_{h,0}-H^*)^2 + (E^{-,n}_{h,0}-E^*)^2\,\mathrm{d}x\,\mathrm{d}t, \\
	\mathcal{H}_{d,0}^{-,n}(H^{-,n}_{h,0},E^{-,n}_{h,0}) =&\,\, \frac {1}{2} \frac{c_H}{\Delta x} \int\limits_{t_{n-1}}^{t_{n-\frac {1}{2}}}\int\limits_{S_{0,d}^{\mathcal{H}^-}} (Z^{-})^2 (H^{-,n}_{h,0}-H^*)^2 + (E^{-,n}_{h,0}-E^*)^2\,\mathrm{d}x\,\mathrm{d}t.
\end{aligned}
\end{equation}
\begin{figure}
 	\centering
	\includegraphics[width=4.75in]{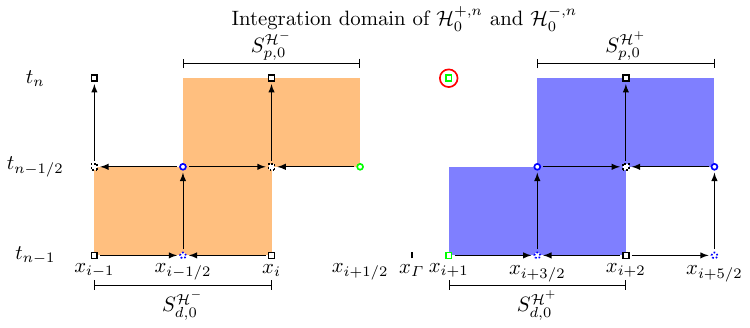}
       \caption{Illustration of the domains of integration $S_{0,p}^{\mathcal{H}^+} \times[t_{n-1/2},t_n]$, 
       		$S_{0,d}^{\mathcal{H}^+} \times[t_{n-1},t_{n-1/2}]$, 
       		$S_{0,p}^{\mathcal{H}^-} \times[t_{n-1/2},t_n]$ and $S_{0,d}^{\mathcal{H}^-} \times[t_{n-1},t_{n-1/2}]$ of 
		$\mathcal{H}_0^{+,n}$ and $\mathcal{H}_0^{-,n}$. 	   
		   The primal CF and Hermite nodes are respectively represented by green squares and black squares 
		   while the dual CF and Hermite nodes are represented by green circles and blue circles.
		   The CFM seeks the information located at $(x_{i+1},t_{n})$ which is enclosed by the red circle.
		   The domains of $\mathcal{H}_0^{+,n}$ and $\mathcal{H}_0^{-,n}$,
			where we enforce the correction functions to match the Hermite-Taylor polynomials,
			are denoted respectively by blue boxes and orange boxes.}
	\label{fig:HermiteTaylor_patch_1D_interface}
\end{figure}

A similar procedure is used to define the functional $J^{n-1/2}$ associated with the dual CF node $x_{i+1/2}$.
}

\subsubsection{The Linear System of Equations that Solves the Optimization Problem}

For each CF node, 
	we solve the following minimization problem 
	\begin{equation} \label{eq:minPblm_1D_interface}
		\begin{aligned}
			&\text{Find } (H^{+,n}_{h,q},E^{+,n}_{h,q},H^{-,n}_{h,q},E^{-,n}_{h,q}) \in V \times V \times V \times V \text{ such that }\\
 			&\qquad (H^{+,n}_{h,q},E^{+,n}_{h,q},H^{-,n}_{h,q},E^{-,n}_{h,q}) =  \underset{v^+,w^+,v^-,w^-\in V}{\arg\min}\, J_q^n(v^+,w^+,v^-,w^-).
		\end{aligned}
	\end{equation}
We solve the minimization problem \eqref{eq:minPblm_1D_interface} using a procedure similar to that of the 
	embedded boundary case.
We therefore have the same properties as before, 
	except that the dimension of the resulting linear system becomes $4\,(k+1)^2$.
The algorithm of the Hermite-Taylor correction function method to evolve the numerical solution remains the same as 
	for the embedded boundary case. 

\subsection{Multi-Dimensional Case}

In this subsection, 
	we extend the Hermite-Taylor correction function method to two and three dimensions for embedded boundary 
	and interface problems.

\subsubsection{Computation of the Local Patches}

In the multi-dimensional case, 
	the time component of the local patches remains the same for all CF nodes while 
	their spatial components are adapted to the geometry of the boundary (or interface).
The spatial component $S_q$ of the local patch associated with the $i^{\text{th}}$ CF node needs to satisfy the following three constraints:
\begin{itemize}
	\item[1.] The $i^{\text{th}}$ CF node must be inside $S_q$; 
	\item[2.] The part of the boundary (or interface) closest to the $i^{\text{th}}$ CF node must be included in $S_q$; 
	\item[3.] The cells of the Hermite nodes closest to the $i^{\text{th}}$ CF node must be included in $S_q$.
\end{itemize}
To reduce the number of minimization problems, 
    we base the construction of the local patches on a parametrization of the boundary or interface and associate multiple CF nodes to a local patch. 
{\color{black} In the following, 
    the subscript $q$ is now associated with the $q^{\text{th}}$ local patch that contains multiple CF nodes.}
    
Focusing in detail on two space dimensions for simplicity,  
	we define the dimension of the spatial component $S_q$ to be $\beta\,h \times \beta\,h$.
Here $h=\Delta x = \Delta y$ is the mesh size and $\beta$ is a given positive constant that depends on the geometry of the boundary (or interface). 
Since the update of the numerical solution with the Hermite-Taylor method for one half time step uses a five-node stencil,
	we therefore have one layer of CF nodes inside the domain along the boundary.
As for the interface case,
	we have one layer of CF nodes inside each subdomain along the interface.

To construct the local patches, 
    consider $\Gamma$ to be a parametrized curve with respect to the parameter $s \in [s_a, s_b]$.
We want to find the nodes $s_q$, 
    discretizing $\Gamma$.
The arc length $\Delta L_\Gamma$ between two nodes $s_q$ and $s_{q+1}$ on $\Gamma$ is 
\begin{equation}
    \Delta L_\Gamma = \int_{s_q}^{s_{q+1}} \sqrt{\bigg(\frac{dx}{ds}\bigg)^2+\bigg(\frac{dy}{ds}\bigg)^2} ds.
\end{equation}
Starting at $s_0 = s_a$, 
    requiring $\Delta L_\Gamma = \alpha \, h$ and approximating the integral with the trapezoidal rule, 
    the remaining nodes $s_q$ are approximated recursively using the secant method.
Here $\alpha$ is a positive constant.

In the vicinity of each node $s_q$ a local patch, $S_q$, is constructed.  
    We first find the closest primal CF node to $(x(s_q),y(s_q))$ and center $S_q$ at the spatial coordinate of this primal CF node,
    denoted as $\mathbold{x}_{s_q}$.
The spatial component of the resulting local patch then encloses 
	the closest part of the boundary (or interface) to the primal CF node located at $\mathbold{x}_{s_q}$ as well as its closest Hermite cells.
For each of the remaining primal and dual CF nodes, 
    we find $q$ such that the distance to $\mathbold{x}_{s_q}$ is minimized and associate the corresponding local patch $S_q$ to it.
For a reasonable value of $\alpha$, 
    the requirements of the spatial component of the local patch should be satisfied for all CF nodes within a local patch. 
Once the space-time domains of the local patches are computed, 
	it is easy to identify the primal and dual Hermite nodes located inside the local patches and 
	compute the space-time regions for the Hermite functionals. 
There are two functionals $\mathcal{J}$ associated with a space-time local patch, 
    one for the primal CF nodes and one for the dual CF nodes.
Note that the difference is due to the time integration of the Hermite functionals $\mathcal{H}_p$ and $\mathcal{H}_d$. 
Fig.~\ref{fig:local_patch_embedded_boundary_2D} and Fig.~\ref{fig:local_patch_interface_2D} illustrate 
	examples of a local patch with $\beta = 5$ for respectively an embedded boundary problem and an interface problem.

\begin{figure}   
	\centering
	\begin{adjustbox}{max width=1.0\textwidth,center}
		\includegraphics[width=2.5in]{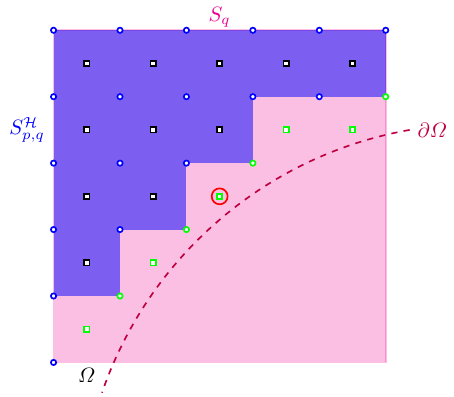}
	   	\includegraphics[width=2.5in]{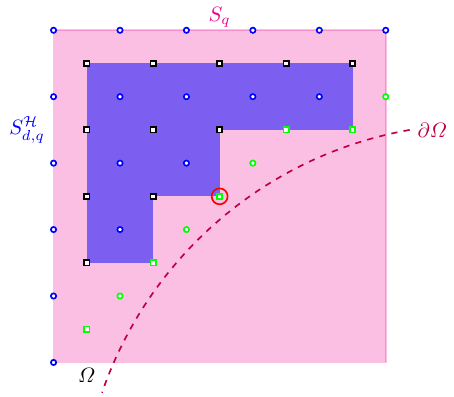}
	\end{adjustbox} 
       \caption{Illustration of a 2-D local patch for an embedded boundary problem. 
       The local patch is associated with the node $s_q$ enclosed by the black circle and centered at the primal CF node enclosed by the dashed red diamond.
       The primal CF and Hermite nodes are respectively represented by green squares and black squares 
		   while the dual CF and Hermite nodes are represented by green circles and blue circles.
	  The CFM seeks the information located at the CF nodes enclosed by a red diamond or circle.
		   The spatial domain $S_q$ of the local patch is denoted by a magenta box.
		   The boundary $\partial \Omega$ is denoted by a dashed purple line.
		   The domains $S_{p,q}^{\mathcal{H}}$ and $S_{d,q}^{\mathcal{H}}$ where we enforce the correction functions to match the Hermite-Taylor polynomials
			are denoted by blue boxes.
		   The left plot is for the spatial components of the local patch over the time interval $[t_{n-1/2},t_{n}]$ while the right plot is for $[t_{n-1},t_{n-1/2}]$ when we update primal CF nodes at $t_n$.
     For the update of dual CF nodes at $t_{n+1/2}$, 
        the time intervals for the left and right plots are respectively 
        $[t_{n-1/2},t_{n}]$ and $[t_{n},t_{n+1/2}]$.}
       \label{fig:local_patch_embedded_boundary_2D}
\end{figure}
\begin{figure}   
	\centering
	\begin{adjustbox}{max width=1.0\textwidth,center}
		\includegraphics[width=2.5in]{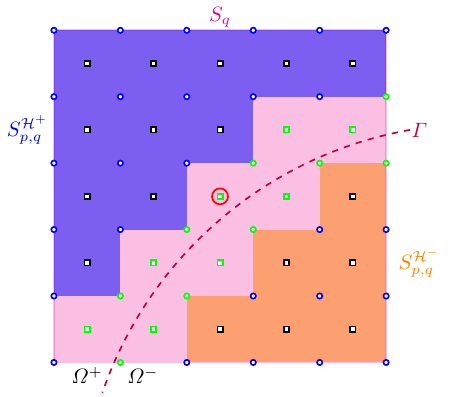}
	   	\includegraphics[width=2.5in]{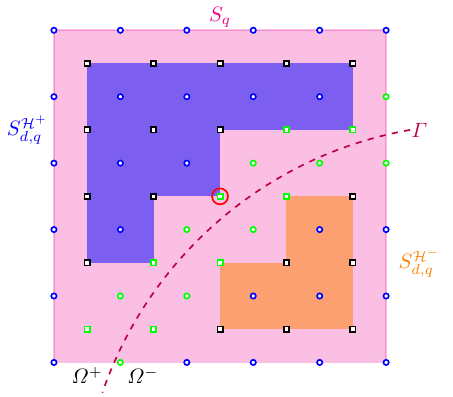}
	\end{adjustbox} 
       \caption{Illustration of a 2-D local patch for an interface problem. The local patch is associated with the node $s_q$ enclosed    by the black circle and centered at the primal CF node enclosed by the   dashed red diamond.
       		   The primal CF and Hermite nodes are respectively represented by green squares and black squares 
		   while the dual CF and Hermite nodes are represented by green circles and blue circles.
		   The CFM seeks the information located at the CF nodes enclosed by a red diamond or circle. 
		   The spatial domain $S_q$ of the local patch is denoted by a magenta box.
		   The interface $\Gamma$ is denoted by a dashed purple line.
		   The domains $S_{p,q}^{\mathcal{H}^+}$ and $S_{d,q}^{\mathcal{H}^+}$ where we enforce the correction functions 
		   	to match the Hermite-Taylor polynomials in $\Omega^+$ are denoted by blue boxes.
		   The domains $S_{p,q}^{\mathcal{H}^-}$ and $S_{d,q}^{\mathcal{H}^-}$ are denoted by orange boxes.
		   The left plot is for the spatial components of the local patch over the time interval $[t_{n-1/2},t_{n}]$ while the right plot is for $[t_{n-1},t_{n-1/2}]$ when we update primal CF nodes at $t_n$.
     For the update of dual CF nodes at $t_{n+1/2}$, 
        the time intervals for the left and right plots are respectively 
        $[t_{n-1/2},t_{n}]$ and $[t_{n},t_{n+1/2}]$.}
       \label{fig:local_patch_interface_2D}
\end{figure}

\subsubsection{Definition of the Correction Function Functional}

Let us first focus on the embedded boundary case.
For simplicity, 
	we consider that all the CF nodes belong to the primal mesh.
The described procedure below can be easily adapted to define the functional for a dual CF node.  
For each local patch, 
	we define the functional \eqref{eq:functional_boundary} to seek the polynomials $\mathbold{H}_{h,q}^n$ and $\mathbold{E}_{h,q}^n$ 
	approximating the electromagnetic fields in the local patch. 
The governing equations functional becomes 
\begin{equation}
	\begin{aligned}
		\mathcal{G}_q^n(\mathbold{H}^n_{h,q},\mathbold{E}^n_{h,q}) =&\,\,  \frac{\ell_q}{2} \,\int\limits_{t_{n-1}}^{t_n}\int\limits_{S_q} (\mu\,\partial_t \mathbold{H}^n_{h,q} + 
		\nabla\times\mathbold{E}^n_{h,q})\cdot(\mu\,\partial_t \mathbold{H}^n_{h,q} + \nabla\times\mathbold{E}^n_{h,q}) \\
		+&\,\, Z^2 ( \epsilon\,\partial_t \mathbold{E}^n_{h,q} - \nabla\times\mathbold{H}^n_{h,q})\cdot( \epsilon\,\partial_t \mathbold{E}^n_{h,q} - 
		\nabla\times\mathbold{H}^n_{h,q}) \\
		+&\,\,  c^2 (\nabla\cdot (\mu\, \mathbold{H}^n_{h,q}))^2 +\frac {1}{\epsilon^2} (\nabla\cdot (\epsilon\, \mathbold{E}^n_{h,q}))^2\,\mathrm{d}\mathbold{x}\,\mathrm{d}t.
	\end{aligned}
\end{equation}
Here $\ell_q = \beta\,h$ is the characteristic length of the local patch.
The second part of the functional $J_q^n$ that weakly enforces the boundary condition \eqref{eq:boundary_condition} becomes 
\begin{equation}
		\mathcal{B}_q^n(\mathbold{E}^n_{h,q}) = \frac{1}{2} \, \sum_{j=0}^{N_d}\, \bigg(\frac{\ell_q}{c}\bigg)^{2s}\,\int\limits_{t_{n-1}}^{t_n}\int\limits_{\partial\Omega\cap S_q} (\mathbold{n}\times\partial_t^j\mathbold{E}^n_{h,q} - \partial_t^j\mathbold{g})\cdot(\mathbold{n}\times\partial_t^j\mathbold{E}^n_{h,q} - \partial_t^j\mathbold{g})\, \mathrm{d}s\,\mathrm{d}t.
\end{equation}
Here 
\begin{equation}
    \partial_t^{j} \mathbold{E} = \left.\Bigg\{ 
    \begin{aligned}
        \nabla^{j-1}(\nabla\times \mathbold{H})/(\epsilon^{\theta}\,\mu^{\theta-1}), \qquad &\mbox{if } j \mbox{ odd}, \\
        -\nabla^{j-2}(\nabla\times\nabla\times \mathbold{E})/(\epsilon\,\mu)^{j/2}, \qquad &\mbox{otherwise},
    \end{aligned}\right.
\end{equation}
and $\theta = (j+1)/2$.
Finally, 
	the two parts in the Hermite functional $\mathcal{H}_{q}^n$ become
\begin{equation} 
	\begin{aligned}
		\mathcal{H}_{p,q}^n(\mathbold{H}^n_{h,q},\mathbold{E}^n_{h,q}) =&\,\, \frac {1}{2} \frac{c_H}{h} \int\limits_{t_{n-\frac {1}{2}}}^{t_n}\int\limits_{S^\mathcal{H}_{p,q}} Z^2 (\mathbold{H}^n_{h,q}-\mathbold{H}^*)\cdot(\mathbold{H}^n_{h,q}-\mathbold{H}^*) +(\mathbold{E}^n_{h,q}-\mathbold{E}^*)\cdot(\mathbold{E}^n_{h,q}-\mathbold{E}^*)\,\mathrm{d}\mathbold{x}\,\mathrm{d}t, \\
		\mathcal{H}_{d,q}^n(\mathbold{H}^n_{h,q},\mathbold{E}^n_{h,q}) =&\,\, \frac {1}{2}\frac{c_H}{h} \int\limits_{t_{n-1}}^{t_{n-\frac {1}{2}}}\int\limits_{S^\mathcal{H}_{d,q}} Z^2 (\mathbold{H}^n_{h,q}-\mathbold{H}^*)\cdot(\mathbold{H}^n_{h,q}-\mathbold{H}^*) + (\mathbold{E}^n_{h,q}-\mathbold{E}^*)\cdot(\mathbold{E}^n_{h,q}-\mathbold{E}^*)\,\mathrm{d}\mathbold{x}\,\mathrm{d}t.
	\end{aligned}
\end{equation}

As for the interface case,
	the functionals of the governing equations and the Hermite functionals in \eqref{eq:functional_interface_1D} are modified in the same way as the embedded boundary case. 
The interface functional that weakly enforces the interface conditions \eqref{eq:interface_cdns} becomes  
\begin{equation}
	\begin{aligned}
	\mathcal{I}_q^n(\mathbold{H}^{+,n}_{h,q},\mathbold{E}^{+,n}_{h,q},\mathbold{H}^{-,n}_{h,q},\mathbold{E}^{-,n}_{h,q}) =&\,\, \frac{1}{2} \,\sum_{j=0}^{N_d}\, \bigg(\frac{\ell_q}{\bar{c}}\bigg)^{2j}\,\int\limits_{t_{n-1}}^{t_n}\int\limits_{\Gamma\cap S_q} \left( \hat{\mathbold{n}}\times\llbracket\partial_t^j\mathbold{E}^n_{h,q}\rrbracket \right) \cdot \left( \hat{\mathbold{n}}\times\llbracket\partial_t^j\mathbold{E}^n_{h,q} \rrbracket \right) \\
			+&\,\, \bar{Z}^2 \left( \hat{\mathbold{n}}\times\llbracket\partial_t^j\mathbold{H}^n_{h,q}\rrbracket \right) \cdot \left( \hat{\mathbold{n}}\times\llbracket\partial_t^j\mathbold{H}^n_{h,q}\rrbracket \right) \\
			+&\,\,  \frac {1}{\bar{\epsilon}^2} (\hat{\mathbold{n}}\cdot\llbracket \epsilon\,\partial_t^j\mathbold{E}^n_{h,q} \rrbracket)^2 
			+ \bar{c}^2 (\hat{\mathbold{n}}\cdot\llbracket \mu\,\partial_t^j\mathbold{H}^n_{h,q} \rrbracket)^2 \, \mathrm{d}s\,\mathrm{d}t. 
	\end{aligned}
\end{equation}
Here again the barred coefficients can be taken as averages or values from either side of the interface, 
    and 
\begin{equation}
    \partial_t^{j} \mathbold{H} = \left.\Bigg\{ 
    \begin{aligned}
        -\nabla^{j-1}(\nabla\times \mathbold{E})/(\epsilon^{\theta-1}\,\mu^{\theta}), \qquad &\mbox{if } j \mbox{ odd}, \\
        -\nabla^{j-2}(\nabla\times\nabla\times \mathbold{H})/(\epsilon\,\mu)^{j/2}, \qquad &\mbox{otherwise}.
    \end{aligned}\right.
\end{equation}

\subsubsection{The Linear System of Equations that Solves the Optimization Problems}

For each local patch and for each time step, 
	we solve the following minimization problems 
\begin{equation} \label{eq:minPblm_3D}
	\begin{aligned}
		&\text{Find } (\mathbold{H}^n_{h,q},\mathbold{E}^n_{h,q}) \in V \times V \text{ such that }\\
		&\qquad (\mathbold{H}^n_{h,q},\mathbold{E}^n_{h,q}) =  \underset{\mathbold{v},\mathbold{w}\in V}{\arg\min}\, J_q^n(\mathbold{v},\mathbold{w}),
	\end{aligned}
\end{equation}
for the embedded boundary case and 
	\begin{equation} \label{eq:minPblm_3D_interface}
		\begin{aligned}
			&\text{Find } (\mathbold{H}^{+,n}_{h,q},\mathbold{E}^{+,n}_{h,q},\mathbold{H}^{-,n}_{h,q},\mathbold{E}^{-,n}_{h,q}) \in V \times V \times V \times V \text{ such that }\\
 			&\qquad (\mathbold{H}^{+,n}_{h,q},\mathbold{E}^{+,n}_{h,q},\mathbold{H}^{-,n}_{h,q},\mathbold{E}^{-,n}_{h,q}) =  \underset{\mathbold{v}^+,\mathbold{w}^+,\mathbold{v}^-,\mathbold{w}^-\in V}{\arg\min}\, J_q^n(\mathbold{v}^+,\mathbold{w}^+,\mathbold{v}^-,\mathbold{w}^-),
		\end{aligned}
	\end{equation}
	for the interface case.
Here, in general,  
\begin{equation}
	V = \{ \mathbold{v} \in [\mathbb{Q}^k(S_q\times[t_{n-1},t_n])]^3\},
\end{equation} 
with the obvious reduction in dimensions if TM or TE modes are evolved. 
We solve the minimization problems \eqref{eq:minPblm_3D} and \eqref{eq:minPblm_3D_interface} using a 
	procedure similar to that of the one-dimensional case.
We therefore compute the matrices of the linear systems of equations, 
	their scaling and LU factorizations as a pre-computation step. 
The dimension of the matrices is $3\,(k+1)^3$ for problems posed in two space dimensions and $6\,(k+1)^4$ in three space dimensions for the embedded boundary case.
As for the interface case, 
	the dimension of the matrices is $6\,(k+1)^3$ in two space dimensions and $12\,(k+1)^4$ in three space dimensions.
	
For each update of the numerical solution, 
	we therefore need to compute the right-hand side of the linear systems of equations,
	solve for the polynomial coefficients and approximate the electromagnetic fields and 
	the required spatial derivatives at the CF nodes using the correction functions. 
The algorithm of the Hermite-Taylor correction function method to evolve the numerical solution remains the same as 
	the one-dimensional case.

The following propositions guarantee that the minimization problems 
	\eqref{eq:minPblm_3D} and \eqref{eq:minPblm_3D_interface} are well-posed.
\begin{proposition} \label{prop:global_minimizer_boundary}
The minimization problem \eqref{eq:minPblm_3D} has a unique global minimizer. 
\end{proposition}
\begin{proof}
Since we are seeking the correction functions in the polynomial space $V$, 
	we have 
\begin{equation} 
	\mathbold{H}_{h,q}^n = \sum_{j} c_{q,j}^{H,n}\mathbold{\phi}_j \qquad \mbox{and} \qquad
	\mathbold{E}_{h,q}^n = \sum_{j} c_{q,j}^{E,n}\mathbold{\phi}_j.
\end{equation}
Here $c_{q,j}^{H,n}$ and $c_{q,j}^{E,n}$ are scalars,
	and $\mathbold{\phi}_j$ are basis functions of the polynomial space $V$.
The quadratic functional $J_q^n$ can therefore be written has 
\begin{equation}
	J_q^n(\mathbold{c}_q^n) = r_q^n + (\mathbold{g}_q^n)^T\mathbold{c}_q^n + \frac{1}{2} (\mathbold{c}_q^n)^TM_q\mathbold{c}_q^n.
\end{equation}
Here $\mathbold{c}_q^n$ is a vector containing the coefficients $c_{q,j}^{H,n}$ and $c_{q,j}^{E,n}$ for all $j$.
Let us now verify that $M$ is a positive definite matrix to ensure that we have a global minimizer.
Assuming $\mathbold{c}_q^n \neq 0$, 
	we notice that 
\begin{eqnarray*}
	& & \frac{1}{2} (\mathbold{c}_q^n)^TM_q\mathbold{c}_q^n = \underbrace{\mathcal{G}_q^n(\mathbold{H}^n_{h,q},\mathbold{E}^n_{h,q})}_{\geq 0} + \underbrace{\frac{1}{2} \,\int\limits_{t_{n-1}}^{t_n}\int\limits_{\partial\Omega\cap S_q} \left( \mathbold{n}\times\mathbold{E}^n_{h,q} \right) \cdot \left( \mathbold{n}\times\mathbold{E}^n_{h,q}\, \right) \mathrm{d}s\,\mathrm{d}t}_{\geq 0}  \\
      & &  +\underbrace{\frac{1}{2} \, \sum_{j=1}^{N_d}\, \bigg(\frac{\ell_q}{c}\bigg)^{2j}\,\int\limits_{t_{n-1}}^{t_n}\int\limits_{\partial\Omega\cap S_q} (\mathbold{n}\times\partial_t^j\mathbold{E}^n_{h,q})\cdot(\mathbold{n}\times\partial_t^j\mathbold{E}^n_{h,q})\, \mathrm{d}s\,\mathrm{d}t}_{\geq 0} \\
	& &	+ \underbrace{\frac{1}{2} \frac{c_H}{h} \int\limits_{t_{n-\frac {1}{2}}}^{t_n}\int\limits_{S^\mathcal{H}_{p,q}} Z^2  \mathbold{H}^n_{h,q}\cdot\mathbold{H}^n_{h,q}
		+ \mathbold{E}^n_{h,q}\cdot\mathbold{E}^n_{h,q}\,\mathrm{d}\mathbold{x}\,\mathrm{d}t}_{>0} 
		+ \underbrace{\frac {1}{2} \frac{c_H}{h} \int\limits_{t_{n-1}}^{t_{n-\frac {1}{2}}}\int\limits_{S^\mathcal{H}_{d,q}} Z^2 \mathbold{H}^n_{h,q}\cdot\mathbold{H}^n_{h,q} 
		+ \mathbold{E}^n_{h,q}\cdot\mathbold{E}^n_{h,q}\,\mathrm{d}\mathbold{x}\,\mathrm{d}t}_{>0}, 
\end{eqnarray*}
since only the zero polynomial can vanish uniformly on $S^{\mathcal{H}_{p,q}} \times (t_{n-1/2},t_n)$ or $S^{\mathcal{H}_{d,q}} \times (t_{n-1},t_{n-1/2})$.
The matrix $M_q$ is therefore positive definite and the minimization problem \eqref{eq:minPblm_3D}
	has a global minimizer. 
\end{proof}

\begin{proposition} \label{prop:global_minimizer_interface}
The minimization problem \eqref{eq:minPblm_3D_interface} has a unique global minimizer. 
\end{proposition}
\begin{proof}
The proof is similar to that of Proposition \ref{prop:global_minimizer_boundary}. 
\end{proof}

\begin{remark}
	The correction function method should preserve the accuracy of the Hermite-Taylor 
		method if $k=2\,m$.
	Let us assume that $k=2\,m$ and that polynomials approximating the correction functions have an 
	accuracy of $\mathcal{O}(\ell_i^{k+1})$. 
	Using 
	\begin{equation}
		\begin{aligned}
			\mathbold{H}^n_q =&\,\, \mathbold{H} + \mathcal{O}(\ell_q^{k+1}), \\ 
		 	\mathbold{E}^n_q =&\,\, \mathbold{E} + \mathcal{O}(\ell_q^{k+1}),
		\end{aligned} 
	\end{equation}
	in the functional $J_q^n$ we obtain that all terms in the functional $J_q^n$ scale as $\mathcal{O}(\ell_q^{2\,k+5})$. 
{\color{black} This is also consistent with the dimensional analysis made in Remark~\ref{rem:dim_analysis}.}
	Here $J_q^n$ is either given by \eqref{eq:functional_boundary} for an 
		embedded boundary problem or \eqref{eq:functional_interface_1D} for an interface problem.
	We require that $c_H > 0$ to have a well-posed minimization problem.
	We also require that $c_H \leq 1$ so that the Hermite functionals are not dominating the Maxwell's 
		equations residuals 
		and the boundary or interface conditions in $J_q^n$.
    {\color{black} We choose $c_H = 1$ since the condition number of the matrices $M_q$ scales as $\mathcal{O}(c_H^{-1})$, 
    as shown in \cite{LawAppelo2023}.}
\end{remark}

\begin{remark}

Taking $k=2\,m$, 
    the computational cost of the CFM is not negligible and increases as $m$ increases because
    the dimension of the matrices scales as $(2\,m+1)^3$ and $(2\,m+1)^4$ in respectively two and three space dimensions. 
Fortunately, 
    for a given time step, 
    the minimization problems are independent and 
    therefore can be solved in parallel,
    reducing the computational cost of the CFM. 
Note that the pre-computation step of the CFM (computation of the matrices $M_i$, their scaling and LU factorizations) can also be performed in parallel.   
We refer the reader to \cite{Abraham2017} for more information about the benefits of a parallel implementation of the CFM. 
\end{remark}

\section{Numerical Examples} \label{sec:numerical_examples}

In this section, 
	we numerically investigate the stability of the Hermite-Taylor correction function method and 
	perform convergence studies in one and two space dimensions.  

\subsection{Hermite-Taylor Correction Function Methods in One Dimension}

We consider the 1-D simplification of Maxwell's equations \eqref{eq:Maxwell_1D}. 
For a $(2\,m+1)$-order Hermite-Taylor method, 
	we use polynomials of degree $2\,m$ as the correction functions to maintain accuracy and choose $c_H = 1$.
The physical parameters are $\mu=1$ and $\epsilon = 1$.
	
\subsubsection{Stability}

The first example considers a domain where
    all the CF nodes belong to the primal mesh. 
This allows us to write the Hermite-Taylor correction function 
    method as a one step method for the solution on the primal nodes since the Hermite functionals for the primal CF nodes only require the numerical solutions at times $t_n$ and $t_{n+1/2}$ to evolve the data from $t_n$ to $t_{n+1}$.
In this situation, we express as a one-step method
\begin{equation} \label{eq:one_step_version}
	\mathbold{W}^{n+1} = A \, \mathbold{W}^n.
\end{equation}
Here $\mathbold{W}^n$ contains all the degrees of freedom on the primal mesh at time $t_n$ and $A$ is a square matrix of 
	dimension $2\,(m+1)\,(N_x+1)$.
A stable method should have all the eigenvalues of $A$ on or inside the unit circle. 
We then compute the spectral radius of $A$,
	denoted as $\rho(A)$.
{\color{black} Note that this is only a necessary condition.
That being said, 
    the stability properties of the method are also corroborated by the results of long time simulations presented at the end of this subsection.}
We consider the physical domain $\Omega = [h_{\max}-\frac{h_{\min}}{3}, 1-h_{\max}+\frac{5\,h_{\min}}{12}]$ and 
	the computational domain $\Omega_c = [0,1]$.
Here $h \in \{\frac{1}{25},\frac{1}{50},\frac{1}{100},\frac{1}{200},\frac{1}{400},\frac{1}{800},\frac{1}{1600}\}$.
	
The absolute difference between $\rho(A)$ and one is illustrated in Fig.~\ref{fig:spatial_derivative_spectral_radius_1D} for different 
    mesh sizes,
    CFL constants and values of $N_d$,
    that is the maximum order of the derivatives in the functional $\mathcal{B}$ \eqref{eq:constraint_derivatives_1D}.
Note that we consider the spatial derivatives of the correction functions in the functional $\mathcal{B}$. 
In other words, 
    we convert the time derivatives of the electromagnetic fields in 
    the functional $\mathcal{B}$ using the relations \eqref{eq:convert_time_derivatives_into_spatial_1D}.
We consider that the method is stable if the spectral radius $\rho(A)$ 
    is at most one with an error of $\mathcal{O}(10^{-10})$.
\begin{figure}   
	\centering
	\begin{adjustbox}{max width=0.8\textwidth,center}
		\includegraphics[width=2.5in,trim={0cm 0cm 1.75cm 0cm},clip]{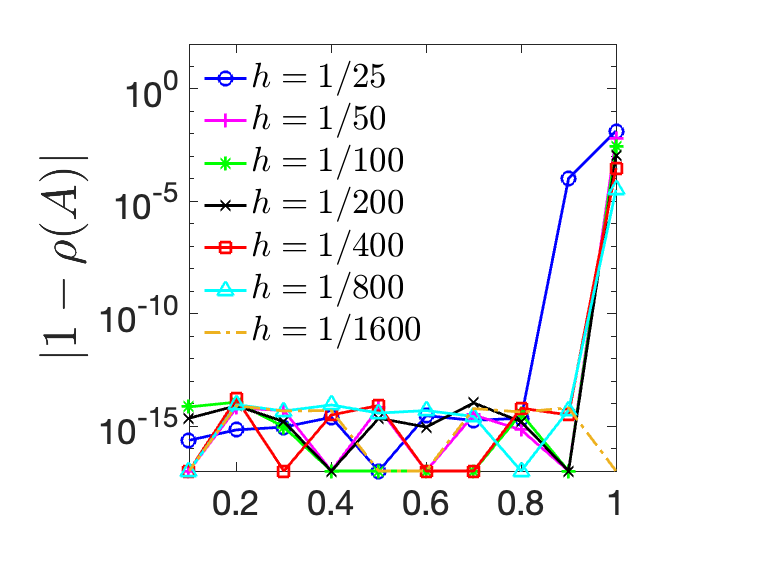} \hspace{-20.0pt}
		\includegraphics[width=2.5in,trim={0cm 0cm 1.75cm 0cm},clip]{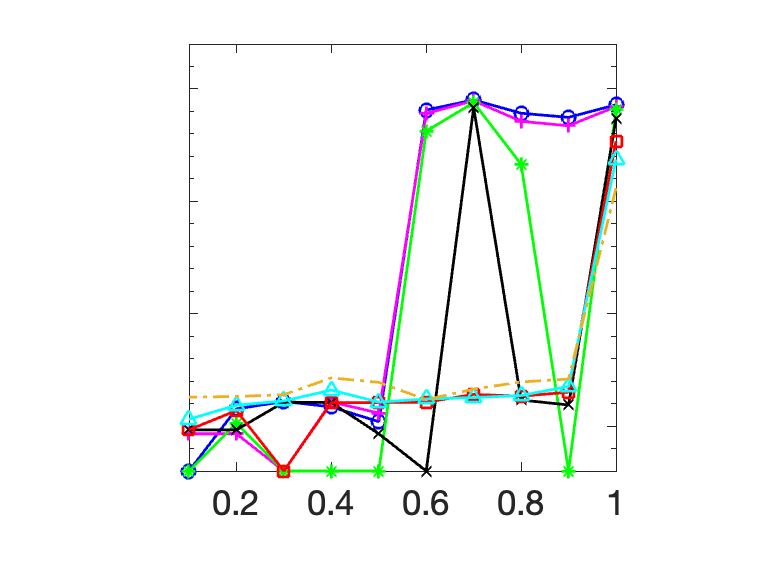}\hspace{-20pt}
		\includegraphics[width=2.5in,trim={0cm 0cm 1.75cm 0cm},clip]{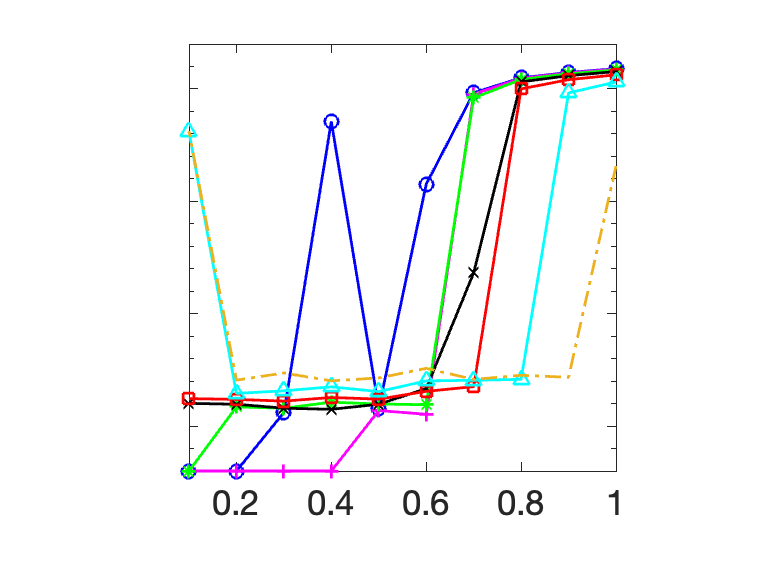}\hspace{-20pt}
		\includegraphics[width=2.5in,trim={0cm 0cm 1.75cm 0cm},clip]{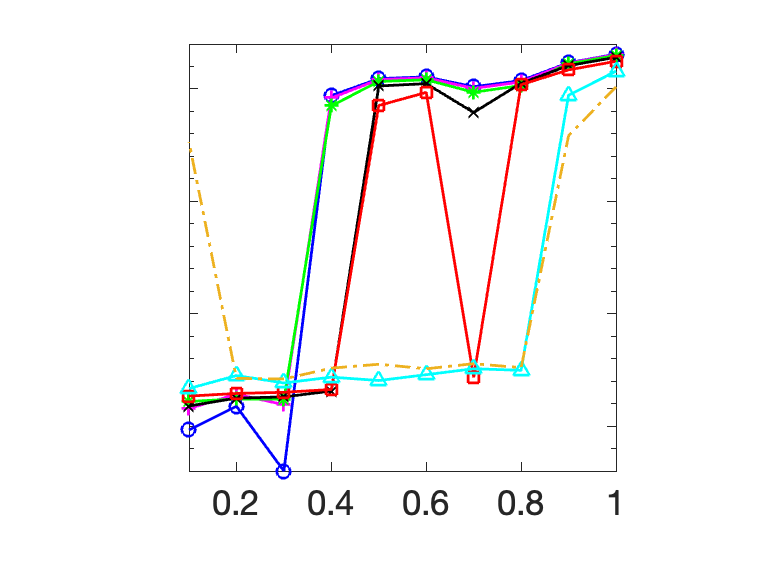}
	\end{adjustbox} 
	\begin{adjustbox}{max width=0.8\textwidth,center}
		\includegraphics[width=2.5in,trim={0cm 0cm 1.75cm 0cm},clip]{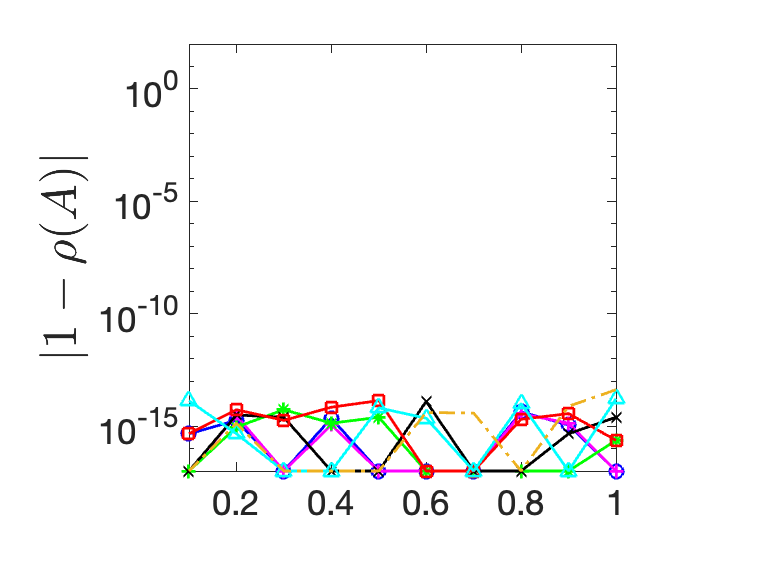} \hspace{-20pt}
		\includegraphics[width=2.5in,trim={0cm 0cm 1.75cm 0cm},clip]{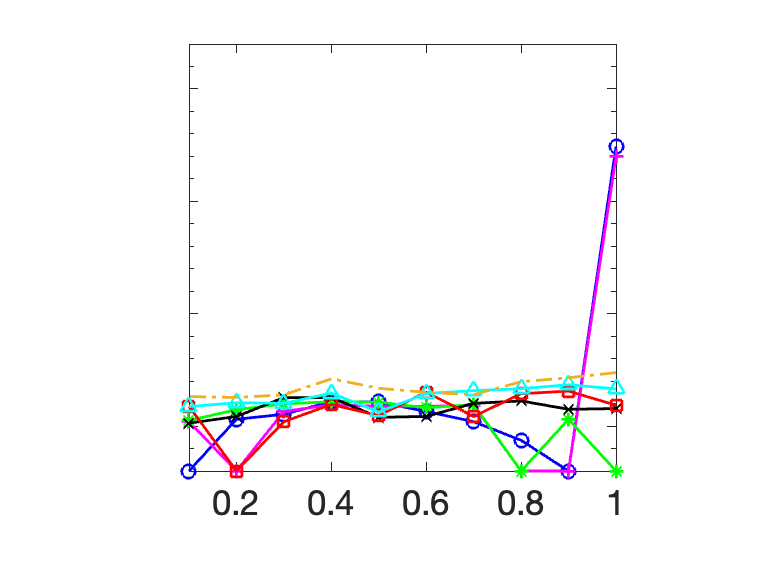}\hspace{-20pt}
		\includegraphics[width=2.5in,trim={0cm 0cm 1.75cm 0cm},clip]{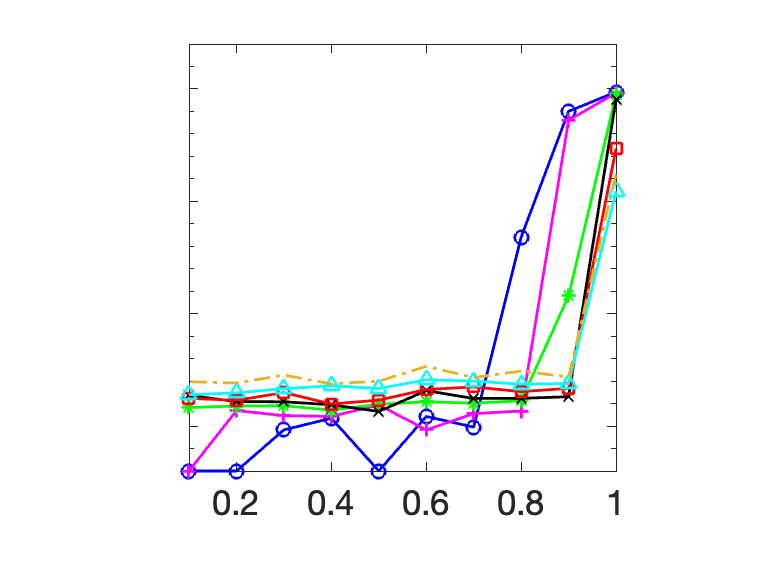}\hspace{-20pt}
		\includegraphics[width=2.5in,trim={0cm 0cm 1.75cm 0cm},clip]{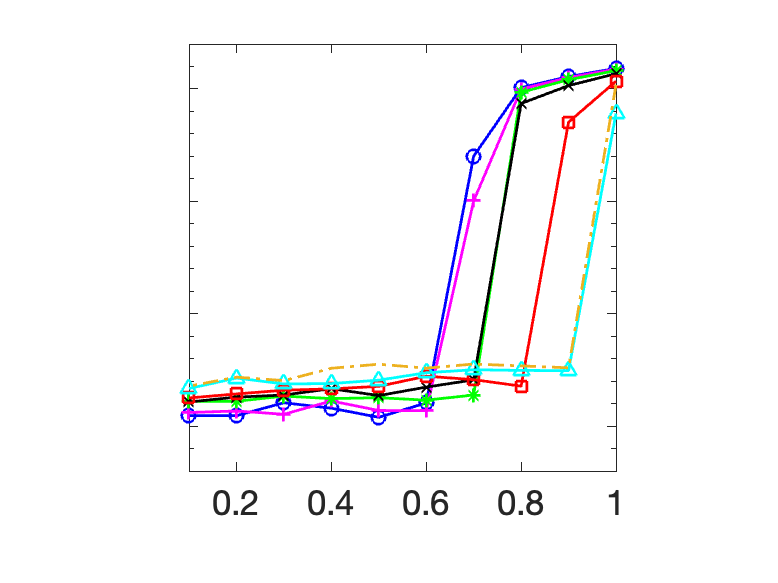}
	\end{adjustbox}  
	\begin{adjustbox}{max width=0.8\textwidth,center}
		\includegraphics[width=2.5in,trim={0cm 0cm 1.75cm 0cm},clip]{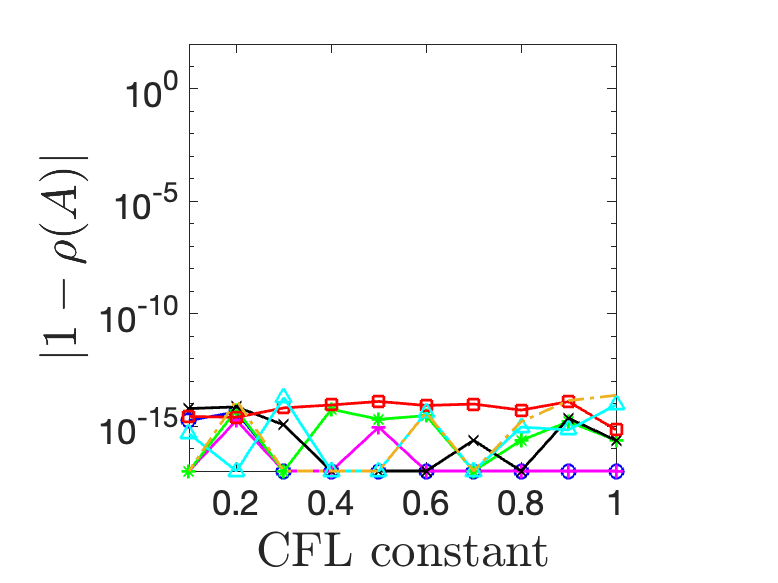} \hspace{-20pt}
		\includegraphics[width=2.5in,trim={0cm 0cm 1.75cm 0cm},clip]{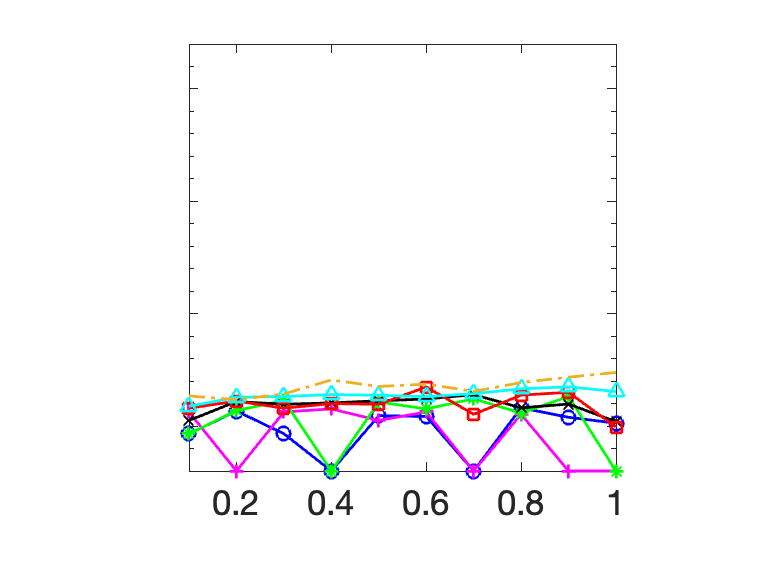}\hspace{-20pt}
		\includegraphics[width=2.5in,trim={0cm 0cm 1.75cm 0cm},clip]{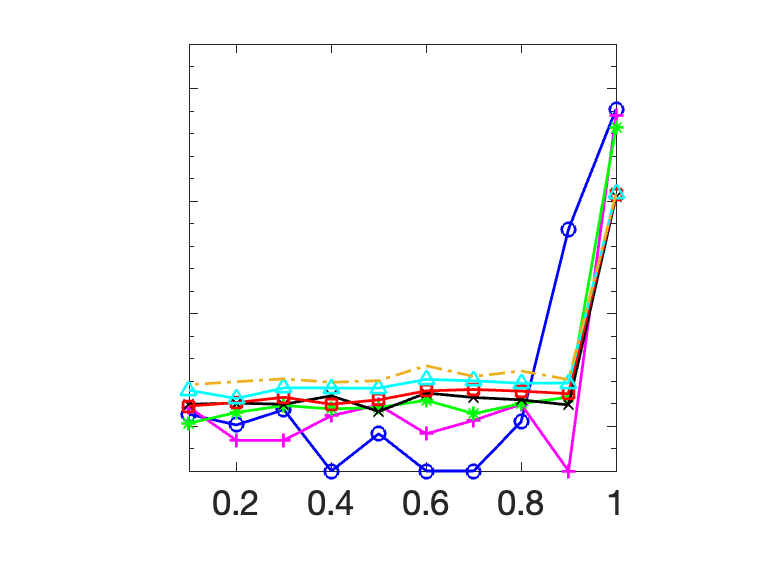}\hspace{-20pt}
		\includegraphics[width=2.5in,trim={0cm 0cm 1.75cm 0cm},clip]{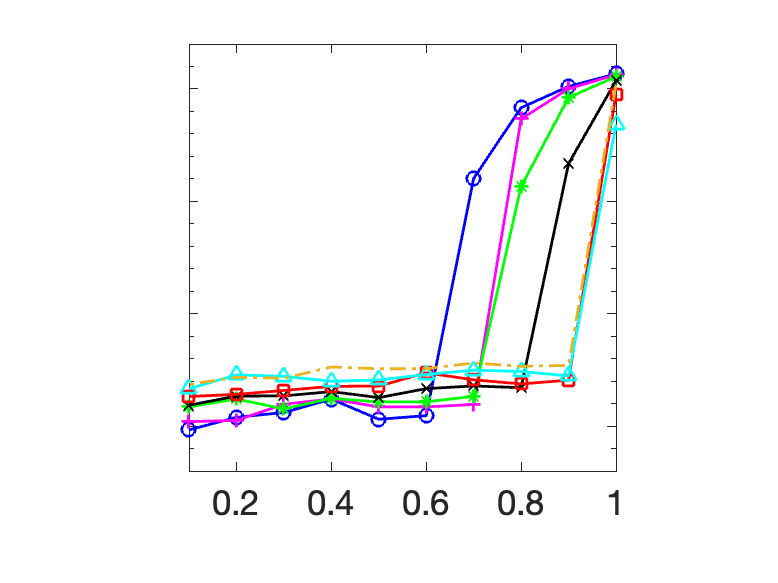}
	\end{adjustbox} 
	\begin{adjustbox}{max width=0.8\textwidth,center}
		\phantom{\includegraphics[width=2.5in,trim={0cm 0cm 1.75cm 0cm},clip]{spatial_derivatives_spectral_radius_m_1_nb_der_2.png}} \hspace{-20pt}
		\includegraphics[width=2.5in,trim={0cm 0cm 1.75cm 0cm},clip]{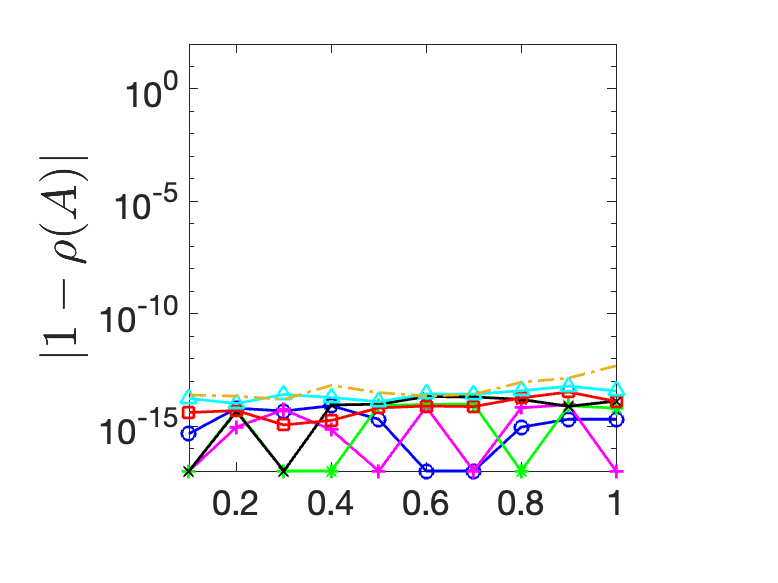}\hspace{-20pt}
		\includegraphics[width=2.5in,trim={0cm 0cm 1.75cm 0cm},clip]{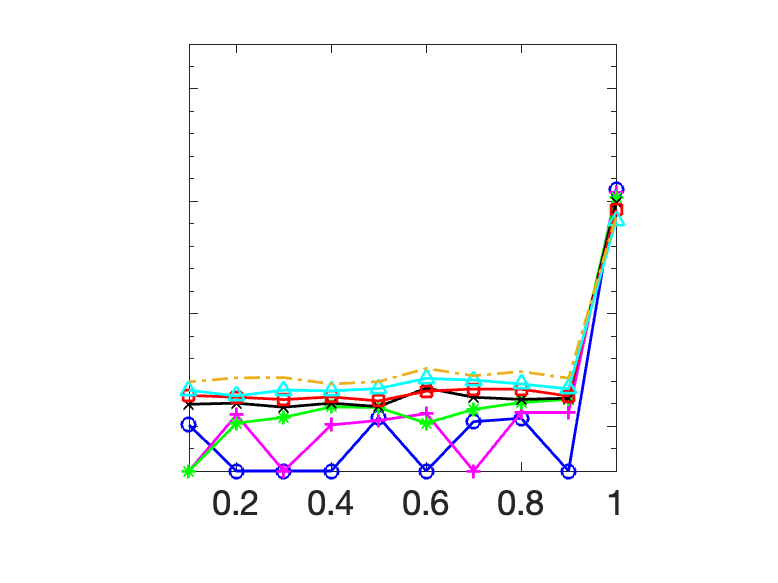}\hspace{-20pt}
		\includegraphics[width=2.5in,trim={0cm 0cm 1.75cm 0cm},clip]{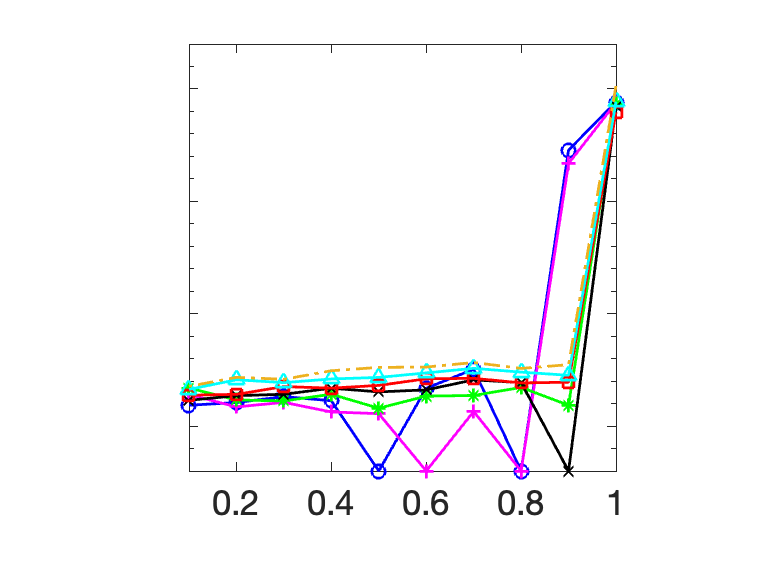}
	\end{adjustbox} 
	\begin{adjustbox}{max width=0.8\textwidth,center}
		\phantom{\includegraphics[width=2.5in,trim={0cm 0cm 1.75cm 0cm},clip]{spatial_derivatives_spectral_radius_m_1_nb_der_2.png}} \hspace{-20pt}
		\includegraphics[width=2.5in,trim={0cm 0cm 1.75cm 0cm},clip]{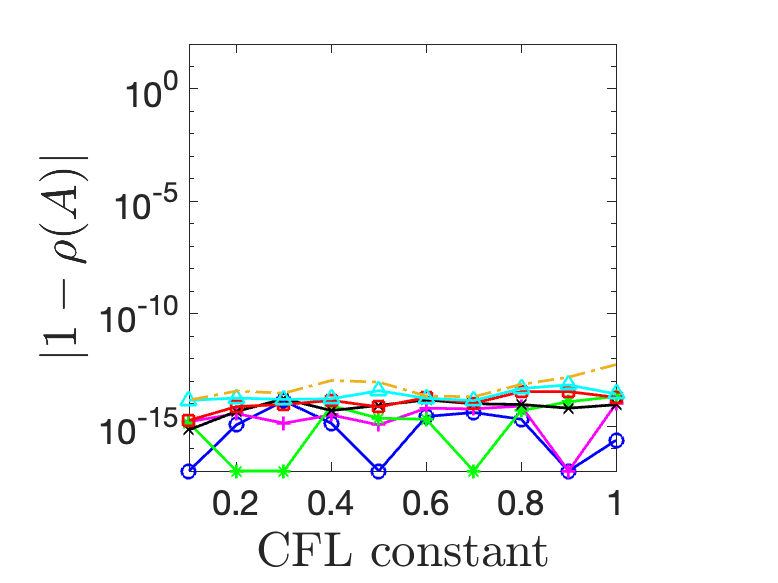}\hspace{-20pt}
		\includegraphics[width=2.5in,trim={0cm 0cm 1.75cm 0cm},clip]{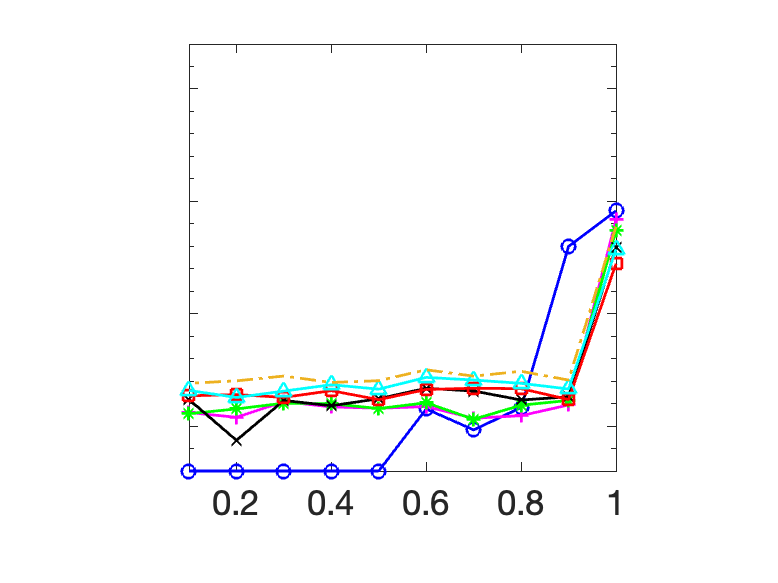}\hspace{-20pt}
		\includegraphics[width=2.5in,trim={0cm 0cm 1.75cm 0cm},clip]{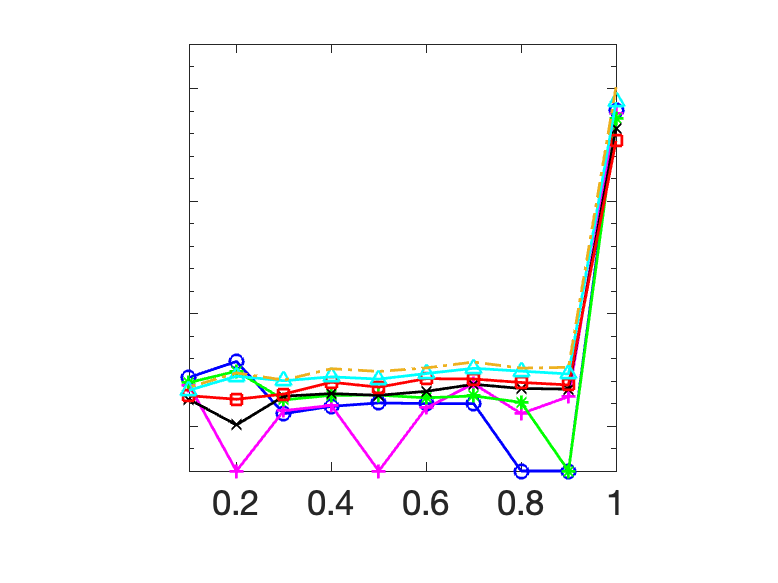}
	\end{adjustbox} 
	\begin{adjustbox}{max width=0.8\textwidth,center}
		\phantom{\includegraphics[width=2.5in,trim={0cm 0cm 1.75cm 0cm},clip]{spatial_derivatives_spectral_radius_m_1_nb_der_2.png}} \hspace{-20pt}
		\phantom{\includegraphics[width=2.5in,trim={0cm 0cm 1.75cm 0cm},clip]{spatial_derivatives_spectral_radius_m_2_nb_der_4.png}}\hspace{-20pt}
		\includegraphics[width=2.5in,trim={0cm 0cm 1.75cm 0cm},clip]{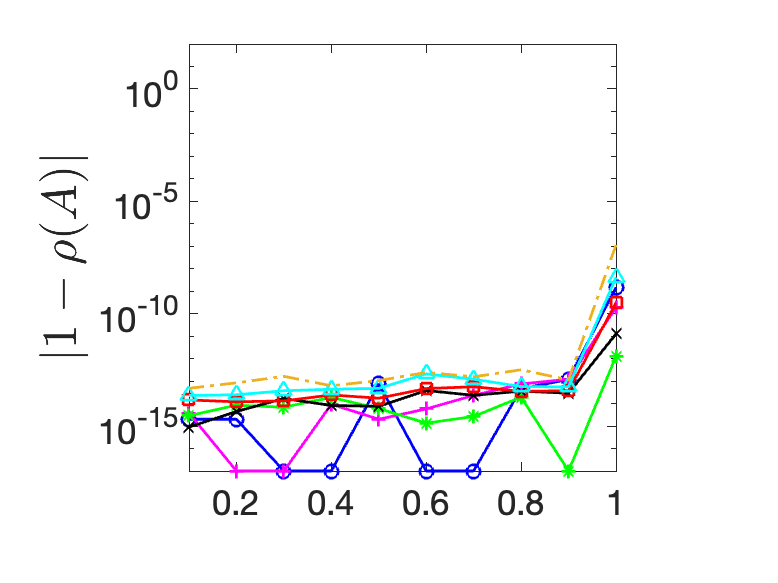}\hspace{-20pt}
		\includegraphics[width=2.5in,trim={0cm 0cm 1.75cm 0cm},clip]{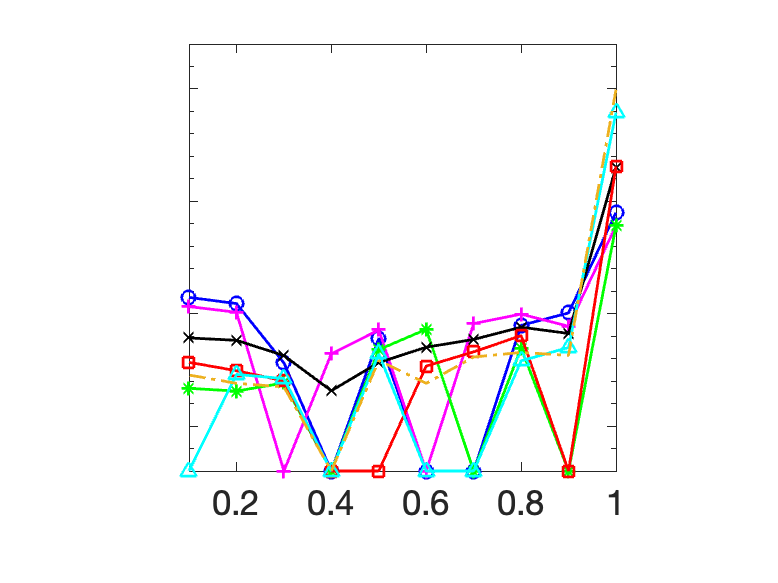}
	\end{adjustbox} 
	\begin{adjustbox}{max width=0.8\textwidth,center}
		\phantom{\includegraphics[width=2.5in,trim={0cm 0cm 1.75cm 0cm},clip]{spatial_derivatives_spectral_radius_m_1_nb_der_2.png}} \hspace{-20pt}
		\phantom{\includegraphics[width=2.5in,trim={0cm 0cm 1.75cm 0cm},clip]{spatial_derivatives_spectral_radius_m_2_nb_der_4.png}}\hspace{-20pt}
		\includegraphics[width=2.5in,trim={0cm 0cm 1.75cm 0cm},clip]{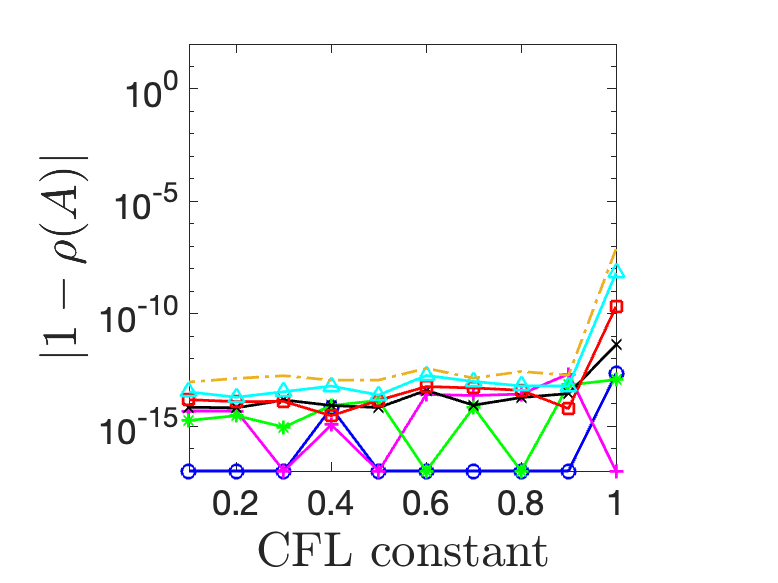}\hspace{-20pt}
		\includegraphics[width=2.5in,trim={0cm 0cm 1.75cm 0cm},clip]{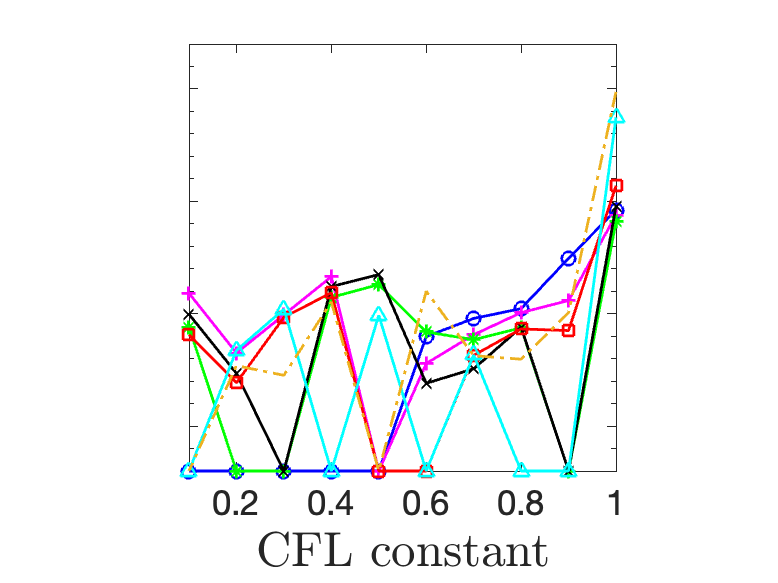}
	\end{adjustbox} 
       \caption{Absolute difference between one and the spectral radius $\rho(A)$ of the matrix $A$ as a function of the CFL constant for different mesh sizes using the constraints on spatial derivatives. 
       The columns are  for different $m$: 1 to 4 from the left to the right. The rows are for the maximum order of the considered spatial derivatives ($N_d$) at the boundary: 0 to 6 from the top to the bottom.}
       \label{fig:spatial_derivative_spectral_radius_1D}
\end{figure}
According to the numerical results, 
    the stability of the method improves as either the CFL constant decreases, 
    the value $N_d$ increases or 
    the mesh size decreases. 
Note that the CFL constant under which the method is stable increases as the mesh size decreases or the value $N_d$ increases.
Moreover,
    for $m = 1 - 2$ and $N_d = m$, 
    we recover the stability condition of the original Hermite-Taylor method ($\mbox{CFL}=1$). 
Although the same behaviour is observed for $m=4$, 
    the difference between $\rho(A)$ and one is larger, 
    particularly for $N_d \geq 5$.
{\color{black} This can be partly explained by the large condition number of the CFM matrices (see subsection~\ref{sec:cond_numb_1D}) that prevents an accurate approximation of the matrix $A$ and therefore of its spectral radius.}

{\color{black} 
The stability of the proposed Hermite-Taylor correction function method can also be further improved by the Hermite smoothing step suggested in \cite{compat_wave_hermite_AAL_DEAA_WDH}. 
Note that we do not need to update the CF nodes at each iteration of the smoothing step.
We repeated the same numerical examples but with one and ten iterations of the Hermite smoothing step at the end of each time step. 
The numerical results suggest that increasing $N_d$ has a much greater impact on the stability. 
}

We now perform the same test but this time we use directly the time derivatives of the electromagnetic fields in the functional $\mathcal{B}$ instead of converting them into spatial derivatives,
    leading to a method that is easier to implement. 
By consistency, 
    this would also implicitly penalize the spatial derivatives and therefore we should expect an improvement of the stability condition.
The absolute difference between $\rho(A)$ and one is illustrated in Fig.~\ref{fig:time_derivative_spectral_radius_1D} for different 
    mesh sizes,
    CFL constants and values of $N_d$,
    that is the maximum order of the time derivatives that are considered in the functional $\mathcal{B}$ \eqref{eq:constraint_derivatives_1D}.
For all values of $m$, 
    there is a clear improvement when $N_d$ goes from 0 to 1. 
However, 
    for $N_d>1$, 
    there is no significant improvement of the stability of the method when compared with $N_d=1$.
For $m=4$ and $N_d \geq 5$, 
    we even observe a deterioration of the stability of the method.
{\color{black} We therefore always convert time derivatives of the variable fields into spatial derivatives since this leads to a Hermite-Taylor correction function method that is stable for a larger CFL constant.}
\begin{figure}   
	\centering
	\begin{adjustbox}{max width=0.8\textwidth,center}
		\includegraphics[width=2.5in,trim={0cm 0cm 1.75cm 0cm},clip]{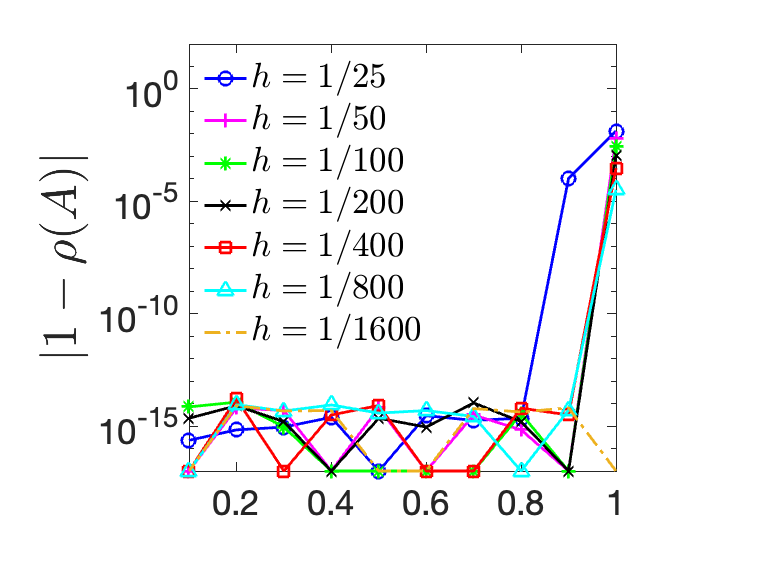} \hspace{-20.0pt}
		\includegraphics[width=2.5in,trim={0cm 0cm 1.75cm 0cm},clip]{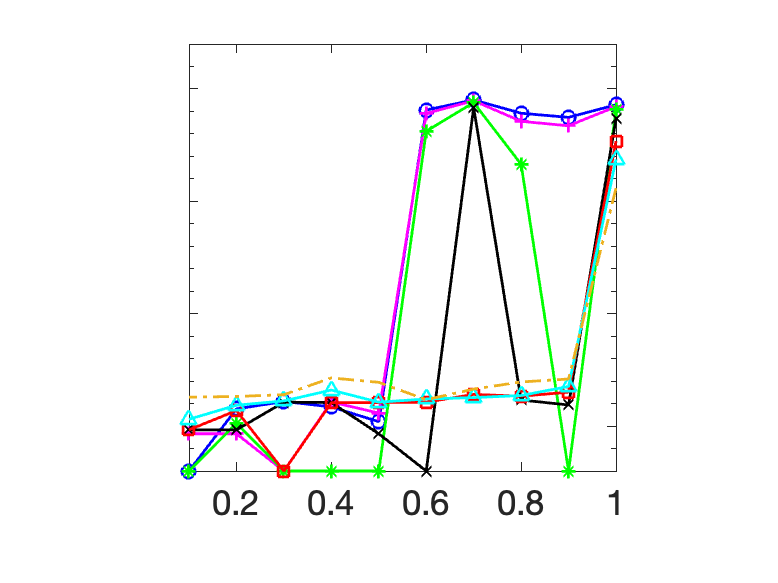}\hspace{-20pt}
		\includegraphics[width=2.5in,trim={0cm 0cm 1.75cm 0cm},clip]{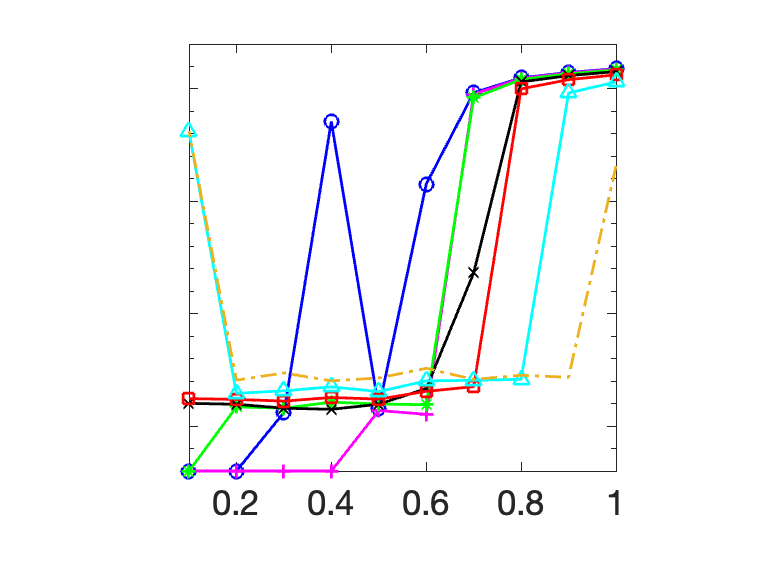}\hspace{-20pt}
		\includegraphics[width=2.5in,trim={0cm 0cm 1.75cm 0cm},clip]{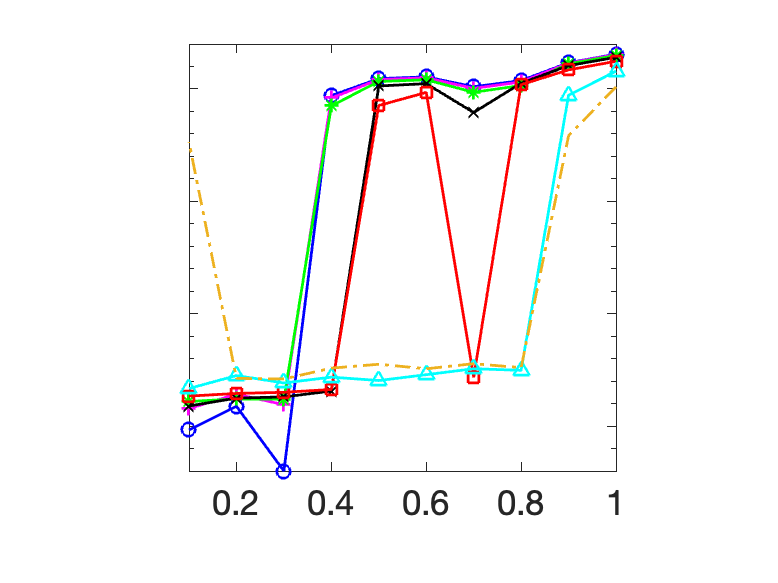}
	\end{adjustbox} 
	\begin{adjustbox}{max width=0.8\textwidth,center}
		\includegraphics[width=2.5in,trim={0cm 0cm 1.75cm 0cm},clip]{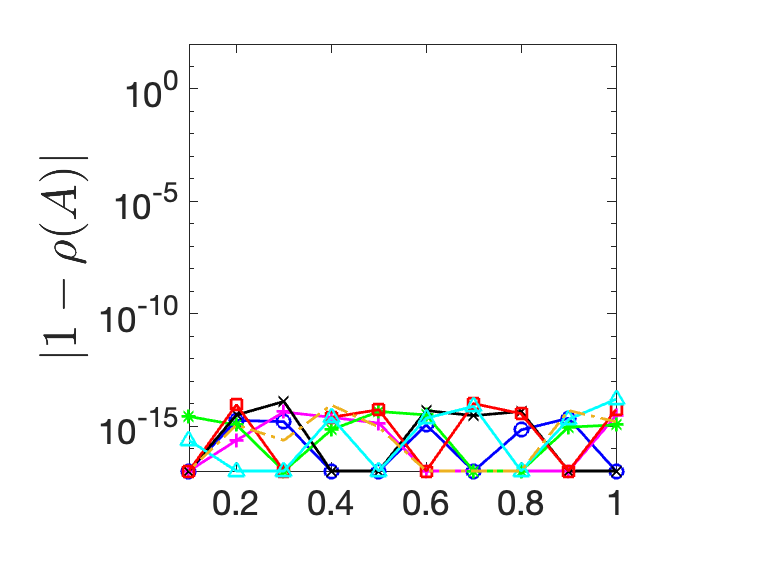} \hspace{-20pt}
		\includegraphics[width=2.5in,trim={0cm 0cm 1.75cm 0cm},clip]{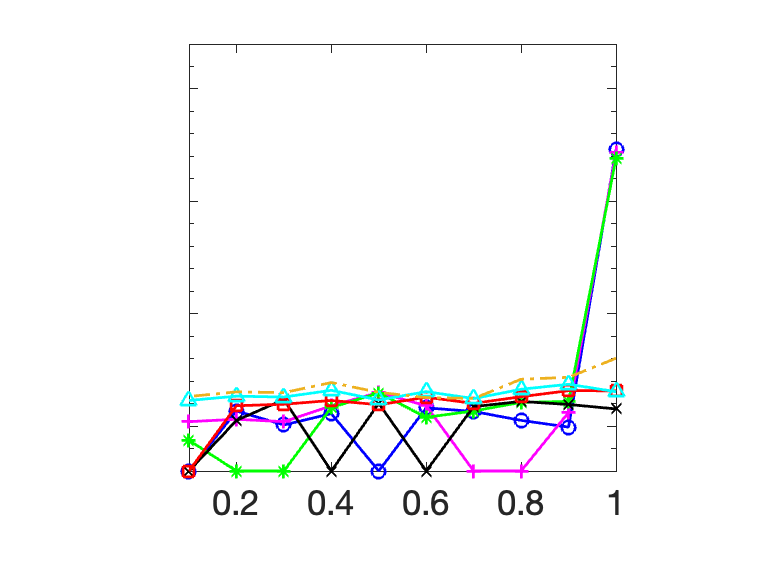}\hspace{-20pt}
		\includegraphics[width=2.5in,trim={0cm 0cm 1.75cm 0cm},clip]{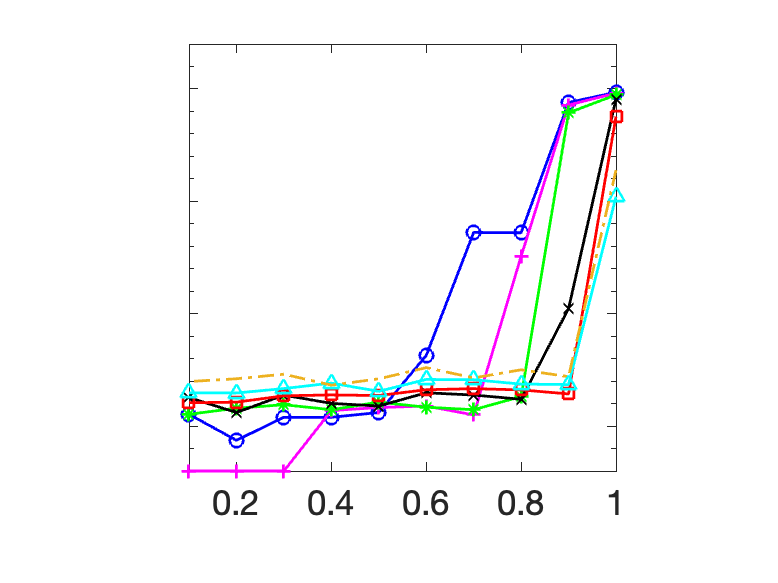}\hspace{-20pt}
		\includegraphics[width=2.5in,trim={0cm 0cm 1.75cm 0cm},clip]{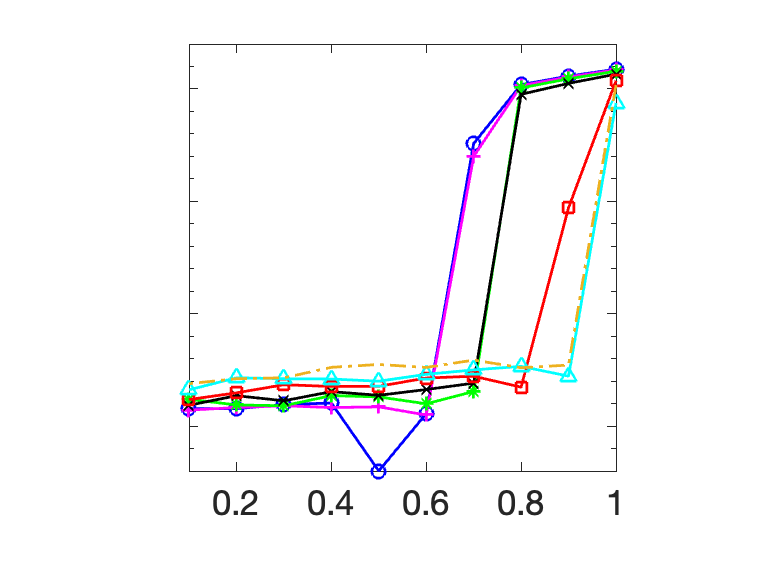}
	\end{adjustbox}  
	\begin{adjustbox}{max width=0.8\textwidth,center}
		\includegraphics[width=2.5in,trim={0cm 0cm 1.75cm 0cm},clip]{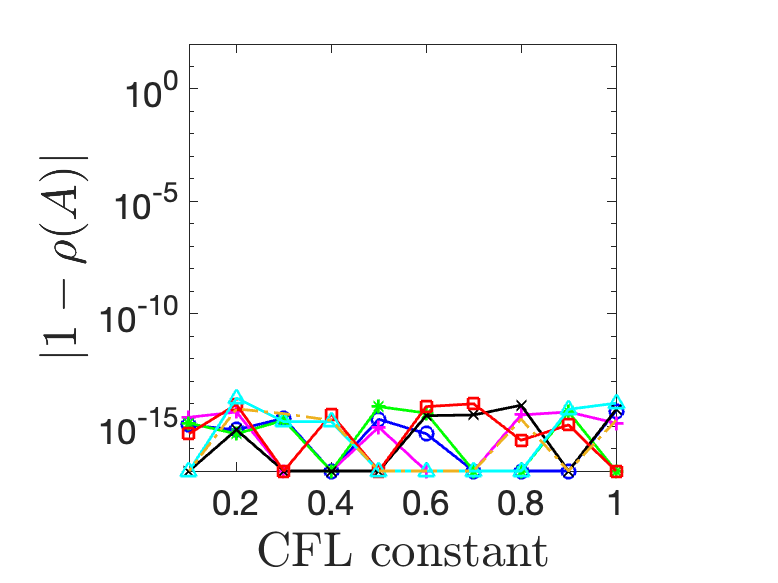} \hspace{-20pt}
		\includegraphics[width=2.5in,trim={0cm 0cm 1.75cm 0cm},clip]{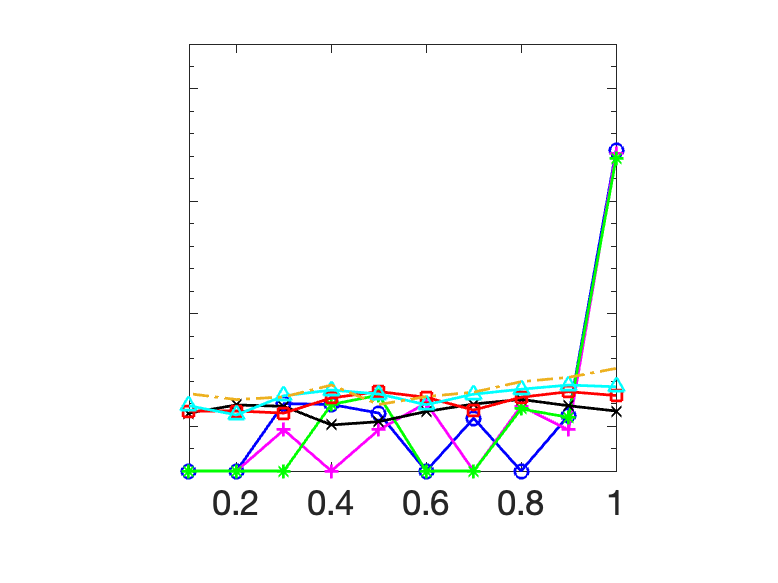}\hspace{-20pt}
		\includegraphics[width=2.5in,trim={0cm 0cm 1.75cm 0cm},clip]{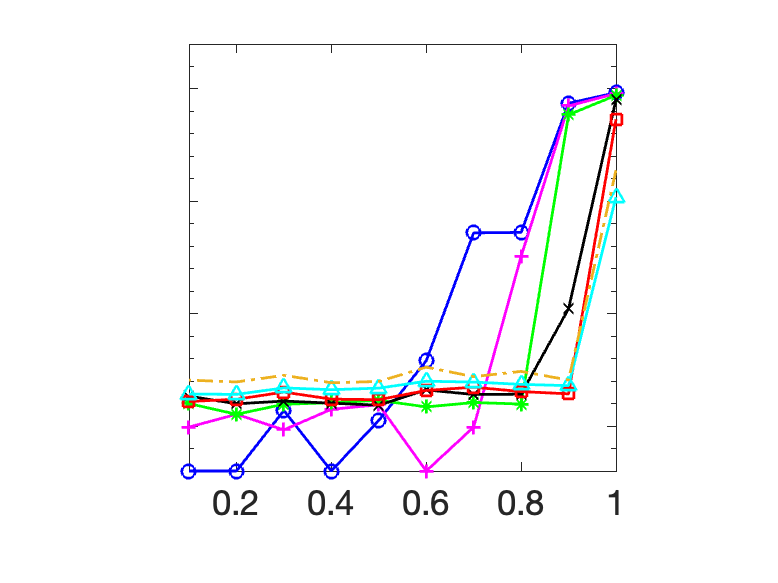}\hspace{-20pt}
		\includegraphics[width=2.5in,trim={0cm 0cm 1.75cm 0cm},clip]{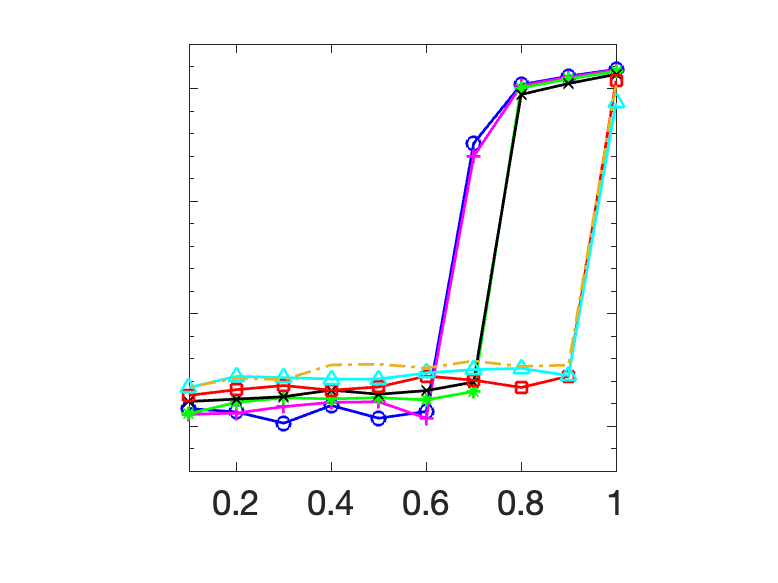}
	\end{adjustbox} 
	\begin{adjustbox}{max width=0.8\textwidth,center}
		\phantom{\includegraphics[width=2.5in,trim={0cm 0cm 1.75cm 0cm},clip]{time_derivatives_spectral_radius_m_1_nb_der_2.png}} \hspace{-20pt}
		\includegraphics[width=2.5in,trim={0cm 0cm 1.75cm 0cm},clip]{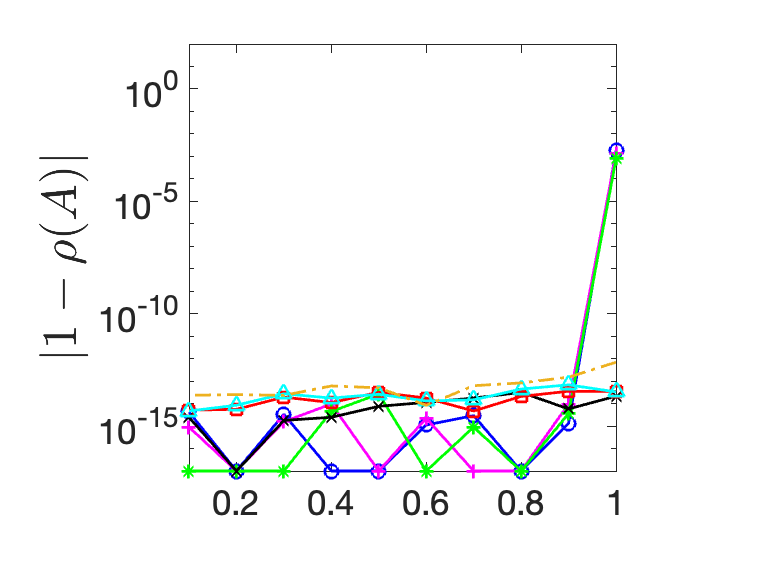}\hspace{-20pt}
		\includegraphics[width=2.5in,trim={0cm 0cm 1.75cm 0cm},clip]{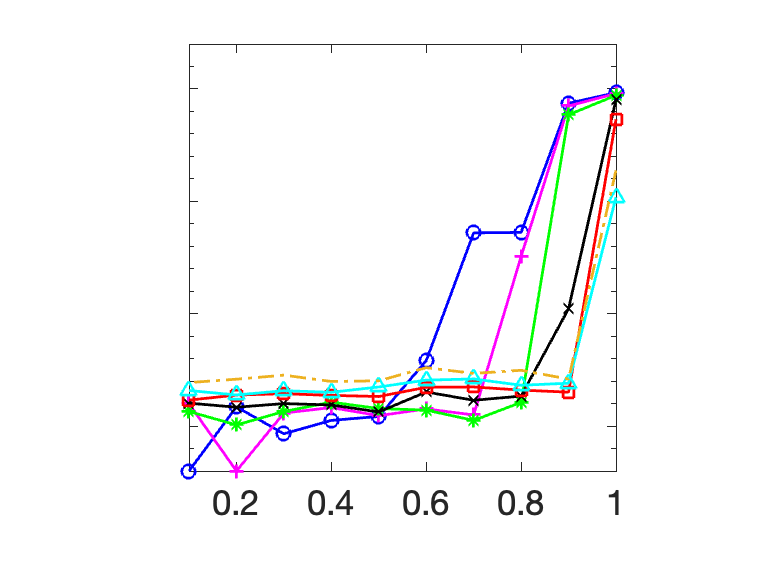}\hspace{-20pt}
		\includegraphics[width=2.5in,trim={0cm 0cm 1.75cm 0cm},clip]{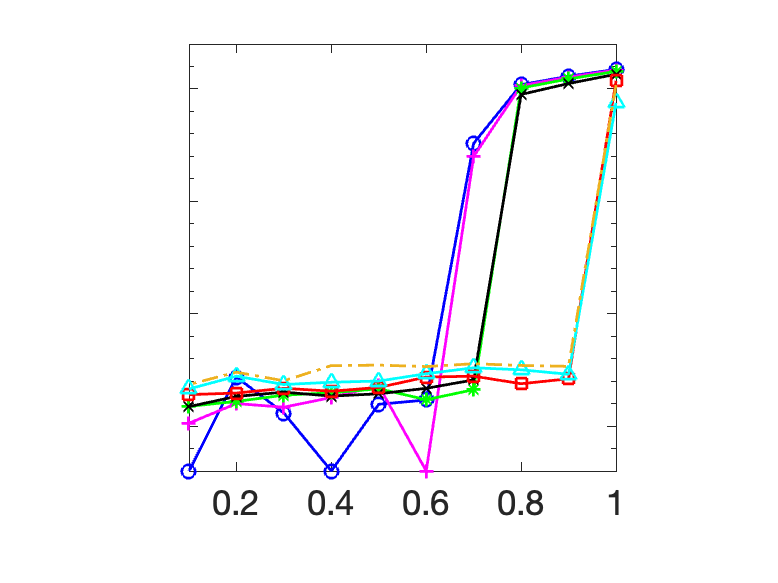}
	\end{adjustbox} 
	\begin{adjustbox}{max width=0.8\textwidth,center}
		\phantom{\includegraphics[width=2.5in,trim={0cm 0cm 1.75cm 0cm},clip]{time_derivatives_spectral_radius_m_1_nb_der_2.png}} \hspace{-20pt}
		\includegraphics[width=2.5in,trim={0cm 0cm 1.75cm 0cm},clip]{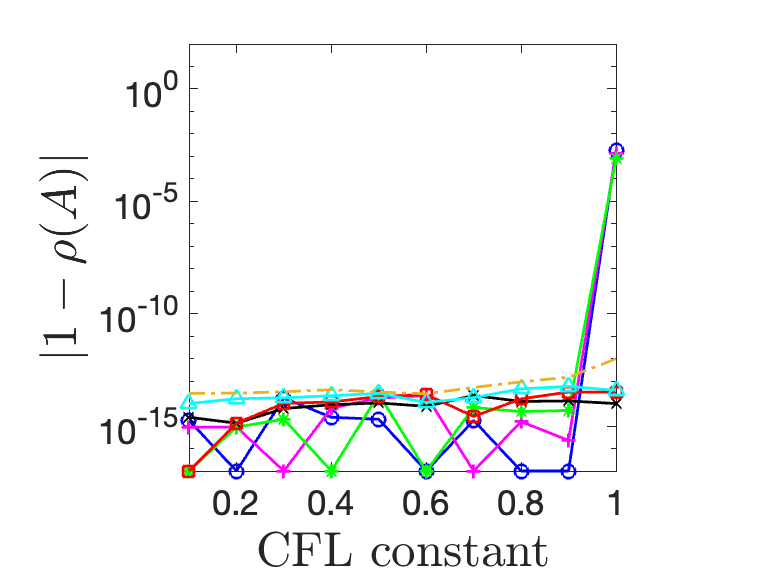}\hspace{-20pt}
		\includegraphics[width=2.5in,trim={0cm 0cm 1.75cm 0cm},clip]{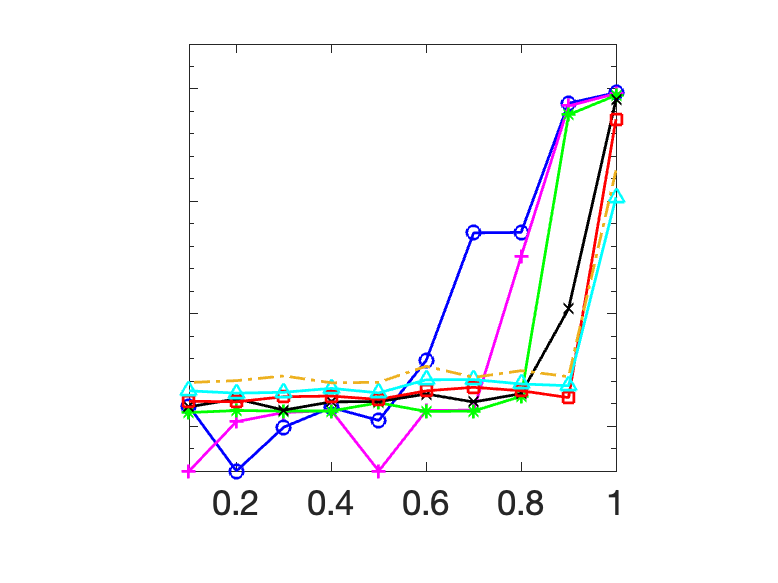}\hspace{-20pt}
		\includegraphics[width=2.5in,trim={0cm 0cm 1.75cm 0cm},clip]{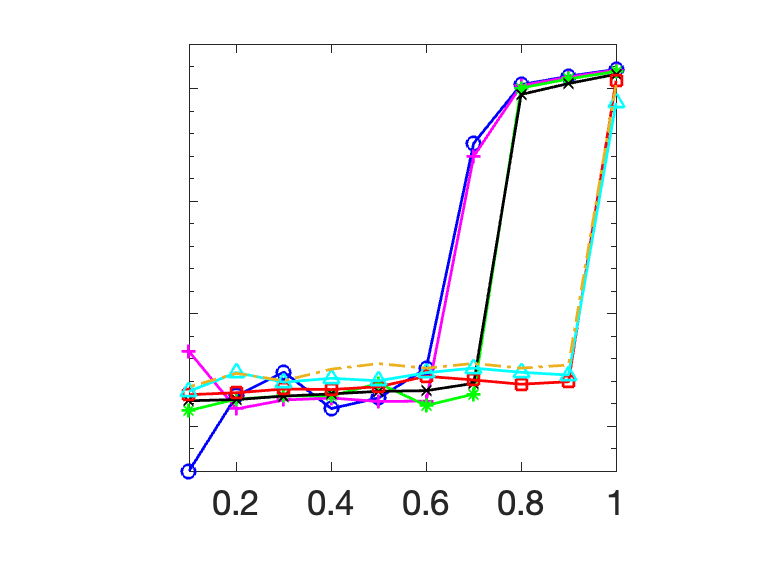}
	\end{adjustbox} 
	\begin{adjustbox}{max width=0.8\textwidth,center}
		\phantom{\includegraphics[width=2.5in,trim={0cm 0cm 1.75cm 0cm},clip]{time_derivatives_spectral_radius_m_1_nb_der_2.png}} \hspace{-20pt}
		\phantom{\includegraphics[width=2.5in,trim={0cm 0cm 1.75cm 0cm},clip]{time_derivatives_spectral_radius_m_2_nb_der_4.png}}\hspace{-20pt}
		\includegraphics[width=2.5in,trim={0cm 0cm 1.75cm 0cm},clip]{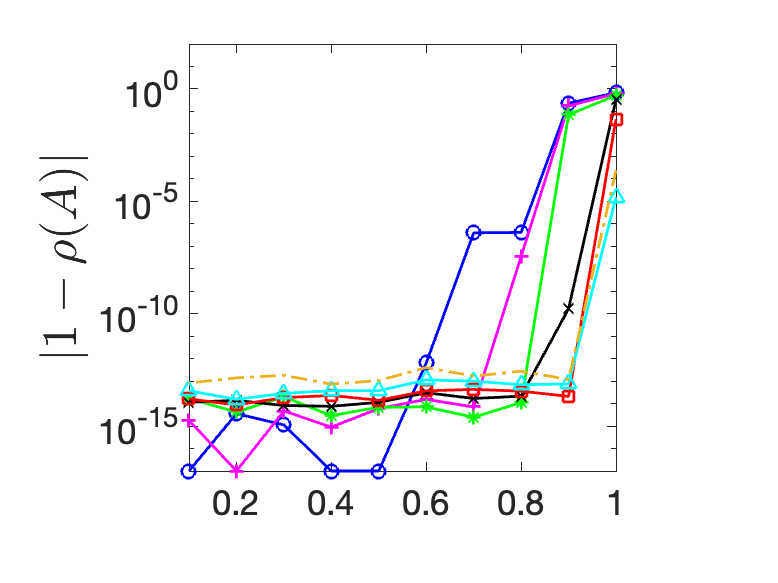}\hspace{-20pt}
		\includegraphics[width=2.5in,trim={0cm 0cm 1.75cm 0cm},clip]{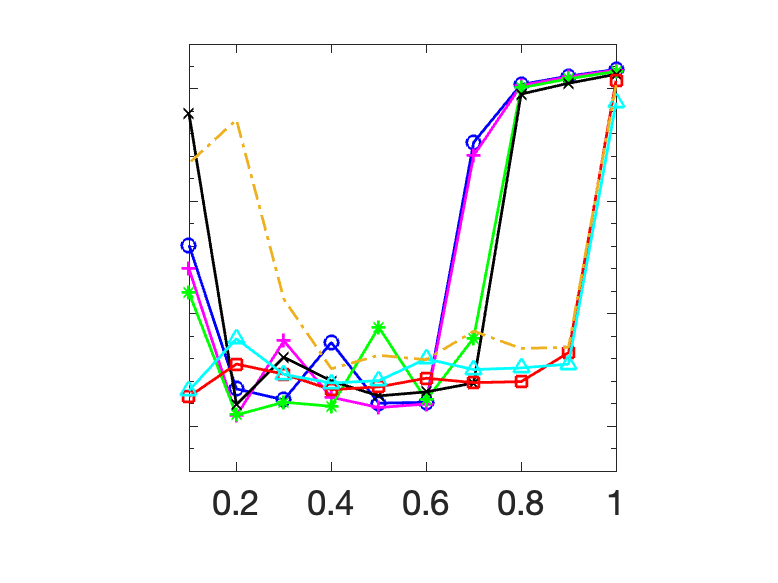}
	\end{adjustbox} 
	\begin{adjustbox}{max width=0.8\textwidth,center}
		\phantom{\includegraphics[width=2.5in,trim={0cm 0cm 1.75cm 0cm},clip]{time_derivatives_spectral_radius_m_1_nb_der_2.png}} \hspace{-20pt}
		\phantom{\includegraphics[width=2.5in,trim={0cm 0cm 1.75cm 0cm},clip]{time_derivatives_spectral_radius_m_2_nb_der_4.png}}\hspace{-20pt}
		\includegraphics[width=2.5in,trim={0cm 0cm 1.75cm 0cm},clip]{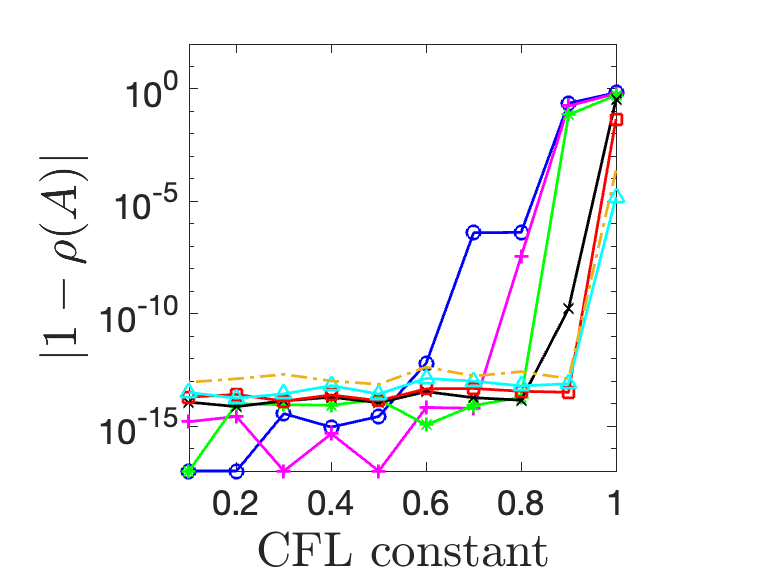}\hspace{-20pt}
		\includegraphics[width=2.5in,trim={0cm 0cm 1.75cm 0cm},clip]{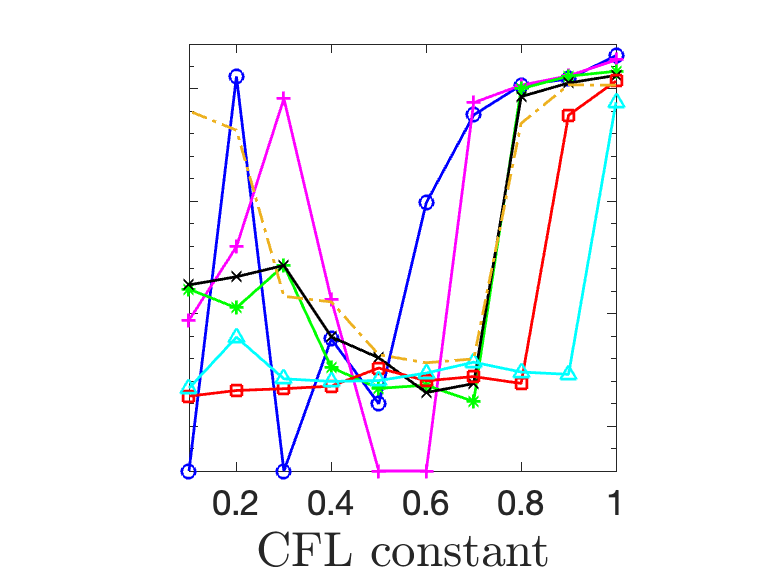}
	\end{adjustbox} 
       \caption{Absolute difference between one and the spectral radius $\rho(A)$ of the matrix $A$ as a function of the CFL constant for different mesh sizes using the constraints on the time derivatives. 
       The columns are  for different $m$: 1 to 4 from the left to the right. The rows are for the maximum order of the considered time derivatives ($N_d$) at the boundary: 0 to 6 from the top to the bottom.}
       \label{fig:time_derivative_spectral_radius_1D}
\end{figure}

Let us now consider a domain where there are primal and dual CF nodes. 
The physical domain $\Omega = [\frac{\pi}{50}, 1- \frac{\pi}{100}]$ while the computational domain is $\Omega_c = [0,1]$.
In this situation, to express the method as a one-step evolution we would have to include both the primal and dual node solutions in the definition of $\mathbold{W}^n$ in \eqref{eq:one_step_version}. Instead we here
investigate 
	the stability using long time simulations. 
{\color{black}
We consider the trivial solution for all electromagnetic fields but with initial data,
	namely the electromagnetic fields and their first $m$ derivatives, 
	to be random numbers in $(-10\,\epsilon_m,10\,\epsilon_m)$.}
Here $\epsilon_m$ is the machine precision.

Fig.~\ref{fig:investigation_long_time_1D} illustrates the maximum norm of the numerical solution over $10^6$ time steps as a function of the CFL constant for different 
	mesh sizes and values of $N_d$.
Here again, 
    the stability of the Hermite-Taylor correction function methods improves as either the CFL constant decreases, the value $N_d$ 
    increases or 
    the mesh size decreases. 
For $m=1,2,3$ with respectively $N_d = 0, 1, 4$,
    we recover the stability condition of the original Hermite-Taylor method ($\mbox{CFL}=1$).

\begin{figure}   
	\centering
	\begin{adjustbox}{max width=0.8\textwidth,center}
		\includegraphics[width=2.5in,trim={0cm 0cm 1.75cm 0cm},clip]{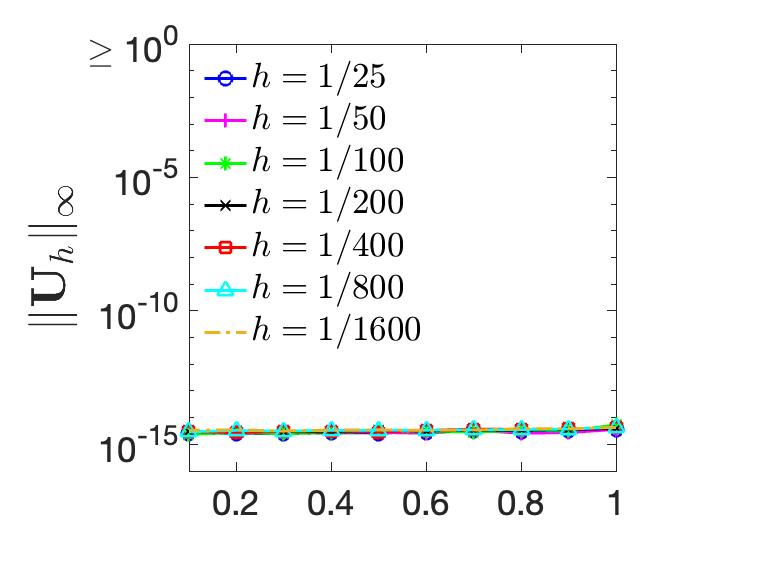} \hspace{-20.0pt}
		\includegraphics[width=2.5in,trim={0cm 0cm 1.75cm 0cm},clip]{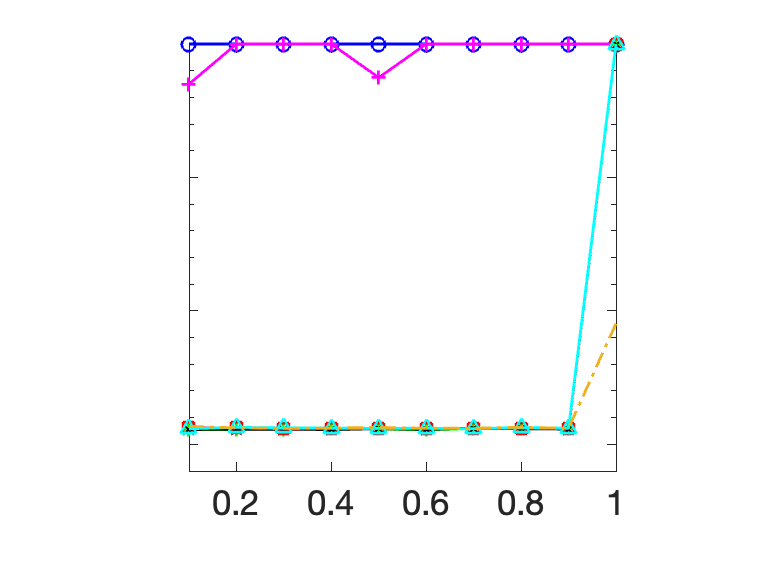}\hspace{-20pt}
		\includegraphics[width=2.5in,trim={0cm 0cm 1.75cm 0cm},clip]{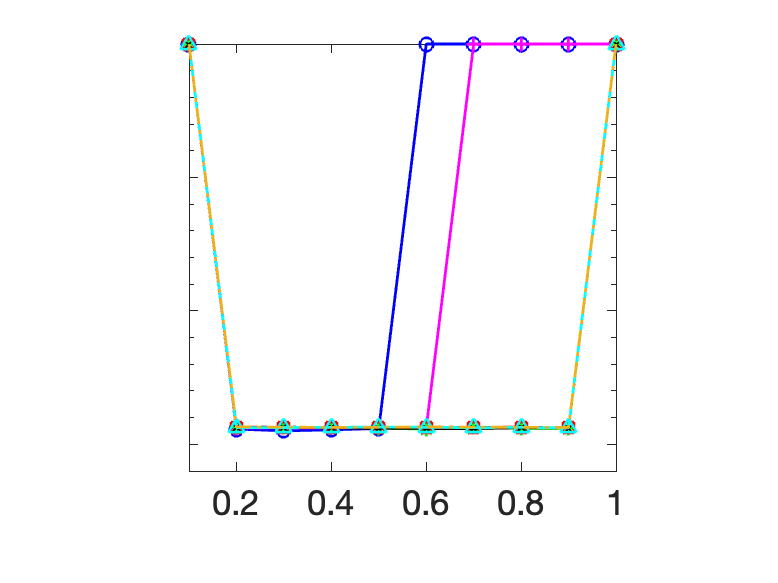}\hspace{-20pt}
		\includegraphics[width=2.5in,trim={0cm 0cm 1.75cm 0cm},clip]{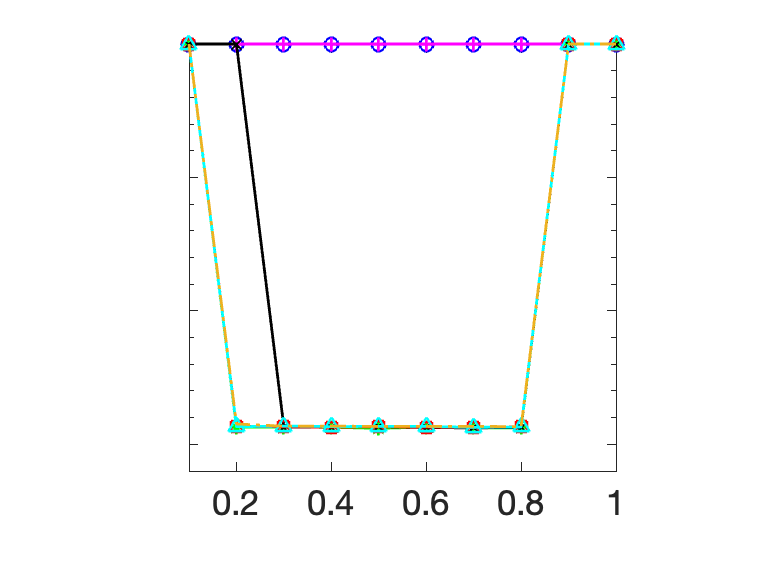}
	\end{adjustbox} 
	\begin{adjustbox}{max width=0.8\textwidth,center}
		\includegraphics[width=2.5in,trim={0cm 0cm 1.75cm 0cm},clip]{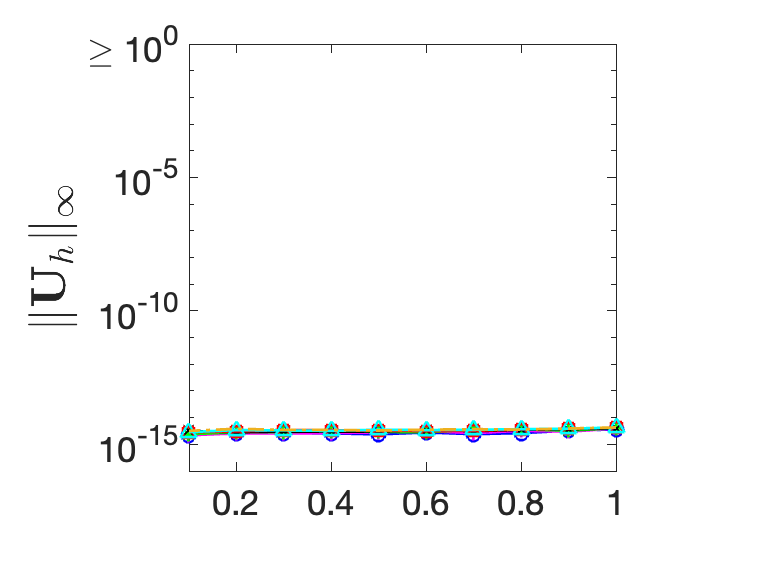} \hspace{-20pt}
		\includegraphics[width=2.5in,trim={0cm 0cm 1.75cm 0cm},clip]{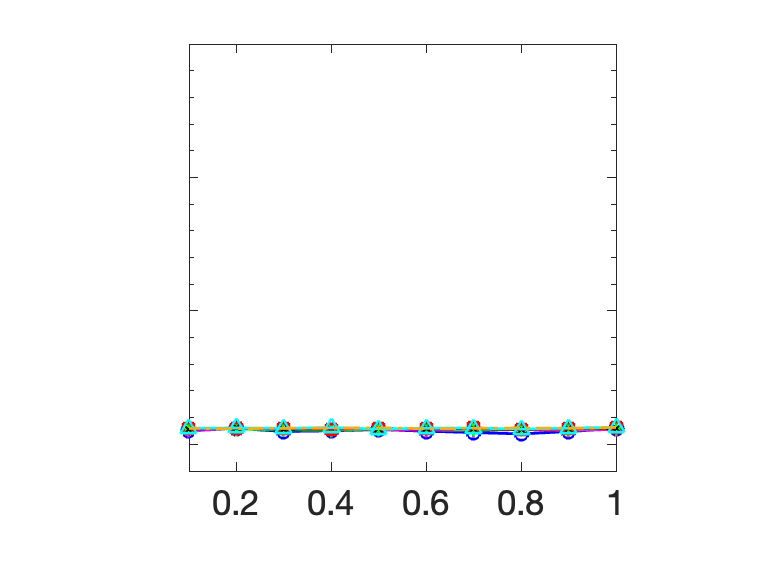}\hspace{-20pt}
		\includegraphics[width=2.5in,trim={0cm 0cm 1.75cm 0cm},clip]{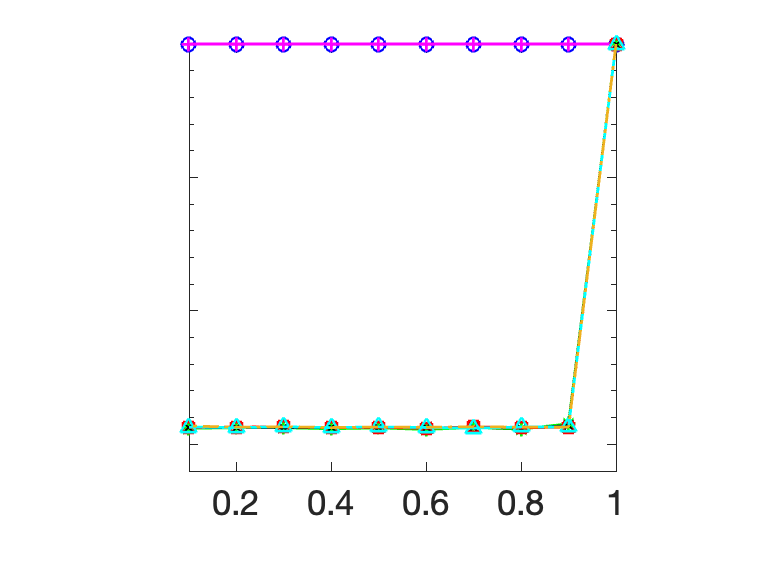}\hspace{-20pt}
		\includegraphics[width=2.5in,trim={0cm 0cm 1.75cm 0cm},clip]{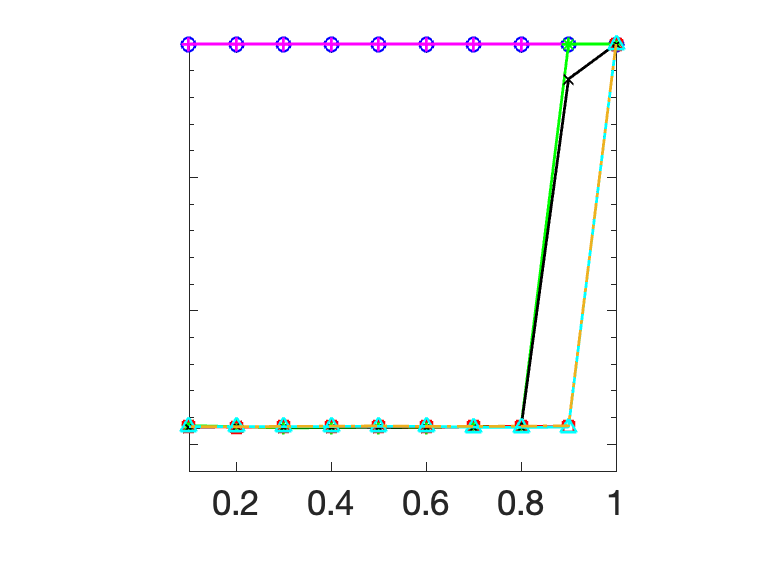}
	\end{adjustbox}  
	\begin{adjustbox}{max width=0.8\textwidth,center}
		\includegraphics[width=2.5in,trim={0cm 0cm 1.75cm 0cm},clip]{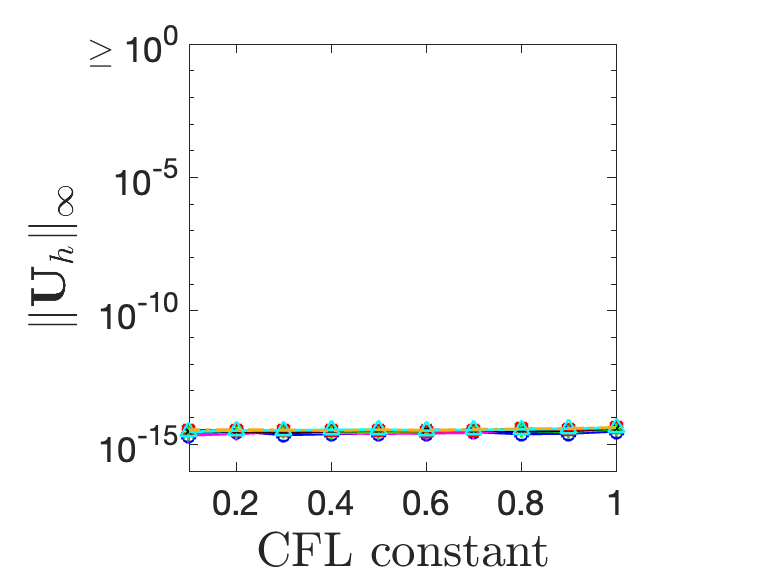} \hspace{-20pt}
		\includegraphics[width=2.5in,trim={0cm 0cm 1.75cm 0cm},clip]{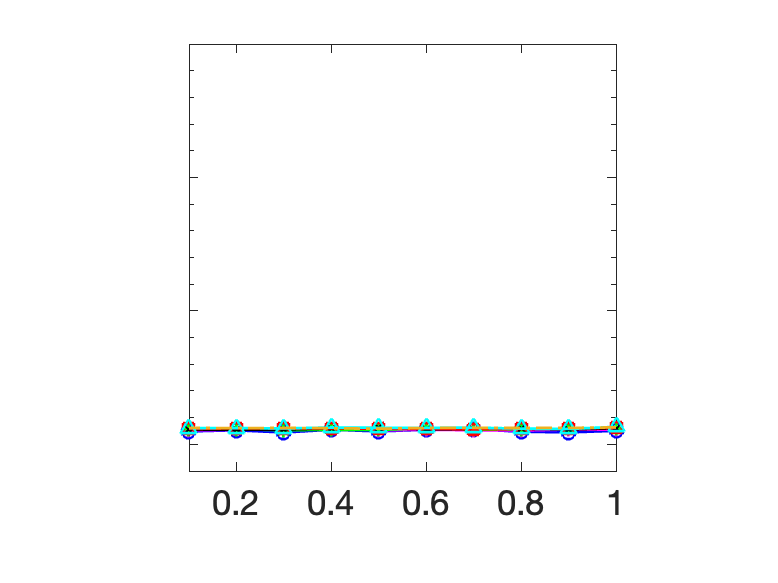}\hspace{-20pt}
		\includegraphics[width=2.5in,trim={0cm 0cm 1.75cm 0cm},clip]{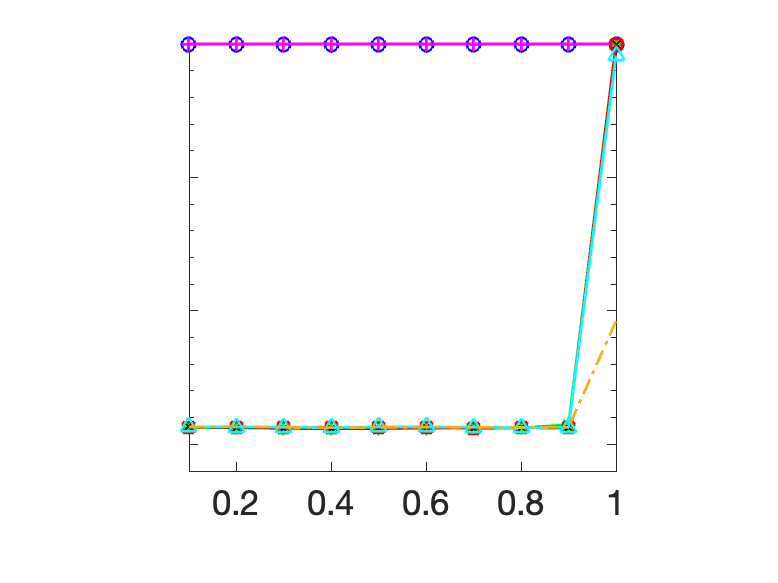}\hspace{-20pt}
		\includegraphics[width=2.5in,trim={0cm 0cm 1.75cm 0cm},clip]{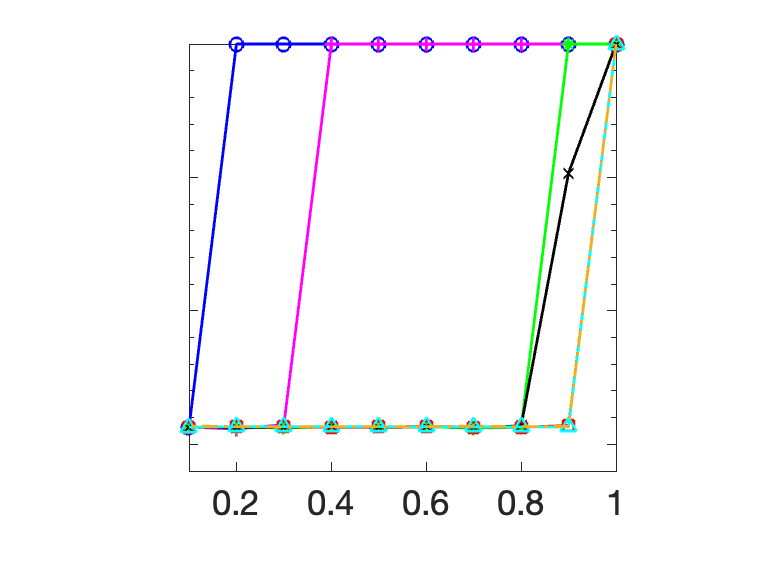}
	\end{adjustbox} 
	\begin{adjustbox}{max width=0.8\textwidth,center}
		\phantom{\includegraphics[width=2.5in,trim={0cm 0cm 1.75cm 0cm},clip]{max_norm_m_1_nb_der_2.png}} \hspace{-20pt}
		\includegraphics[width=2.5in,trim={0cm 0cm 1.75cm 0cm},clip]{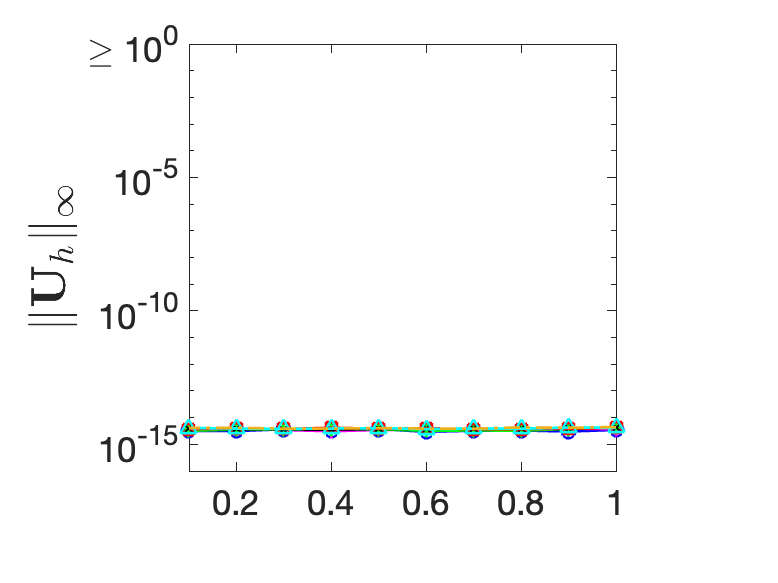}\hspace{-20pt}
		\includegraphics[width=2.5in,trim={0cm 0cm 1.75cm 0cm},clip]{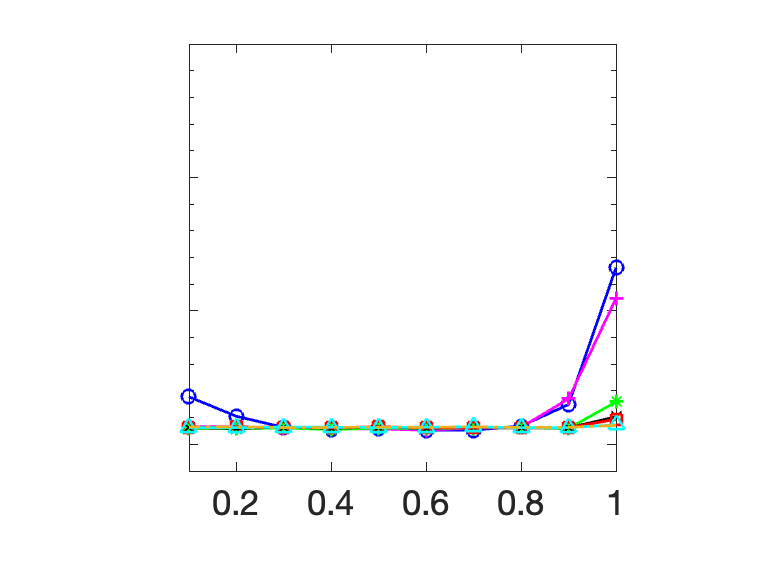}\hspace{-20pt}
		\includegraphics[width=2.5in,trim={0cm 0cm 1.75cm 0cm},clip]{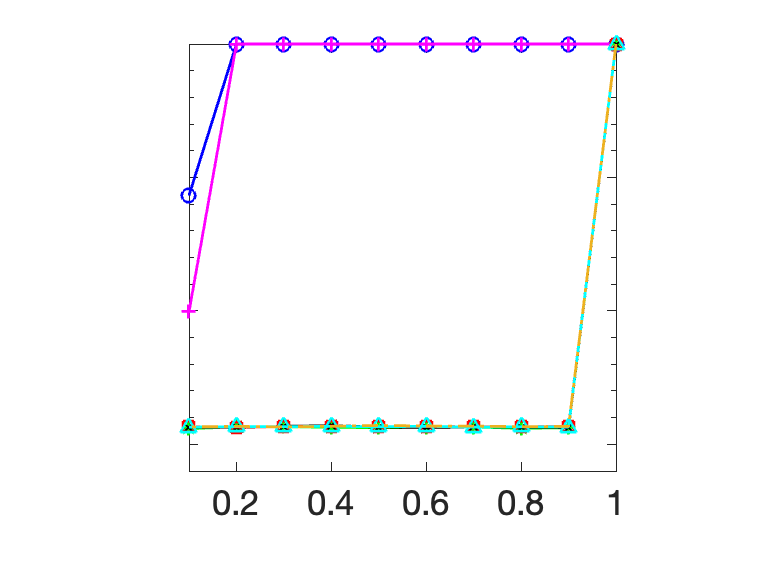}
	\end{adjustbox} 
	\begin{adjustbox}{max width=0.8\textwidth,center}
		\phantom{\includegraphics[width=2.5in,trim={0cm 0cm 1.75cm 0cm},clip]{max_norm_m_1_nb_der_2.png}} \hspace{-20pt}
		\includegraphics[width=2.5in,trim={0cm 0cm 1.75cm 0cm},clip]{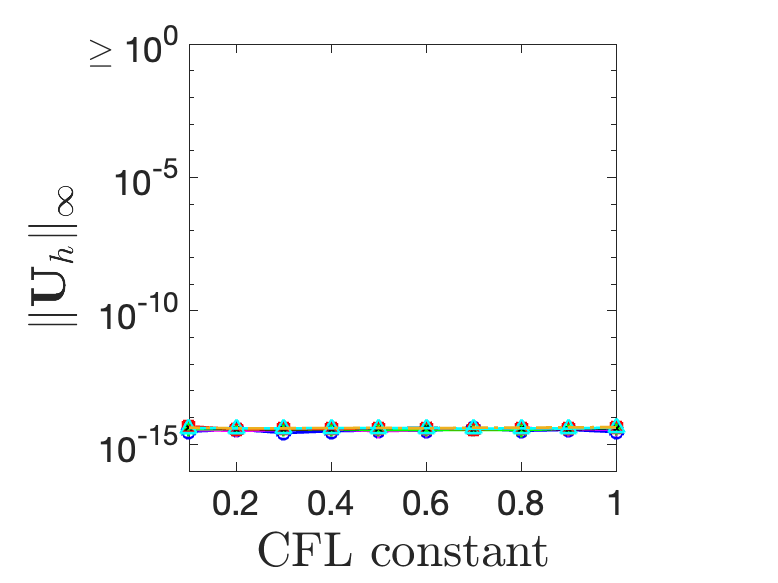}\hspace{-20pt}
		\includegraphics[width=2.5in,trim={0cm 0cm 1.75cm 0cm},clip]{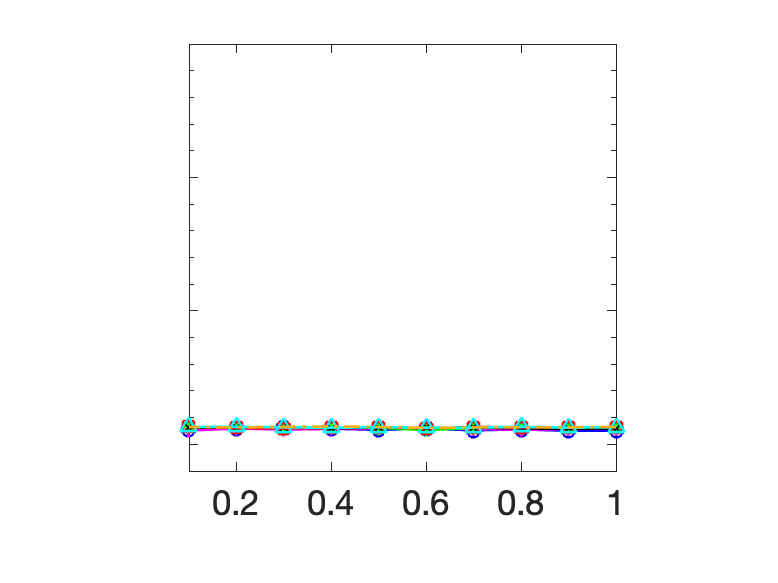}\hspace{-20pt}
		\includegraphics[width=2.5in,trim={0cm 0cm 1.75cm 0cm},clip]{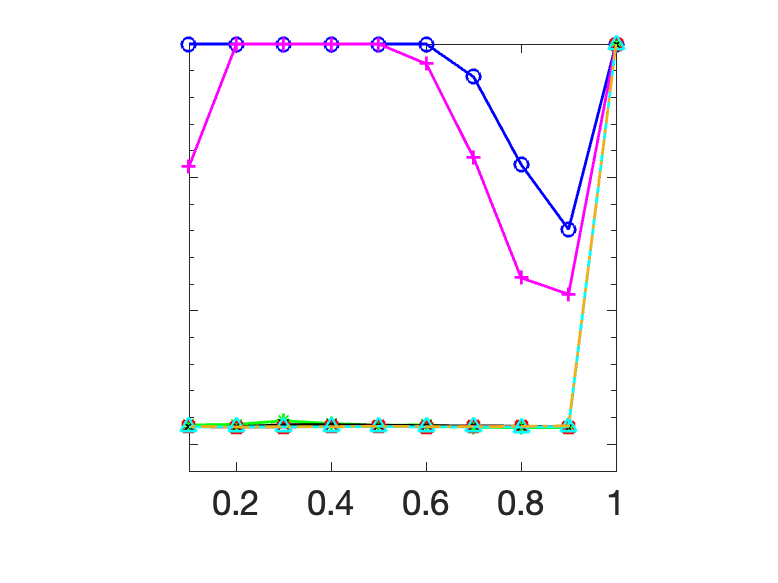}
	\end{adjustbox} 
	\begin{adjustbox}{max width=0.8\textwidth,center}
		\phantom{\includegraphics[width=2.5in,trim={0cm 0cm 1.75cm 0cm},clip]{max_norm_m_1_nb_der_2.png}} \hspace{-20pt}
		\phantom{\includegraphics[width=2.5in,trim={0cm 0cm 1.75cm 0cm},clip]{max_norm_m_2_nb_der_4.png}}\hspace{-20pt}
		\includegraphics[width=2.5in,trim={0cm 0cm 1.75cm 0cm},clip]{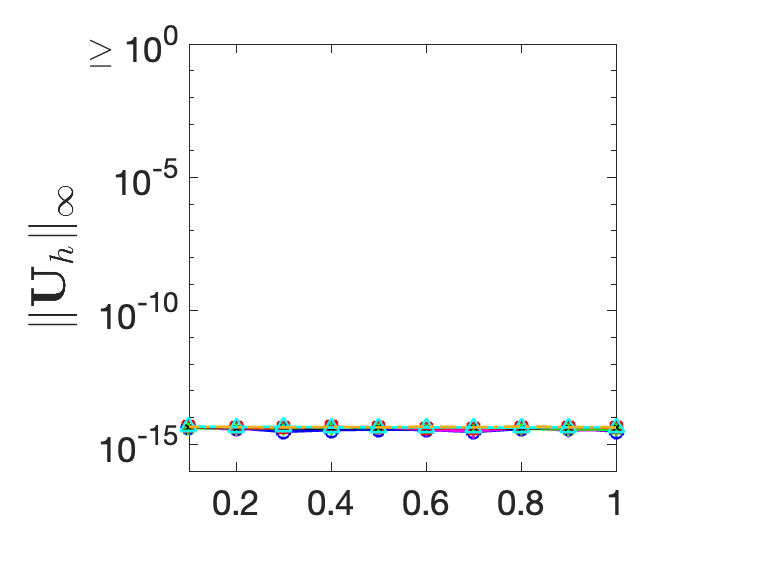}\hspace{-20pt}
		\includegraphics[width=2.5in,trim={0cm 0cm 1.75cm 0cm},clip]{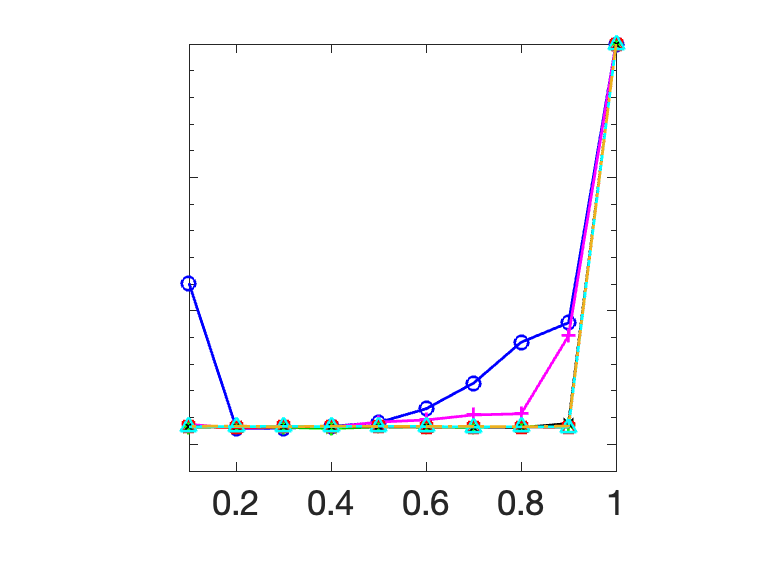}
	\end{adjustbox} 
	\begin{adjustbox}{max width=0.8\textwidth,center}
		\phantom{\includegraphics[width=2.5in,trim={0cm 0cm 1.75cm 0cm},clip]{max_norm_m_1_nb_der_2.png}} \hspace{-20pt}
		\phantom{\includegraphics[width=2.5in,trim={0cm 0cm 1.75cm 0cm},clip]{max_norm_m_2_nb_der_4.png}}\hspace{-20pt}
		\includegraphics[width=2.5in,trim={0cm 0cm 1.75cm 0cm},clip]{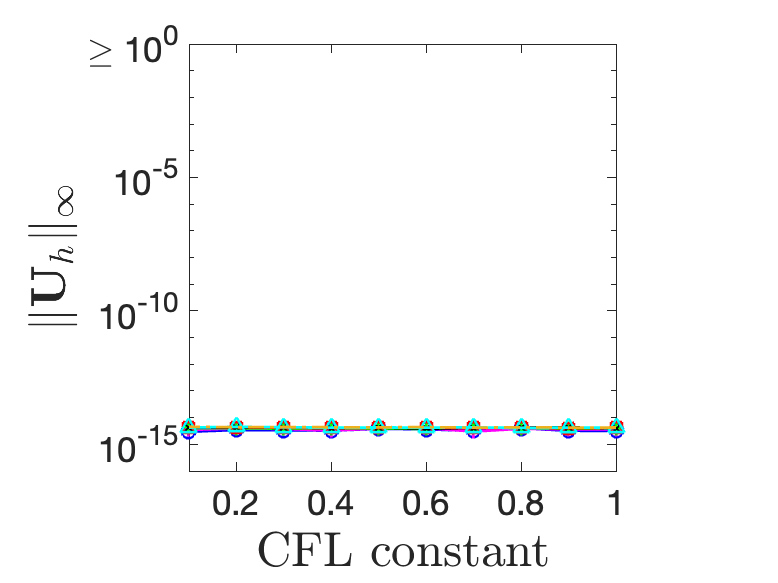}\hspace{-20pt}
		\includegraphics[width=2.5in,trim={0cm 0cm 1.75cm 0cm},clip]{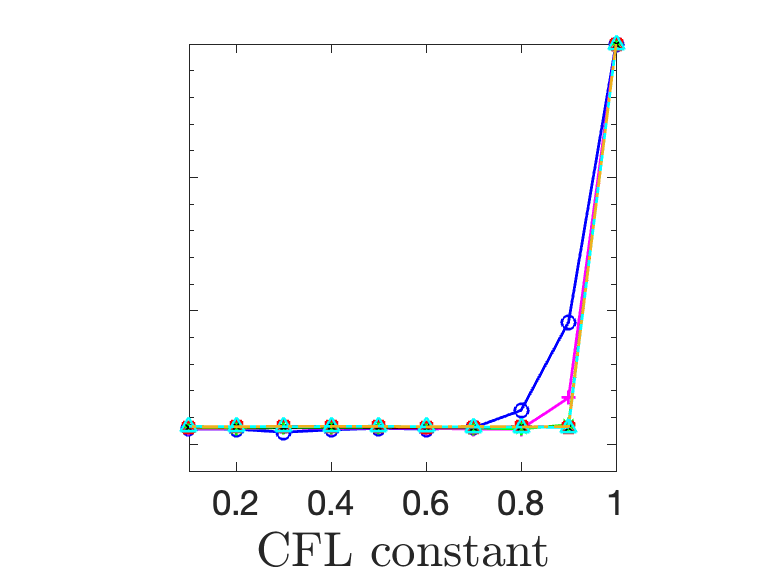}
	\end{adjustbox} 
       \caption{Maximum value of the maximum norm of the numerical solution over $10^6$ time steps as a function of the CFL constant for different mesh sizes.
       The columns are  for different $m$: 1 to 4 from the left to the right. The rows are for the maximum order of the considered spatial derivatives ($N_d$) at the boundary: 0 to 6 from the top to the bottom.}
       \label{fig:investigation_long_time_1D}
\end{figure}

{\color{black}
\subsubsection{Condition Number of Correction Function Matrices} \label{sec:cond_numb_1D}

In this subsection, 
    we investigate the condition number of the correction function matrices coming from the minimization procedure.
We consider the numerical example where only primal CF nodes are needed with the same settings as previously described. 
Fig.~\ref{fig:spatial_derivatives_condition_number_1D} 
    illustrates the maximum condition number of the CF matrices as a function of the CFL for different mesh sizes and values of $N_d$.
We observe that the condition number of the matrices $M$ increases as $N_d$ increases.
Note that the condition number of the CF matrices remains roughly constant with respect to the CFL constant for $m\leq2$. 
For $m=3$ and $m=4$, 
    it decreases when the CFL constant diminishes.
\begin{figure}   
	\centering
	\begin{adjustbox}{max width=1.0\textwidth,center}
		\includegraphics[width=2.5in,trim={0cm 0cm 1.75cm 0cm},clip]{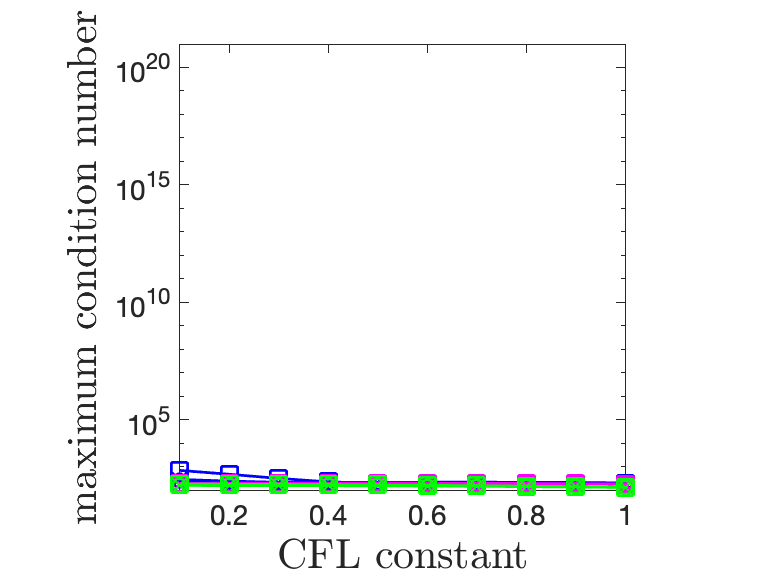} \hspace{-20.0pt}
		\includegraphics[width=2.5in,trim={0cm 0cm 1.75cm 0cm},clip]{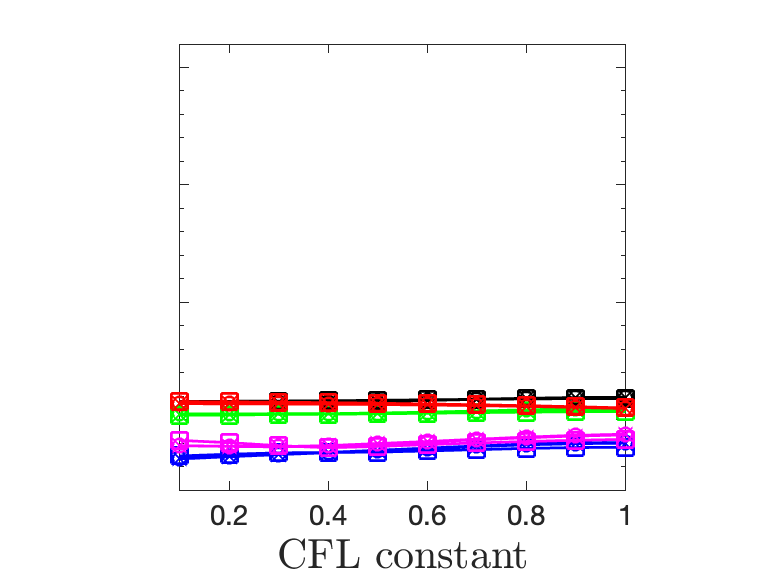}\hspace{-20pt}
		\includegraphics[width=2.5in,trim={0cm 0cm 1.75cm 0cm},clip]{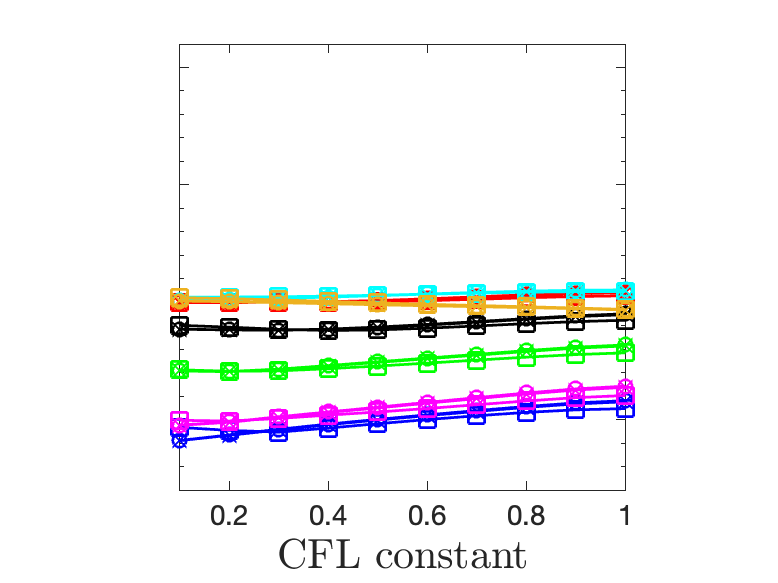}\hspace{-20pt}
		\includegraphics[width=2.5in,trim={0cm 0cm 1.75cm 0cm},clip]{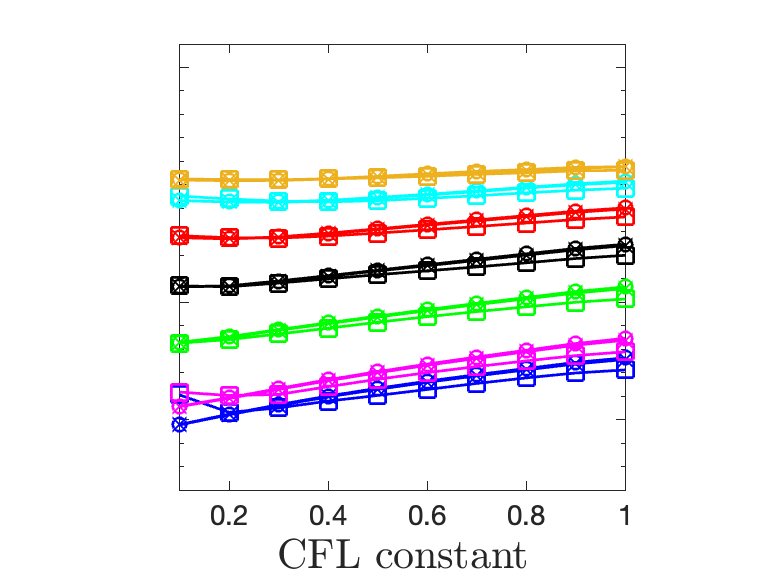}
	\end{adjustbox} 
       \caption{Maximum condition number of CF matrices as a function of the CFL constant for different mesh sizes and number of spatial derivatives at the boundary. 
       The value $m$ is 0 to 4 from the left to the right.
       The circle, cross and square markers stand for $\Delta x = \frac{1}{25}$, $\Delta x = \frac{1}{200}$ and $\Delta x = \frac{1}{1600}$. 
       The colors blue, magenta, green, black, red, cyan and orange are respectively for $N_d = 0 - 6$.}
\label{fig:spatial_derivatives_condition_number_1D}
\end{figure}

Fig.~\ref{fig:time_derivatives_condition_number_1D} illustrates the condition number of the CF matrices as a function of the CFL constant for different meshes and $N_d$ when the time derivatives are directly considered in the functional $\mathcal{B}$.
The condition number increases when $N_d$ increases and, 
    for $N_d > 2$,
    when the CFL constant diminishes. 
Note that,
    when we convert the time derivatives of the electromagnetic fields into spatial derivatives,
    it leads to better conditioned CF matrices.
This further motivates us to convert time derivative of the electromagnetic fields used in the boundary (or interface) conditions into spatial derivatives. 
\begin{figure}
	\centering
	\begin{adjustbox}{max width=1.0\textwidth,center}
		\includegraphics[width=2.5in,trim={0cm 0cm 1.75cm 0cm},clip]{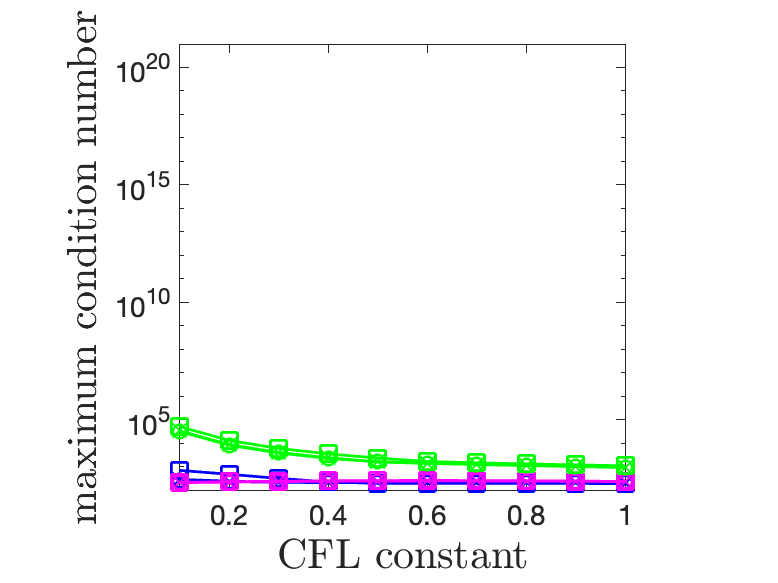} \hspace{-20.0pt}
		\includegraphics[width=2.5in,trim={0cm 0cm 1.75cm 0cm},clip]{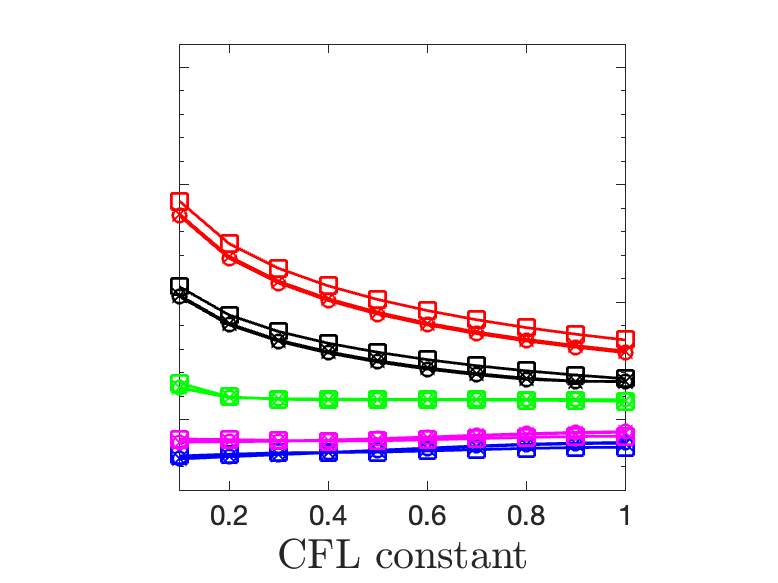}\hspace{-20pt}
		\includegraphics[width=2.5in,trim={0cm 0cm 1.75cm 0cm},clip]{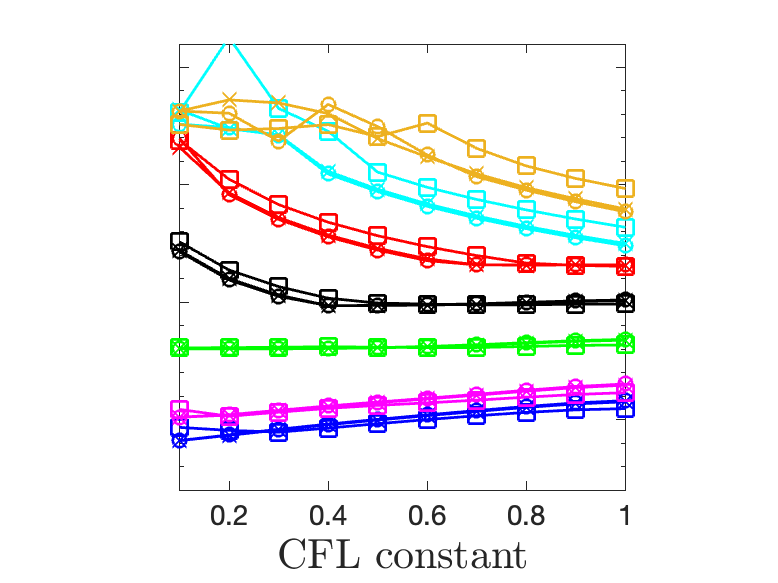}\hspace{-20pt}
		\includegraphics[width=2.5in,trim={0cm 0cm 1.75cm 0cm},clip]{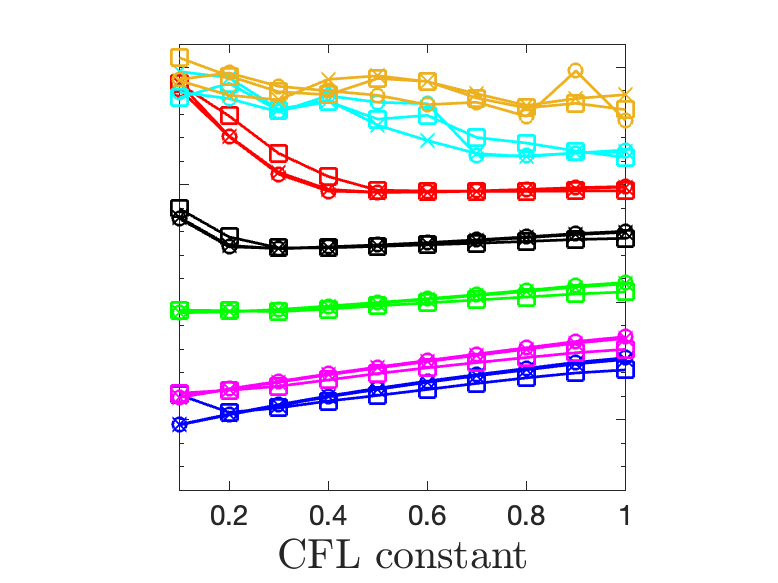}
	\end{adjustbox} 
       \caption{Maximum condition number of CF matrices as a function of the CFL constant for different mesh sizes and number of time derivatives at the boundary. 
       The value $m$ is 0 to 4 from the left to the right.
       The circle, cross and square markers stand for $\Delta x = \frac{1}{25}$, $\Delta x = \frac{1}{200}$ and $\Delta x = \frac{1}{1600}$. 
       The colors blue, magenta, green, black, red, cyan and orange are respectively for $N_d = 0 - 6$.}  \label{fig:time_derivatives_condition_number_1D}
\end{figure}
}
\subsubsection{Accuracy}

Let us now investigate the accuracy of the Hermite-Taylor correction function method with $m = 1 - 4$. 
The physical parameters are $\mu=1$ and $\epsilon=1$. 
We set $\Delta t = 0.9\,h$ for $m\leq 3$ and $N_d=0,2,5$ for respectively $m =1,2,3$. 
For $m=4$, 
    we use $\Delta t = 0.8\,h$ and $N_d =3$.
The physical domain is $\Omega = [\frac{\pi}{50}, 1- \frac{\pi}{100}]$ while the computational domain is $\Omega_c = [0,1]$.  
The time interval is $I=[0,20]$.
The initial and boundary conditions are chosen so that the solution to Maxwell's equations is 
\begin{equation}
	\begin{aligned}
		H(x,t) =&\,\, \sin(250\,x)\,\sin(250\,t), \\
		E(x,t) =&\,\, \cos(250\,x)\,\cos(250\,t).
	\end{aligned}
\end{equation}
The relative error in the $L^2$-norm is computed at the final time.
Fig.~\ref{fig:conv_embedded_boundary_1D} illustrates the convergence plots for $m = 1 - 4$.
We observe a rough $2\,m+1$ convergence order for all values of $m$. 
\begin{figure}   
	\centering
	\includegraphics[width=2.5in,trim={0.0cm 0cm 1.75cm 0cm},clip]{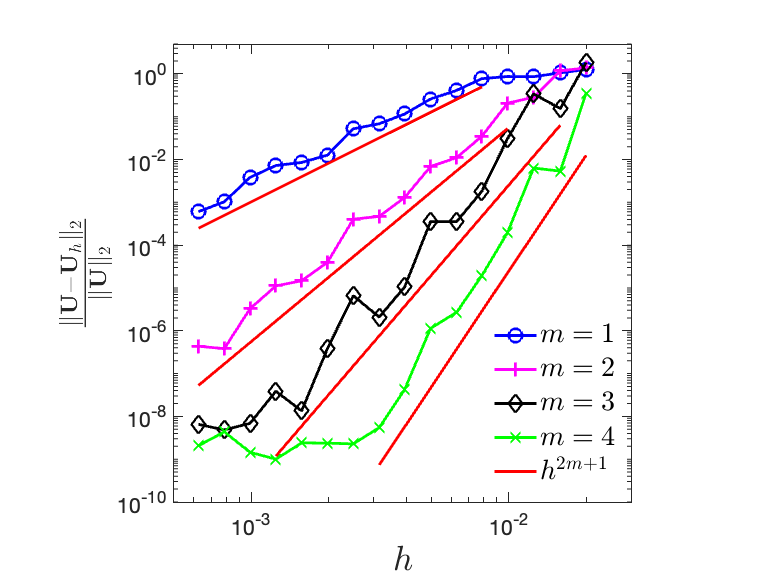} 
       \caption{Convergence plots for embedded boundary problems using the Hermite-Taylor correction function method with $1\leq m \leq 4$ at final time $t_f=20$. 
       Here $\mathbf{U} = [H, E]^T$. }
       \label{fig:conv_embedded_boundary_1D}
\end{figure}

	
\subsection{Hermite-Taylor Correction Function Methods in Two Dimensions}

We consider the two-dimensional simplification of Maxwell's equations \eqref{eq:TMz_syst}. 
The correction function polynomials are chosen to be elements of $\mathbb{Q}^{2\,m}$ to preserve the accuracy of the Hermite-Taylor method and we choose $c_H = 1$.
We set $\Delta L_\Gamma = \alpha h$ with $\alpha = 1.5$ for the local patches.

{\color{black}
Considering an interface problem with two different media, 
    the $L^2$-norm of the divergence of the magnetic field at the final time is computed using 
\begin{equation} \label{eq:computation_divergence_H}
    \|\nabla\cdot(\mu\mathbold{H}_h)\|_2 = \|\nabla\cdot(\mu\mathbold{H}^*)\|_2 + \sum_{q=0}^{N_{cf}-1} \Big(\|\nabla\cdot(\mu\mathbold{H}^+_{h,q})\|_2 + 
    \|\nabla\cdot(\mu\mathbold{H}^-_{h,q})\|_2\Big).
\end{equation}
Here $N_{cf}$ is the number of patches required for the CFM, $\mathbold{H}^*$ is the approximation of the magnetic field coming from the Hermite-Taylor method, 
    and 
    $\mathbold{H}_{h,q}^\pm$, 
    $q = 0,\dots,N_{cf}-1$, 
    are the CF magnetic field approximations associated with the subdomain $\Omega^\pm$.
The first term in \eqref{eq:computation_divergence_H} computes the $L^2$-norm of the divergence of magnetic field approximation coming from the Hermite-Taylor numerical solution.
Since the Hermite-Taylor cells do not cover a narrow band around the interface,
    as shown in Figure~\ref{fig:local_patch_interface_2D},
    we also consider the contribution of the divergence of the magnetic field approximations coming from the correction function method. 
The second term therefore sums the $L^2$-norm of the divergence of the CF magnetic field on each local patch. 
}

\subsubsection{Stability} \label{sec:stability_2d}

In this subsection, 
    we investigate the stability of 
    the Hermite-Taylor correction function method for embedded boundary and interface problems. 
To do so, 
    we use long time simulations.
We first consider an embedded boundary problem. 
The computational domain is $\Omega_c = [0,1]\times[0,1]$ and the embedded boundary $\Gamma$ is a circle with a radius of $0.3$ and centered 
	at $(0.5,0.5)$ that encloses the physical domain $\Omega$.	
{\color{black}
As in the one-dimensional experiments described above, 
	we consider the trivial solution for all electromagnetic fields but with initial data,
	namely the electromagnetic fields and the necessary derivatives, 
	to be random numbers in $(-10\,\epsilon_m,10\,\epsilon_m)$.
    }
The physical parameters are set to $\mu=1$ and $\epsilon=1$. 
	
Fig.~\ref{fig:investigation_long_time_2D_PEC} illustrates the evolution of the maximum norm of the 
	numerical solution over $10^5$ time steps for $m=1-2$ using 
    different mesh sizes,
    values of $N_d$ and CFL constants.
\begin{figure}   
	\centering
	\begin{adjustbox}{max width=1.0\textwidth,center}
		\includegraphics[width=2.5in,trim={0cm 0cm 1.75cm 0cm},clip]{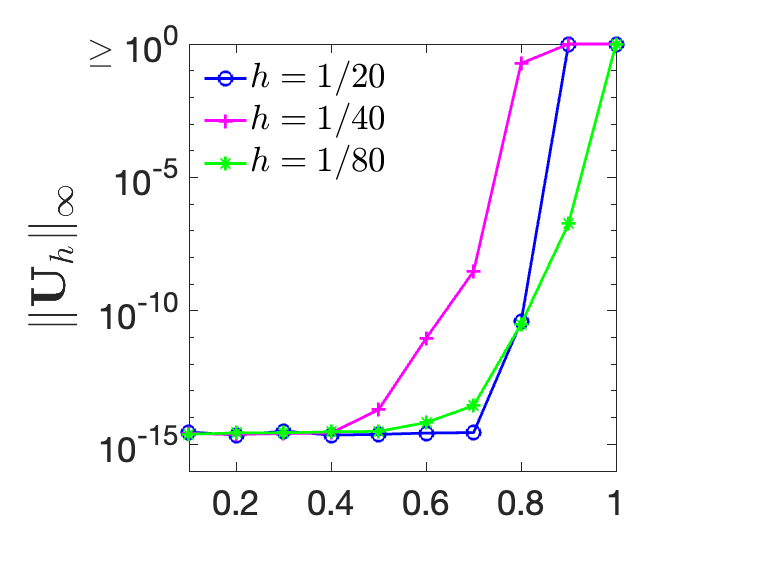} \hspace{-20.0pt}
		\includegraphics[width=2.5in,trim={0cm 0cm 1.75cm 0cm},clip]{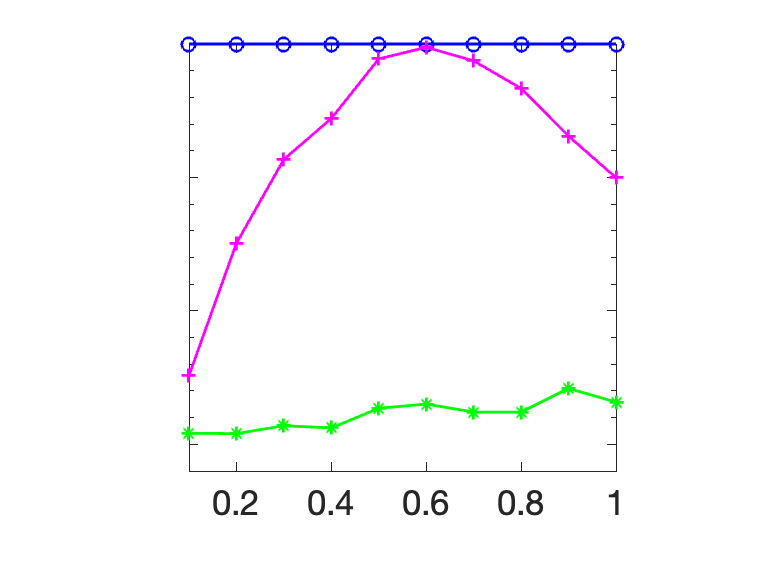}\hspace{-20pt}
		\includegraphics[width=2.5in,trim={0cm 0cm 1.75cm 0cm},clip]{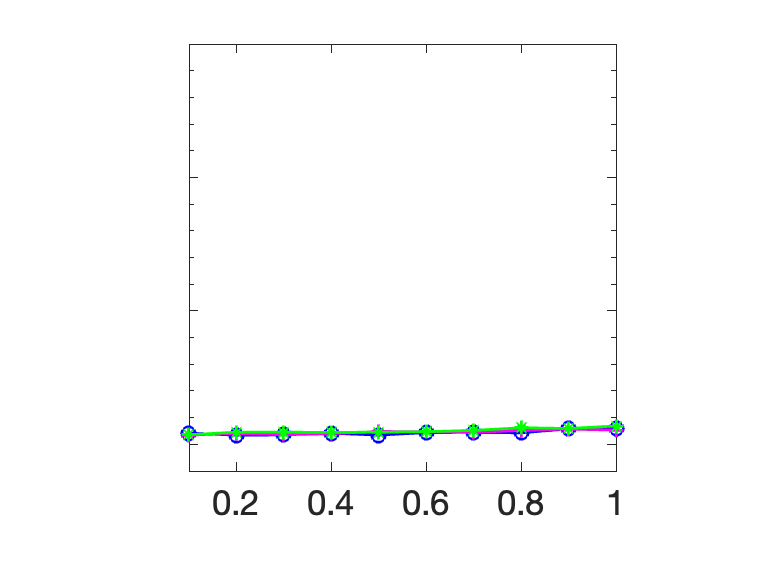}\hspace{-20pt}
		\phantom{{\includegraphics[width=2.5in,trim={0cm 0cm 1.75cm 0cm},clip]{pec_max_norm_m_1_nb_der_0.png}}}\hspace{-20pt}
	        \phantom{{\includegraphics[width=2.5in,trim={0cm 0cm 1.75cm 0cm},clip]{pec_max_norm_m_1_nb_der_0.png}}}
	\end{adjustbox} 
	\begin{adjustbox}{max width=1.0\textwidth,center}
		\includegraphics[width=2.5in,trim={0cm 0cm 1.75cm 0cm},clip]{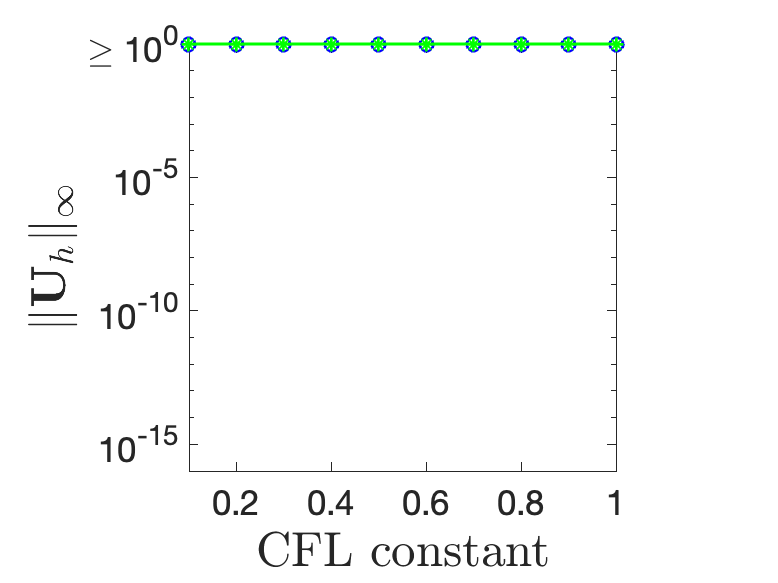} \hspace{-20pt}
		\includegraphics[width=2.5in,trim={0cm 0cm 1.75cm 0cm},clip]{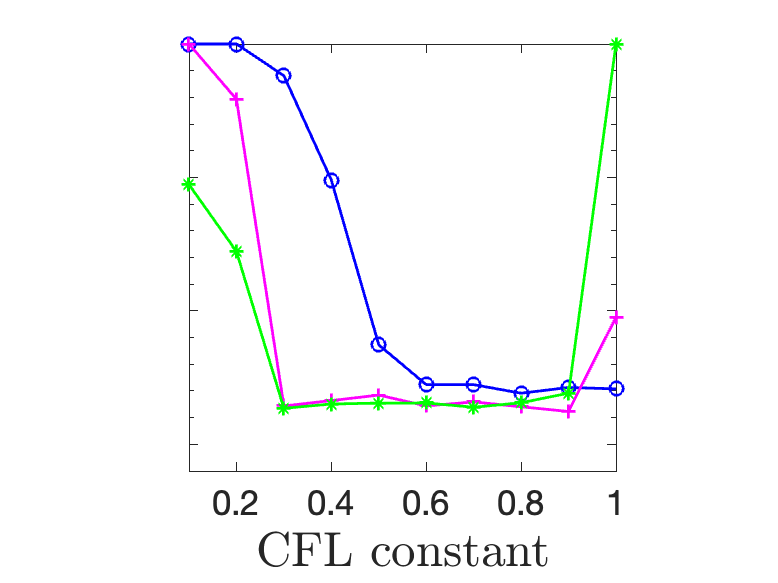}\hspace{-20pt}
		\includegraphics[width=2.5in,trim={0cm 0cm 1.75cm 0cm},clip]{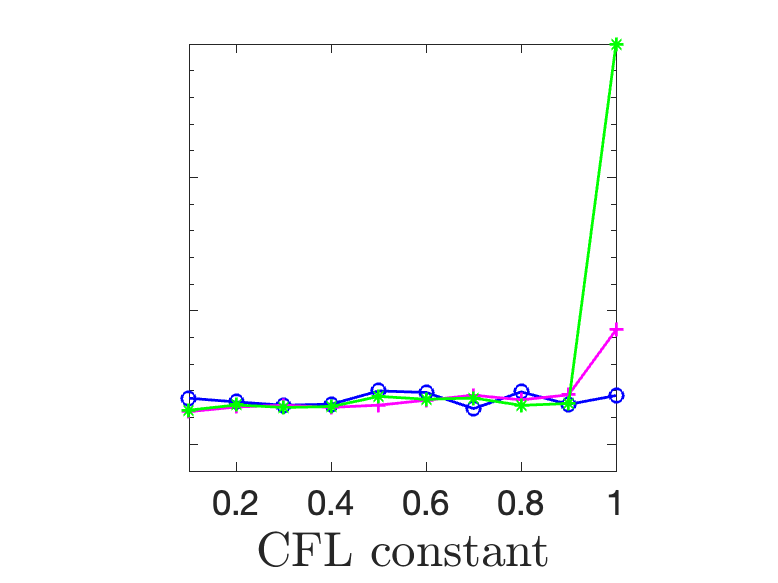}\hspace{-20pt}
		\includegraphics[width=2.5in,trim={0cm 0cm 1.75cm 0cm},clip]{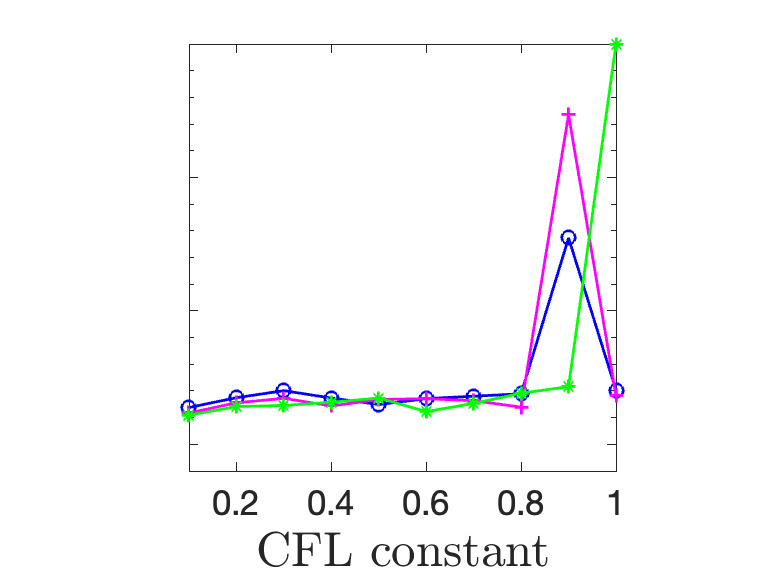}\hspace{-20pt}
		\includegraphics[width=2.5in,trim={0cm 0cm 1.75cm 0cm},clip]{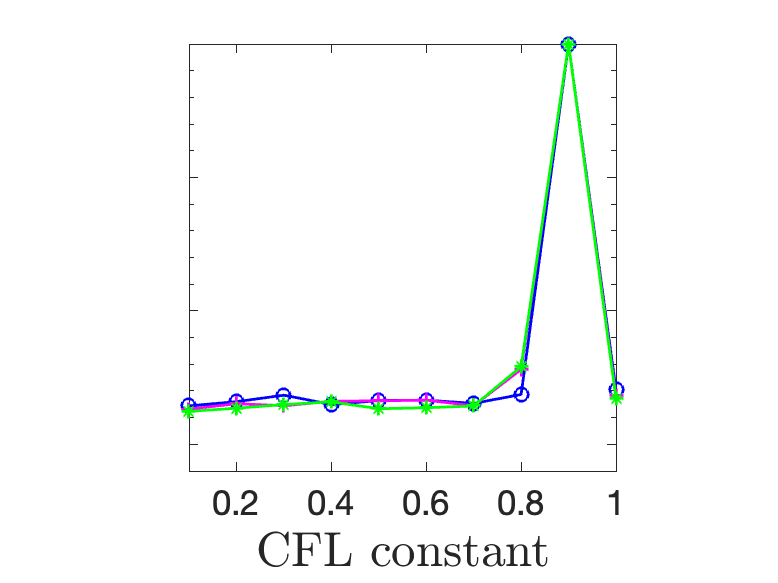}
	\end{adjustbox}  
       \caption{Maximum value of the maximum norm of the numerical solution over $10^5$ time steps as a function of the CFL constant for different mesh sizes in 2-D for PEC boundary conditions.
       The rows are for different $m$: 1 to 2 from the top to bottom.The columns are for the maximum order of the considered spatial derivatives ($N_d$) at the boundary: 0 to 4 from the left to the right.}
       \label{fig:investigation_long_time_2D_PEC}
\end{figure}
For $m=1$, 
    the Hermite-Taylor correction function method is stable for all considered mesh sizes and CFL constants when $N_d = 2$, 
    the maximum value of $N_d$ in this situation.
As for $m=2$, 
    the stability of the Hermite-Taylor correction function method is clearly improving as $N_d$ increases. 
More specifically, 
    the method for $m=2$ is stable for all considered meshes when the CFL constant is smaller than $0.8$ and $N_d \geq 2$.

Let us now consider interface problems. 
We investigate the stability using the same setup as for the embedded boundary problems.
In this situation, 
    the domain is $\Omega =[0,1]\times[0,1]$, 
	$\Gamma$ is an interface and the subdomain $\Omega^+$ is enclosed by $\Gamma$.
We consider periodic boundary conditions.
We are seeking a numerical solution in $\Omega^+$ and $\Omega^-$.
We consider $\mu^+=1$,
	$\epsilon^+=1$, 
	$\mu^-=2$ and $\epsilon^-=2.25$
\footnote{Note that we set $Z^+=Z^-=1$ and $c^+=c^-=1$ in the correction function functional for all numerical examples in this work.}.
Fig.~\ref{fig:investigation_long_time_2D_dielectric} illustrates the evolution of the maximum norm of the 
	numerical solution over $10^5$ time steps for $m=1-2$ using 
    different mesh sizes,
    values of $N_d$ and CFL constants.
For $m=1$ and $N_d\geq 1$, 
    the Hermite-Taylor correction function method is stable for all considered CFL constants. 
Regarding $m=2$, 
    the method is stable for $N_d\geq 3$ and a CFL constant under 0.8. 
\begin{figure}   
	\centering
	\begin{adjustbox}{max width=1.0\textwidth,center}
		\includegraphics[width=2.5in,trim={0cm 0cm 1.75cm 0cm},clip]{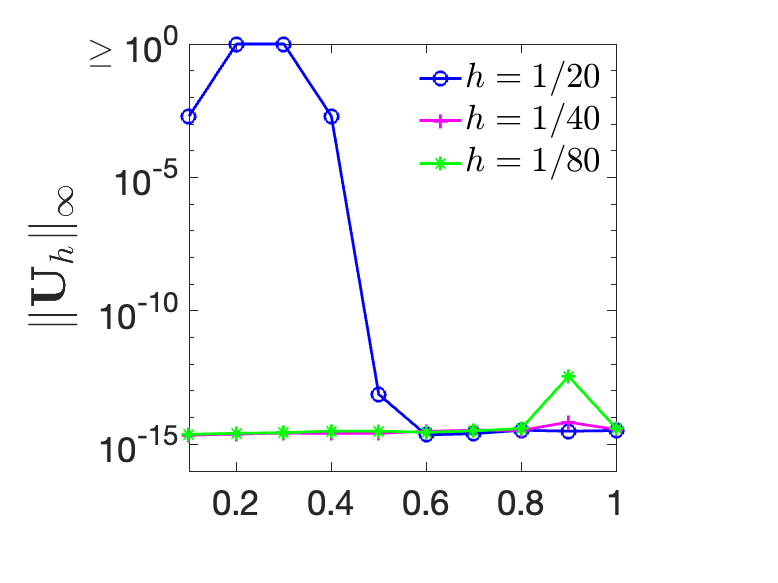} \hspace{-20.0pt}
		\includegraphics[width=2.5in,trim={0cm 0cm 1.75cm 0cm},clip]{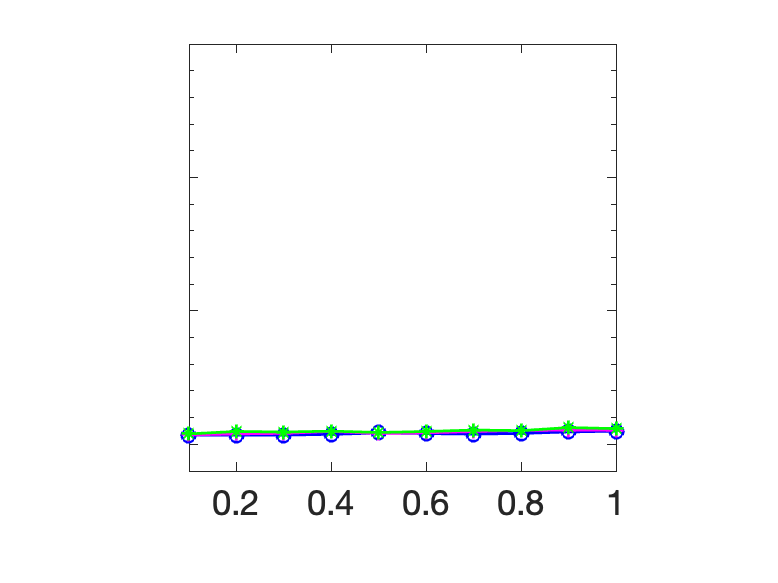}\hspace{-20pt}
		\includegraphics[width=2.5in,trim={0cm 0cm 1.75cm 0cm},clip]{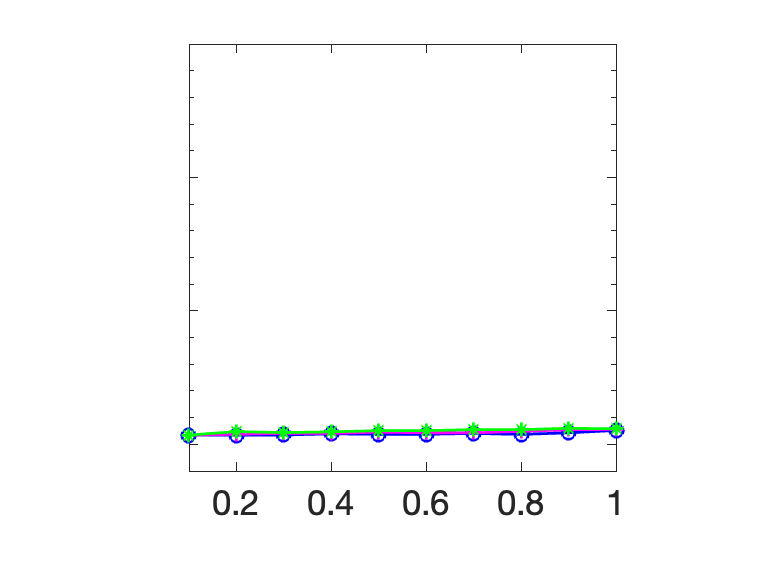}\hspace{-20pt}
		\phantom{{\includegraphics[width=2.5in,trim={0cm 0cm 1.75cm 0cm},clip]{interface_max_norm_m_1_nb_der_0.png}}}\hspace{-20pt}
	        \phantom{{\includegraphics[width=2.5in,trim={0cm 0cm 1.75cm 0cm},clip]{interface_max_norm_m_1_nb_der_0.png}}}
	\end{adjustbox} 
	\begin{adjustbox}{max width=1.0\textwidth,center}
		\includegraphics[width=2.5in,trim={0cm 0cm 1.75cm 0cm},clip]{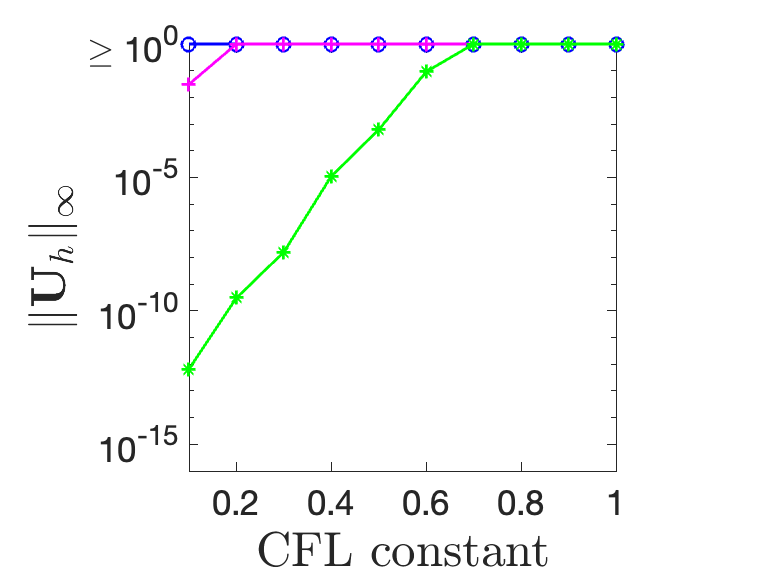} \hspace{-20pt}
		\includegraphics[width=2.5in,trim={0cm 0cm 1.75cm 0cm},clip]{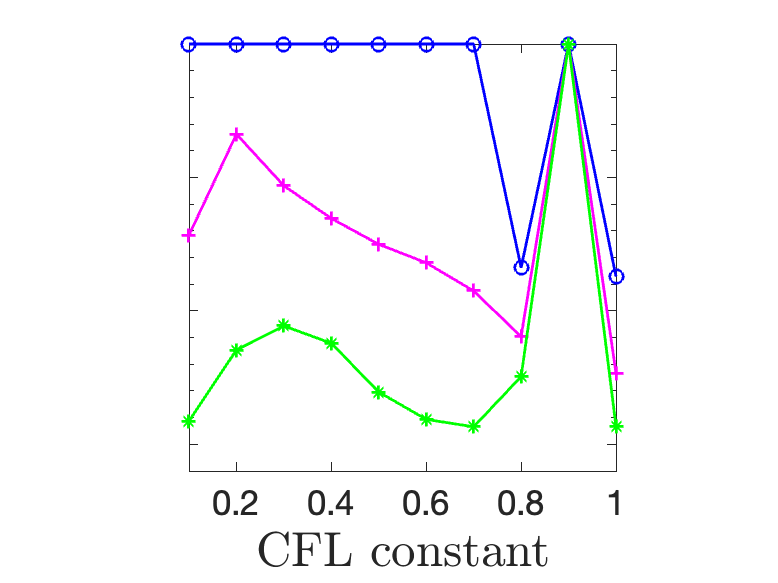}\hspace{-20pt}
		\includegraphics[width=2.5in,trim={0cm 0cm 1.75cm 0cm},clip]{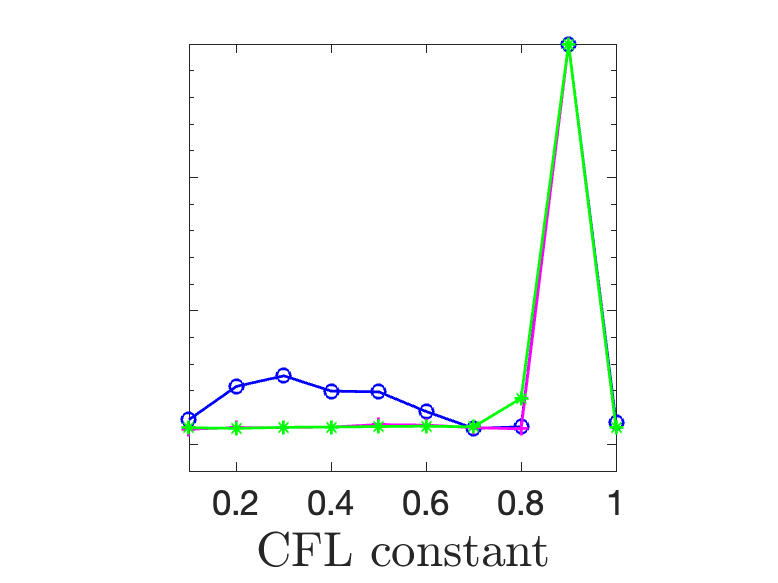}\hspace{-20pt}
		\includegraphics[width=2.5in,trim={0cm 0cm 1.75cm 0cm},clip]{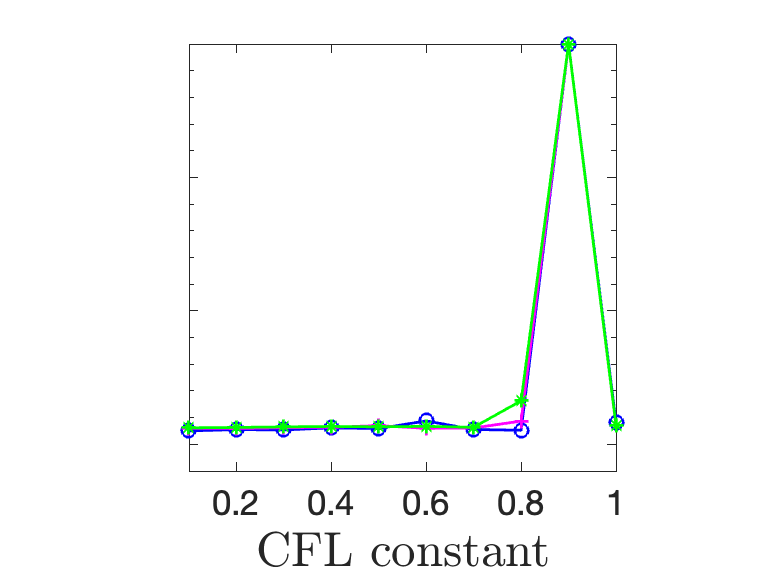}\hspace{-20pt}
		\includegraphics[width=2.5in,trim={0cm 0cm 1.75cm 0cm},clip]{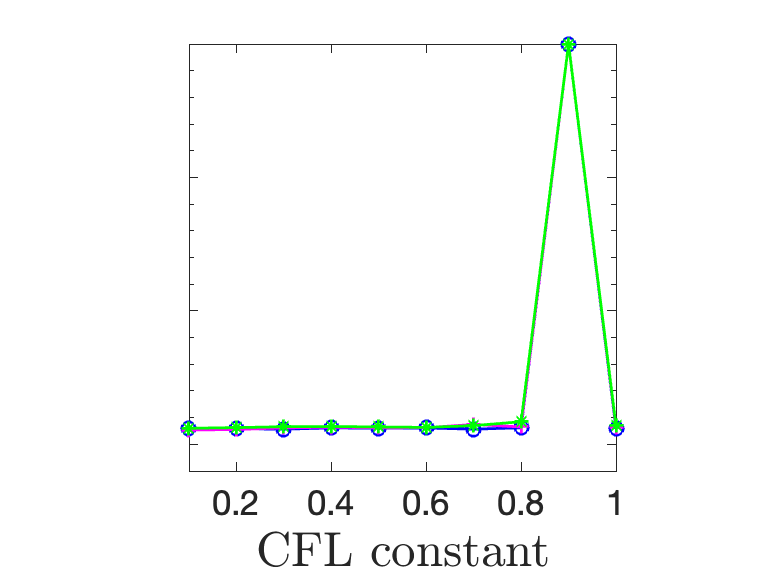}
	\end{adjustbox}  
       \caption{Maximum value of the maximum norm of the numerical solution over $10^5$ time steps as a function of the CFL constant for different mesh sizes in 2-D for interface conditions.
       The rows are for different $m$: 1 to 2 from the top to bottom.The columns are for the maximum order of the considered spatial derivatives ($N_d$) at the boundary: 0 to 4 from the left to the right.}
       \label{fig:investigation_long_time_2D_dielectric}
\end{figure}
{\color{black}
\subsubsection{Condition Number of Correction Function Matrices}

We first investigate the condition number of the CF matrices for the embedded boundary problem described previously.
The maximum condition number of CFM matrices as a function of the CFL constant for different mesh sizes 
    and values of $N_d$ is illustrated in Fig.~\ref{fig:investigation_condition_number_2D_PEC}. 
As in the one space dimension case,
    the condition number increases as $m$ increases. 
For $m=2$, 
    the condition number increases when the CFL constant diminishes, 
    suggesting that we should choose the largest CFL constant that leads to a stable method. 
As for $N_d$,
    we observe that the CFM matrices are better conditioned when $N_d$ goes from 0 to 1.
However, 
    for $N_d>1$,
    the condition number increases as $N_d$ increases.
For $m=2$, 
	we notice that the condition number increases as  the mesh size decreases.
    
Let us now consider the interface problem described in the previous subsection. 
Fig.~\ref{fig:investigation_condition_number_2D_dielectric} illustrates the condition number of 
    the CFM matrices as a function of the CFL constant for different meshes and values of $N_d$. 
For $m=1$, 
	we observe an improvement of the condition number when $N_d$ goes from $0$ to $1$ but 
	it increases as the CFL constant diminishes for $N_d=0-1$.
For $m=2$ and $N_d<3$, 
    the condition number diminishes as $N_d$ increases while it increases as the CFL constant diminishes.
For $m=2$ and $N_d=4$,
    the condition number decreases as the CFL decreases.

\begin{figure}   
	\centering
	\begin{adjustbox}{max width=1.0\textwidth,center}
		\includegraphics[width=2.5in,trim={0cm 0cm 1.75cm 0cm},clip]{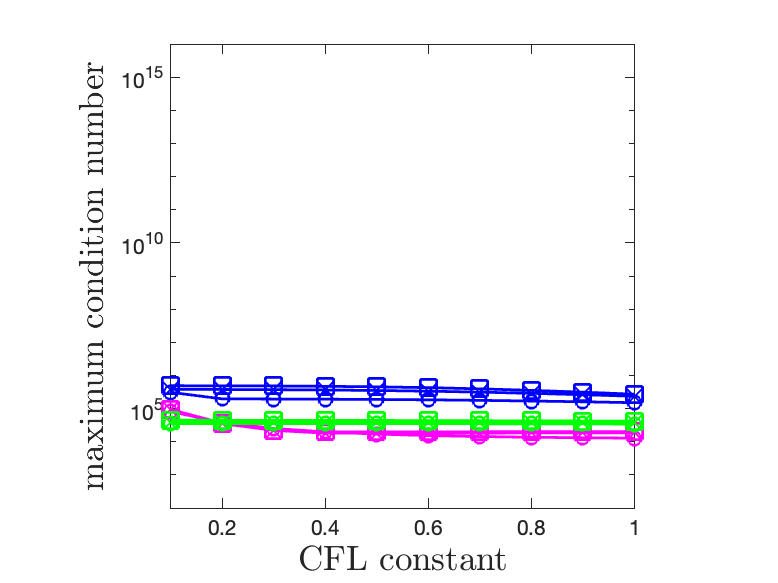} \hspace{-20.0pt}
		\includegraphics[width=2.5in,trim={0cm 0cm 1.75cm 0cm},clip]{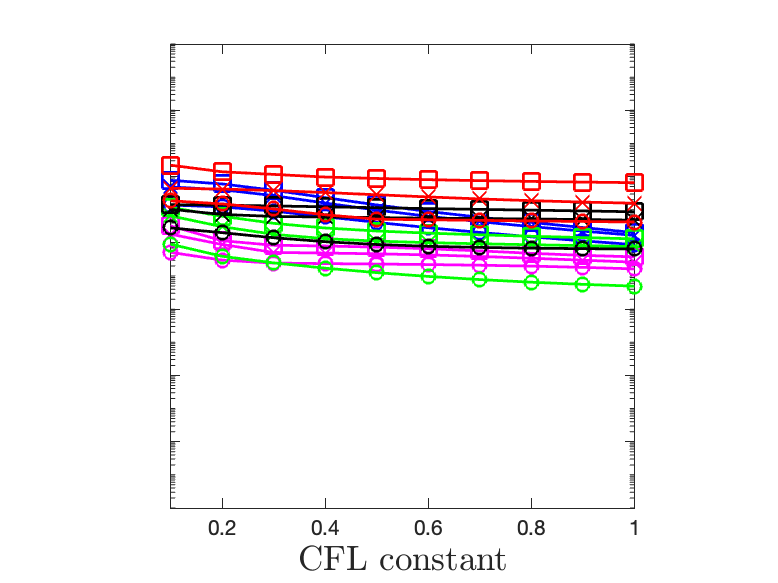}
	\end{adjustbox} 
       \caption{Maximum condition number of CF matrices as a function of the CFL constant for different mesh sizes and number of derivatives in 2-D for PEC boundary conditions. 
       The left and right plots illustrate the results for respectively $m=1$ and $m=2$.
       The circle, cross and square markers stand for $\Delta x = \frac{1}{25}$, $\Delta x = \frac{1}{200}$ and $\Delta x = \frac{1}{1600}$. 
       The colors blue, magenta, green, black and red are respectively for $N_d = 0 - 4.$}
       \label{fig:investigation_condition_number_2D_PEC}
\end{figure}

Due to the condition number of the CF matrices, 
    we limit the value of $m$ to 2 in the multi-dimensional case. 
Nevertheless,
    we obtain third-order and fifth-order Hermite-Taylor correction function methods that are stable under a CFL constant of around 1 and 0.8 respectively.
Future research will explore a collocation method, 
    similar to what is done in \cite{Zhou2024},
    to obtain better conditioned linear systems for the correction functions and therefore to consider larger values of $m$.
\begin{figure}   
	\centering
	\begin{adjustbox}{max width=1.0\textwidth,center}
		\includegraphics[width=2.5in,trim={0cm 0cm 1.75cm 0cm},clip]{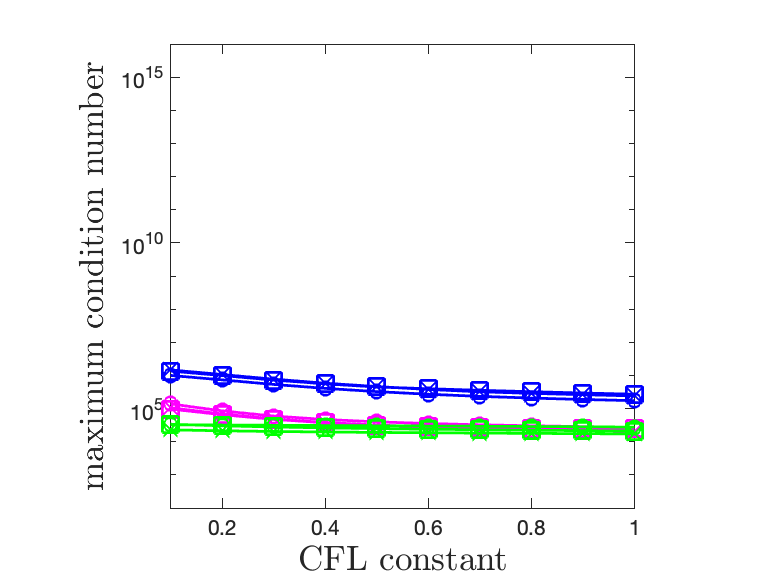} \hspace{-20.0pt}
		\includegraphics[width=2.5in,trim={0cm 0cm 1.75cm 0cm},clip]{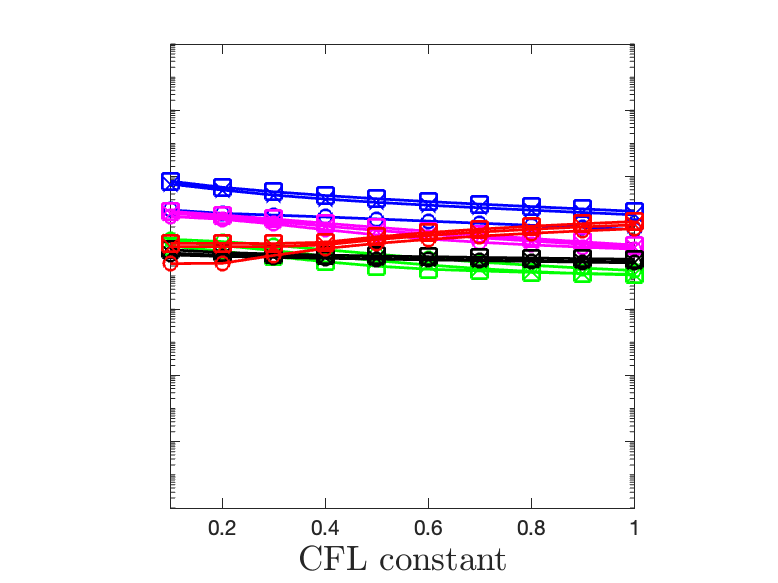}
	\end{adjustbox} 
       \caption{Maximum condition number of CF matrices as a function of the CFL constant for different mesh sizes and number of derivatives in 2-D for interface conditions. 
       The left and right plots illustrate the results for respectively $m=1$ and $m=2$.
       The circle, cross and square markers stand for $\Delta x = \frac{1}{25}$, $\Delta x = \frac{1}{200}$ and $\Delta x = \frac{1}{1600}$. 
       The colors blue, magenta, green, black and red are respectively for $N_d = 0 - 4$.}
       \label{fig:investigation_condition_number_2D_dielectric}
\end{figure}
}

\subsubsection{Accuracy}

Let us now investigate the accuracy of the Hermite-Taylor correction function method. 
Since the degree of the correction function polynomials is $2\,m$, 
	we should expect third and fifth order convergence, respectively, for Hermite-Taylor correction function methods 
	with $m=1$ and $m=2$.
We use a CFL constant of 0.9 and 0.7 for the third and fifth order Hermite-Taylor correction function methods.

\subsubsection*{Embedded Boundary Problems}

\begin{figure}
 	\centering
	\includegraphics[width=2.0in]{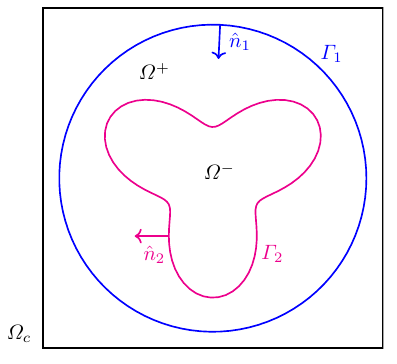}
       \caption{Geometry of the computational domain $\Omega_c$.}
	\label{fig:geo_pblm_5}
\end{figure}
We first consider the computational domain $\Omega_c = [-0.1,1.1]\times[-0.1,1.1]$,
	illustrated in Fig.~\ref{fig:geo_pblm_5},
	and the time interval is $I=[0,1]$.
In this situation, 
	we are seeking a numerical solution in $\Omega^+$, 
    enclosed by $\Gamma_1$ and $\Gamma_2$.
The physical parameters are $\mu=1$ and $\epsilon=1$. 
The initial data, 
	and boundary conditions on $\Gamma_1$ and $\Gamma_2$ are chosen so that the solution in $\Omega^+$ is given by 
\begin{equation}
	\begin{aligned}
	H_x =&\,\, -\frac{1}{\sqrt{2}}\sin(\omega\,\pi\,x)\,\cos(\omega\,\pi\,y)\,\sin(\sqrt{2}\,\omega\,\pi\,t), \\
	H_y =&\,\, \frac{1}{\sqrt{2}}\cos(\omega\,\pi\,x)\,\sin(\omega\,\pi\,y)\,\sin(\sqrt{2}\,\omega\,\pi\,t), \\ 
	E_z =&\,\, \sin(\omega\,\pi\,x)\,\sin(\omega\,\pi\,y)\,\cos(\sqrt{2}\,\omega\,\pi\,t),
	\end{aligned}
\end{equation}
	where $\omega = 20$.
{\color{black} Fig.~\ref{fig:conv_embedded_boundary} and Fig.~\ref{fig:conv_divH_embedded_boundary} show the convergence plots of the electromagnetic fields and the divergence of the magnetic field in the $L^2$-norm for $m=1$ and $m=2$.} 
\begin{figure}   
	\centering
	\begin{adjustbox}{max width=1.0\textwidth,center}
		\includegraphics[width=2.5in,trim={0.0cm 0cm 1.75cm 0cm},clip]{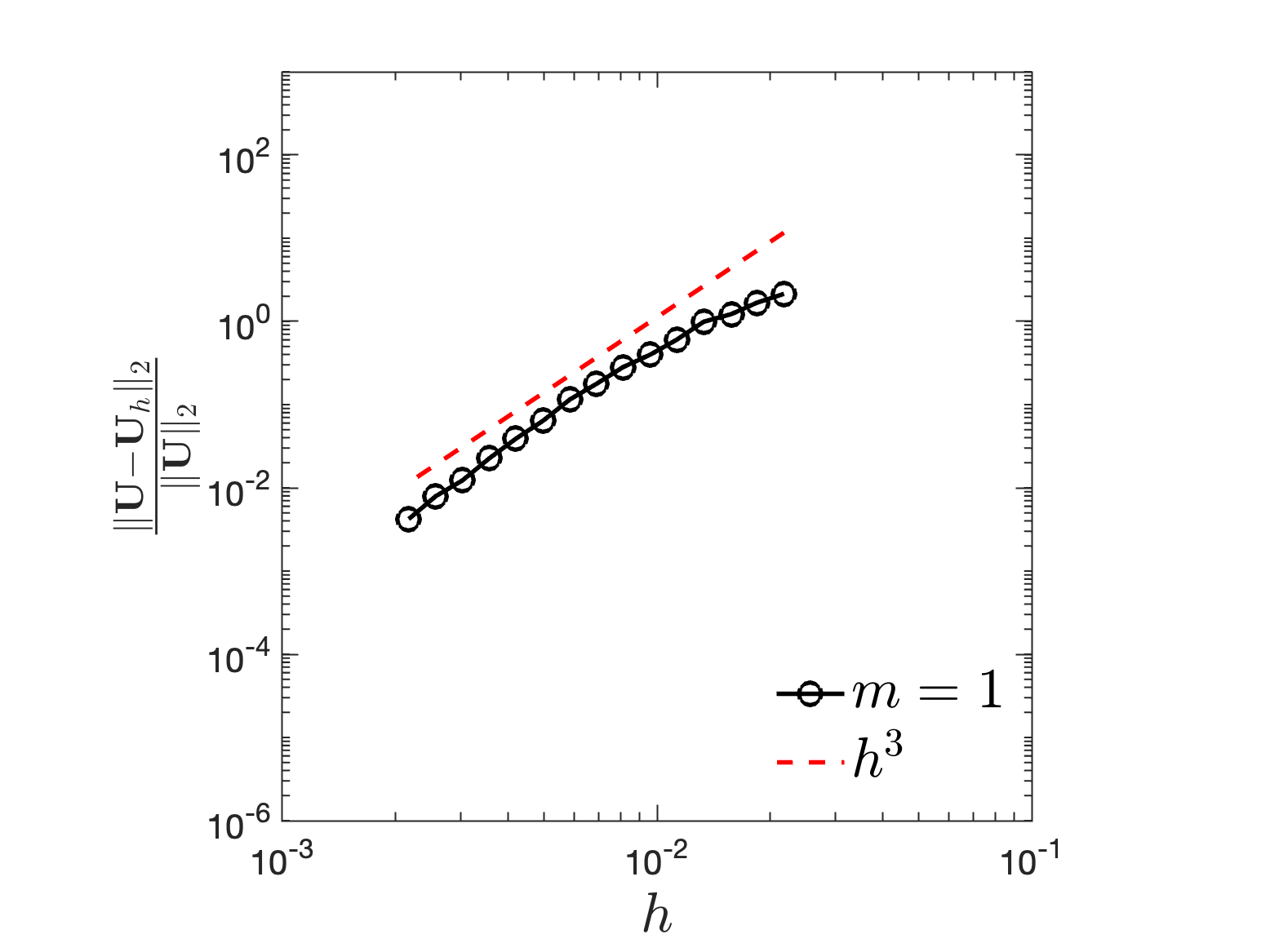} \hspace{-15pt}
		\includegraphics[width=2.5in,trim={0.0cm 0cm 1.75cm 0cm},clip]{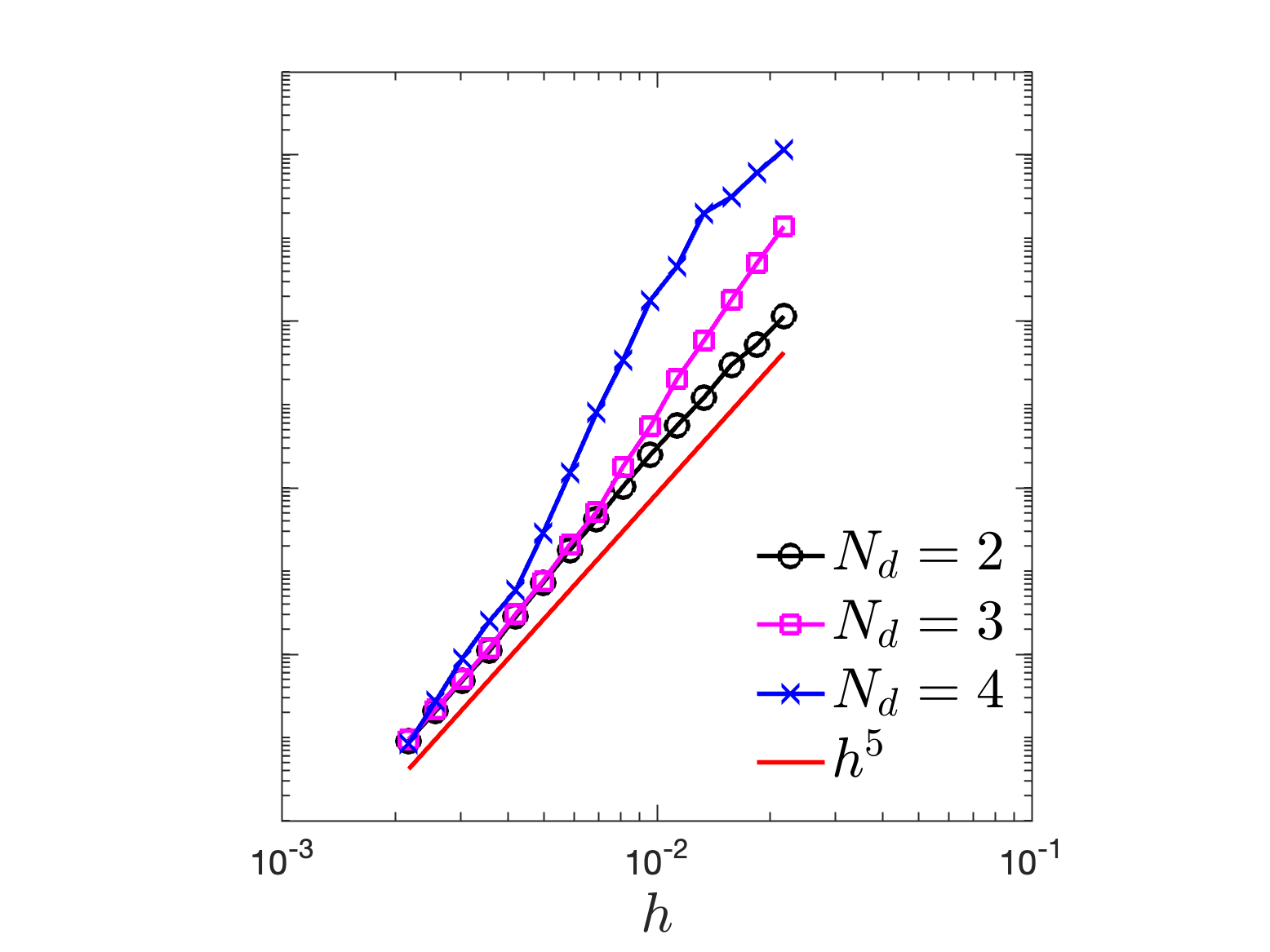}
	\end{adjustbox} 
       \caption{Convergence plots for an embedded boundaries problem with the geometry illustrated in Fig.~\ref{fig:geo_pblm_5} using the third and fifth order Hermite-Taylor correction function methods. 
       The left and right plots illustrate the results for respectively $m=1$ and $m=2$.
       Here $\mathbf{U} = [H_x, H_y, E_z]^T$.}
       \label{fig:conv_embedded_boundary}
\end{figure}
\begin{figure}   
	\centering
	\begin{adjustbox}{max width=1.0\textwidth,center}
		\includegraphics[width=2.5in,trim={0.0cm 0cm 1.75cm 0cm},clip]{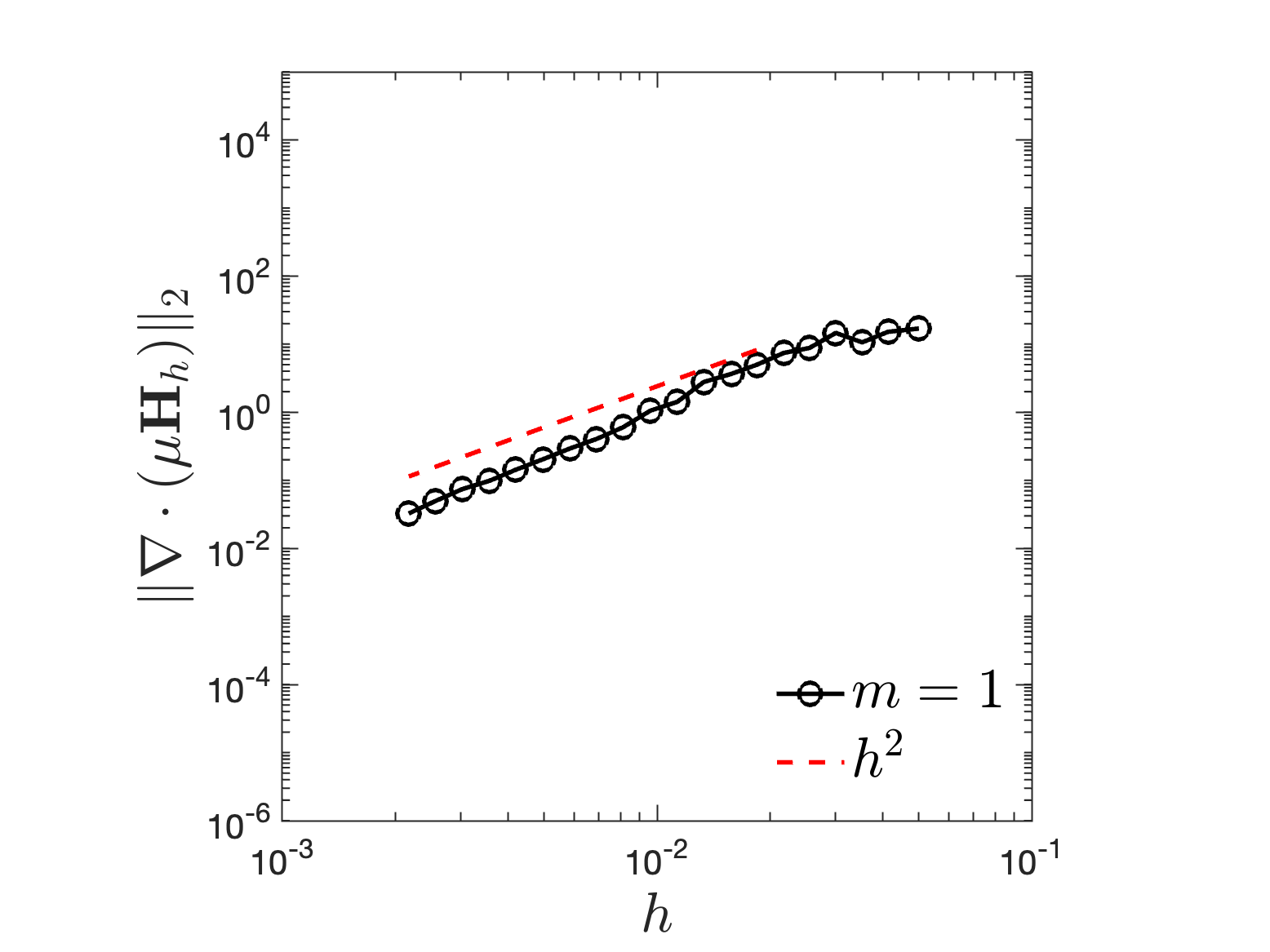} \hspace{-15pt}
		\includegraphics[width=2.5in,trim={0.0cm 0cm 1.75cm 0cm},clip]{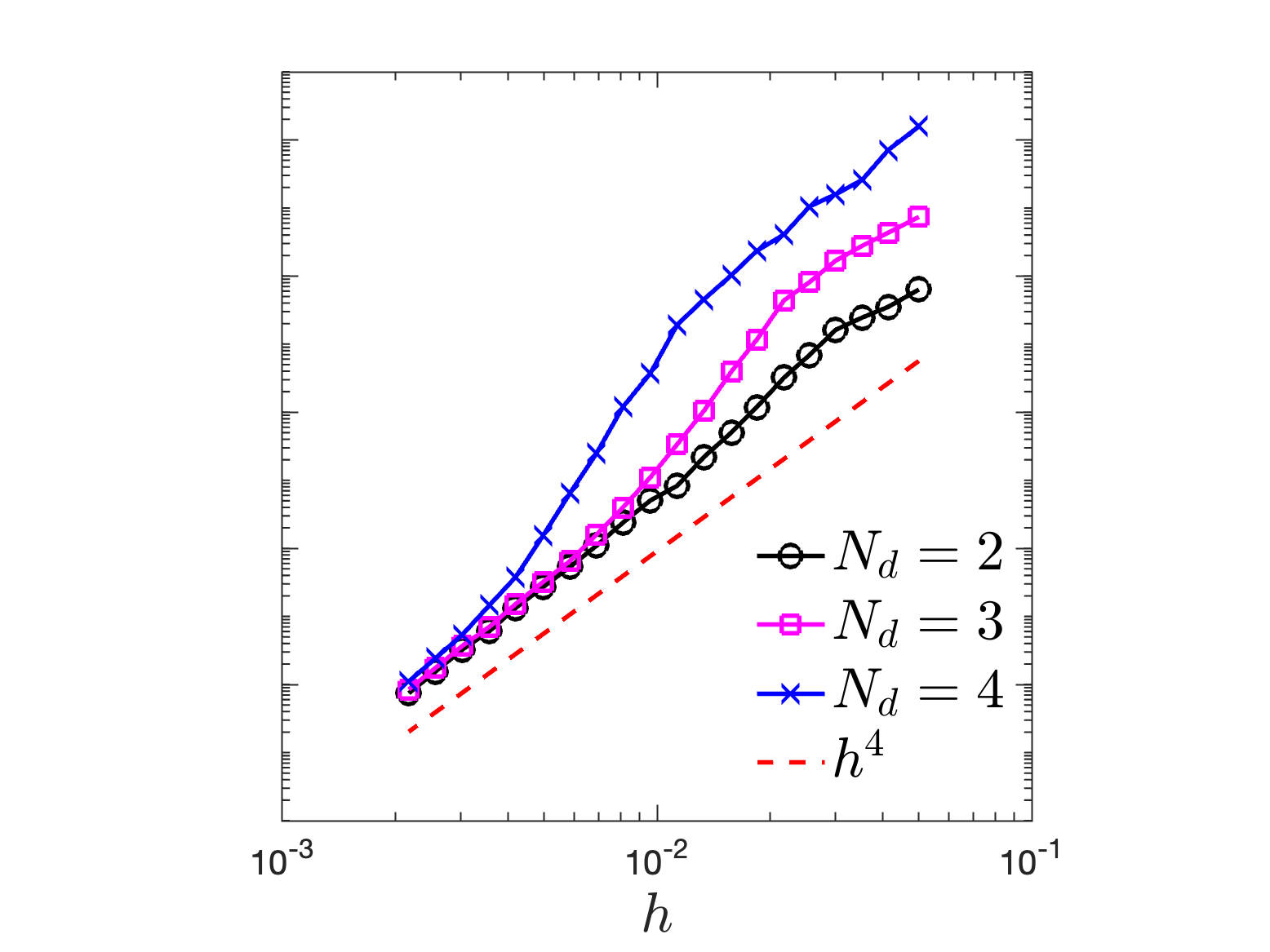}
	\end{adjustbox} 
       \caption{Convergence plots of the divergence of the magnetic field for an embedded boundaries problem with the geometry illustrated in Fig.~\ref{fig:geo_pblm_5} using the third and fifth order Hermite-Taylor correction function methods. 
       The left and right plots illustrate the results for respectively $m=1$ and $m=2$.}
       \label{fig:conv_divH_embedded_boundary}
\end{figure}
{\color{black} For $m=1$, 
    we set $N_d=2$ so that the method is stable and 
    we observe a clear third-order convergence for the electromagnetic fields and second-order convergence for the divergence of the magnetic field,
    as expected.}
As for $m=2$,
    we consider $N_d=2-4$, 
    leading to a stable method when the CFL constant is $0.7$.
{\color{black} When $N_d=2$, 
    we observe a clear fifth-order convergence for the electromagnetic fields and fourth-order for $\nabla\cdot(\mu\mathbold{H}_h)$.}
However, 
    as $N_d$ increases, 
    the relative error also increases and this makes the method inaccurate for coarser meshes. 
Note that the boundary condition \eqref{eq:boundary_condition} varies in space and time, 
and enforces a standing wave with a wavelength of $0.1$. 
In this situation, 
    for coarser meshes,
    it is more challenging for the minimization problem to provide accurate correction functions,
    particularly when we consider additional constraints on the spatial derivatives.
The approximations of electromagnetic fields at the final time are illustrated in Fig.~\ref{fig:pblm_5_caption}.
\begin{figure}   
	\centering
	\begin{adjustbox}{max width=1.0\textwidth,center}
		\includegraphics[width=2.5in,trim={2.5cm 0cm 1.75cm 0cm},clip]{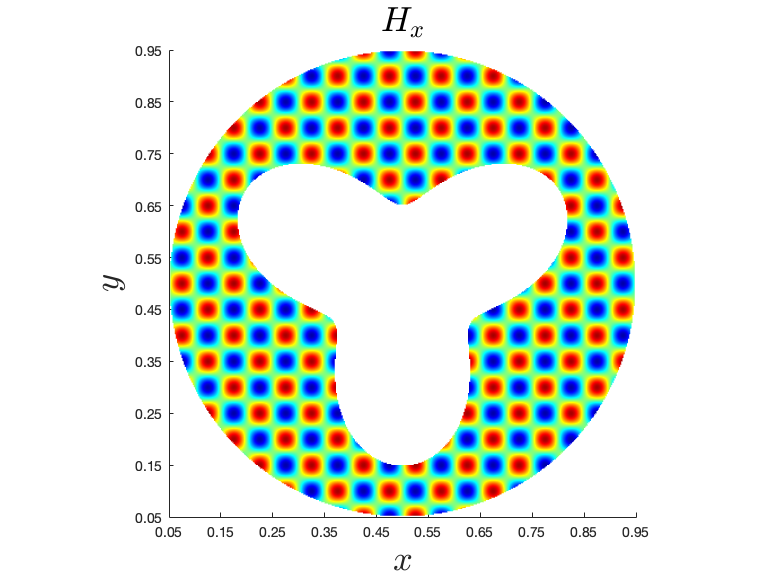}
		\includegraphics[width=2.5in,trim={2.5cm 0cm 1.75cm 0cm},clip]{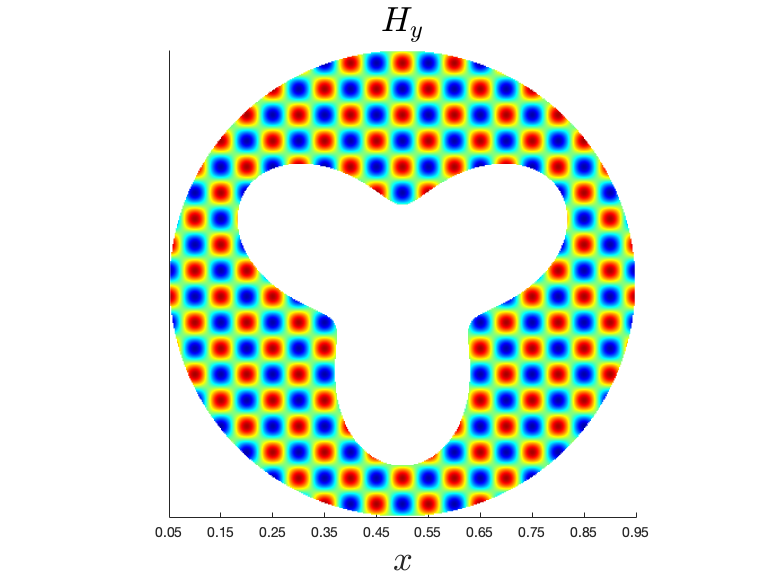}
	   	\includegraphics[width=2.5in,trim={2.5cm 0cm 1.75cm 0cm},clip]{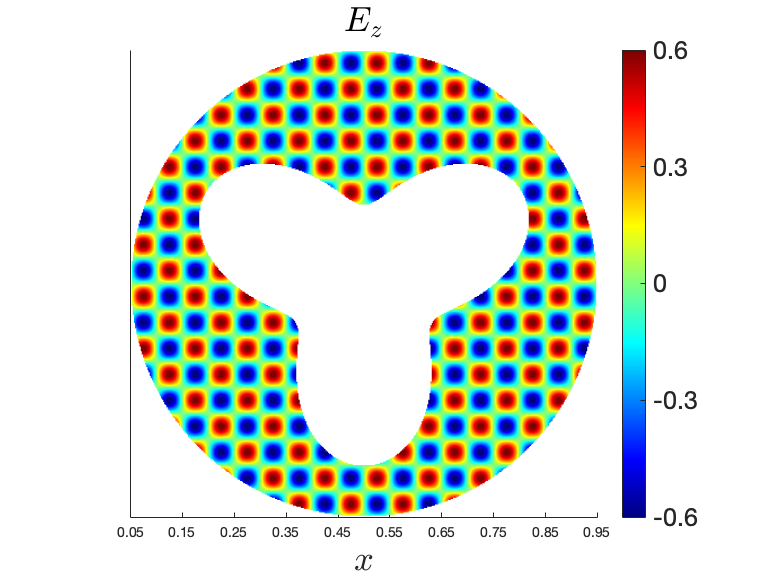}
	\end{adjustbox} 
       \caption{The components $H_x$, $H_y$ and $E_z$ at the final time $t_f=1$ for an embedded boundary problem with the geometry illustrated in Fig.~\ref{fig:geo_pblm_5} using the fifth-order Hermite-Taylor correction function method and $h=\frac{1}{460}$.}
       \label{fig:pblm_5_caption}
\end{figure}
For the remaining examples in this work, 
    we set $N_d$ to 2 and 4 for respectively $m=1$ and $m=2$.

Let us now consider a circular cavity problem.
The computational domain is $\Omega_c = [-1.1,1.1]\times[-1.1,1.1]$ and the 
	embedded boundary $\Gamma$ is a circle of unit radius and centered 
	at $(0,0)$ that encloses the physical domain $\Omega$. 
The time interval is $I=[0,1]$ and we enforce PEC boundary conditions on $\Gamma$. 
The physical parameters are $\epsilon = 1$ and $\mu = 1$, 
	and the solution in cylindrical coordinates in $\Omega$ is given by 
\begin{equation}
	\begin{aligned}
		H_\rho(\rho,\phi,t) =&\,\, \frac{i}{\alpha_{i,j}\,\rho}\,J_i(\alpha_{i,j}\,\rho)\,\sin(i\,\phi)\,\sin(\alpha_{i,j}\,t), \\
		H_\phi(\rho,\phi,t) =&\,\,\frac{1}{2}\,\big(J_{i-1}(\alpha_{i,j}\,\rho)-J_{i+1}(\alpha_{i,j}\,\rho)\big)\,\cos(i\,\phi)\,\sin(\alpha_{i,j}\,t), \\
		E_z(\rho,\phi,t) =&\,\, J_i(\alpha_{i,j}\,\rho)\,\cos(i\,\phi)\,\cos(\alpha_{i,j}\,t),
	\end{aligned}
\end{equation}
	where $\alpha_{i,j}$ is the $j$-th positive real root of the $i$-order Bessel function of first kind $J_i$,
	$i=2$ and $j=11$. 
{\color{black} The left plot of Fig.~\ref{fig:conv_embedded_boundary_cavity_problem}  shows that we obtain the expected $2\,m+1$ rates of convergence for the electromagnetic fields in the $L^2$-norm.
The $2\,m$ rates of convergence for the divergence of the magnetic field are illustrated in the right plot of Fig.~\ref{fig:conv_embedded_boundary_cavity_problem}.}
 Fig.~\ref{fig:pblm_4_caption} illustrates the approximations of electromagnetic fields at the final time.
\begin{figure}   
	\centering
	\begin{adjustbox}{max width=1.0\textwidth,center}
	\includegraphics[width=2.5in,trim={0.0cm 0cm 1.75cm 0cm},clip]{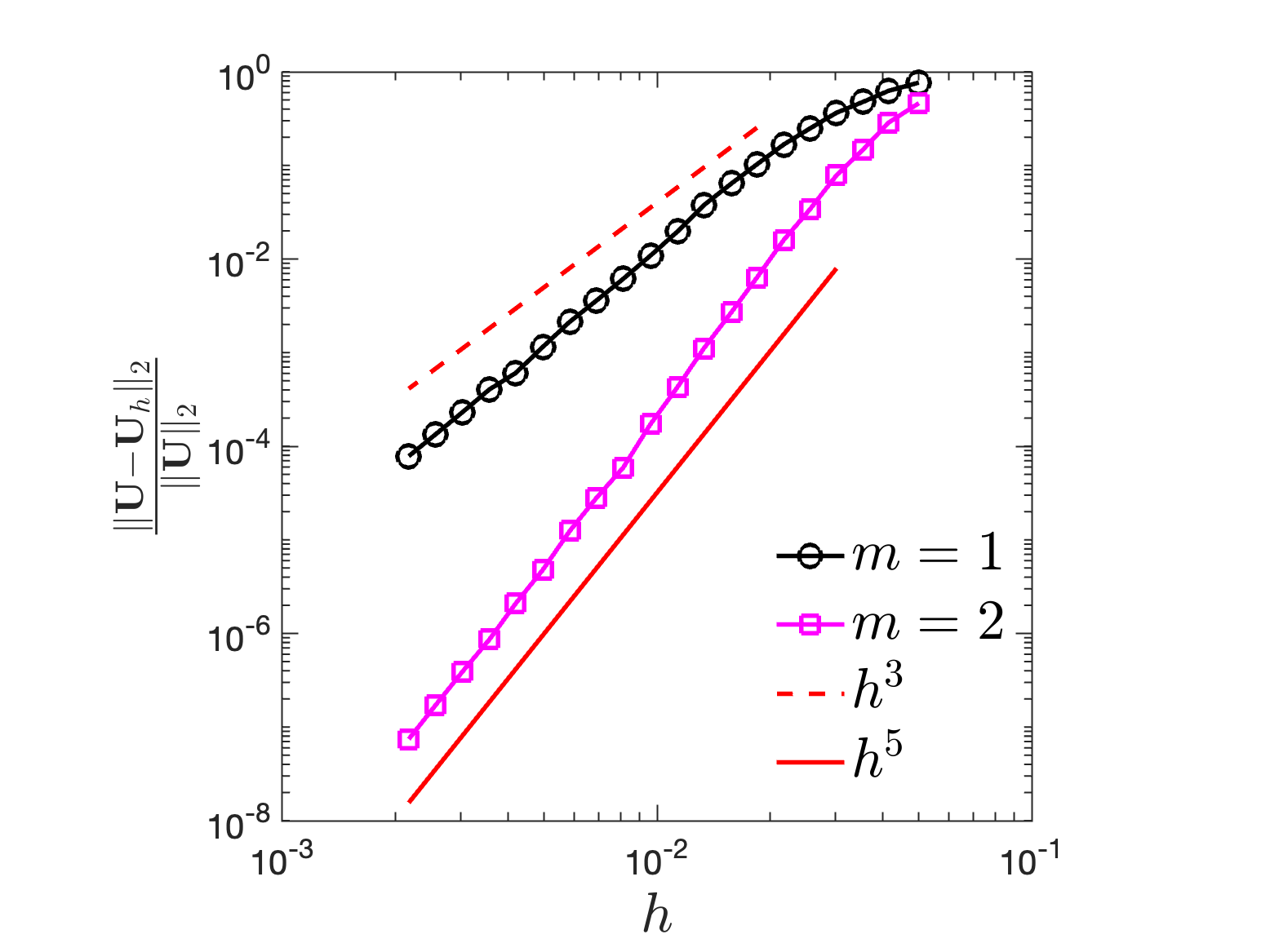}\hspace{-15pt}
	\includegraphics[width=2.5in,trim={0.0cm 0cm 1.75cm 0cm},clip]{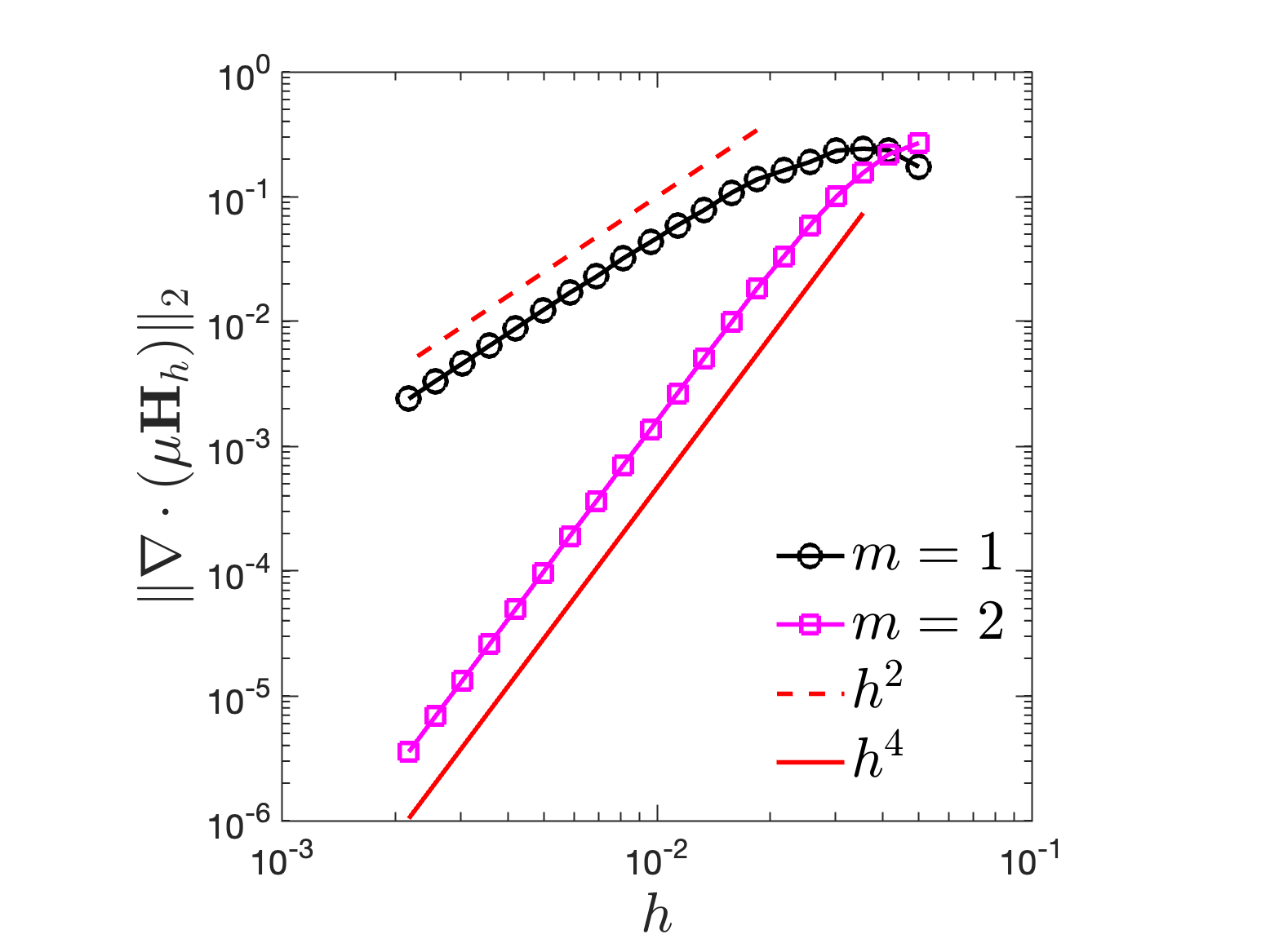}
	\end{adjustbox} 
       \caption{Convergence plots for a circular cavity problem using the third and fifth order Hermite-Taylor correction function methods. 
       The left and right plots show the convergence of the electromagnetic fields and the divergence of the magnetic fields. 
       Here $\mathbf{U} = [H_x, H_y, E_z]^T$.}
    \label{fig:conv_embedded_boundary_cavity_problem}
\end{figure}
\begin{figure}   
	\centering
	\begin{adjustbox}{max width=1.0\textwidth,center}
		\includegraphics[width=2.5in,trim={2.5cm 0cm 1.75cm 0cm},clip]{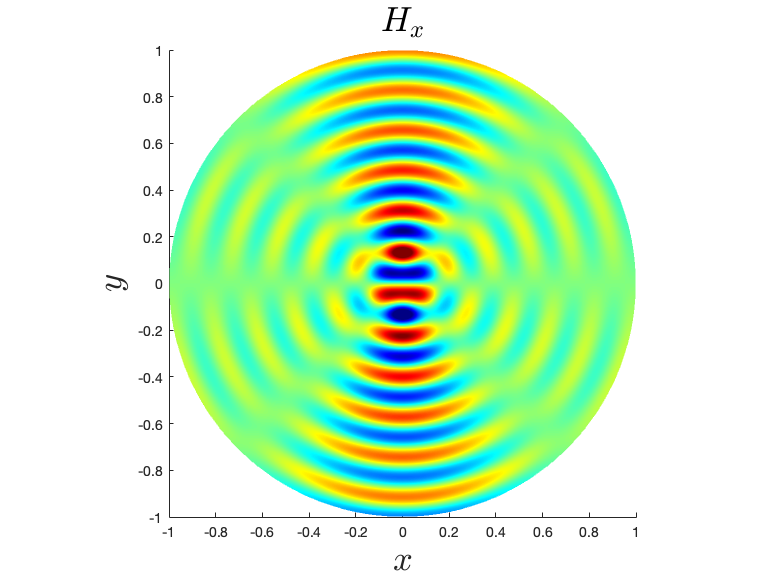}
		\includegraphics[width=2.5in,trim={2.5cm 0cm 1.75cm 0cm},clip]{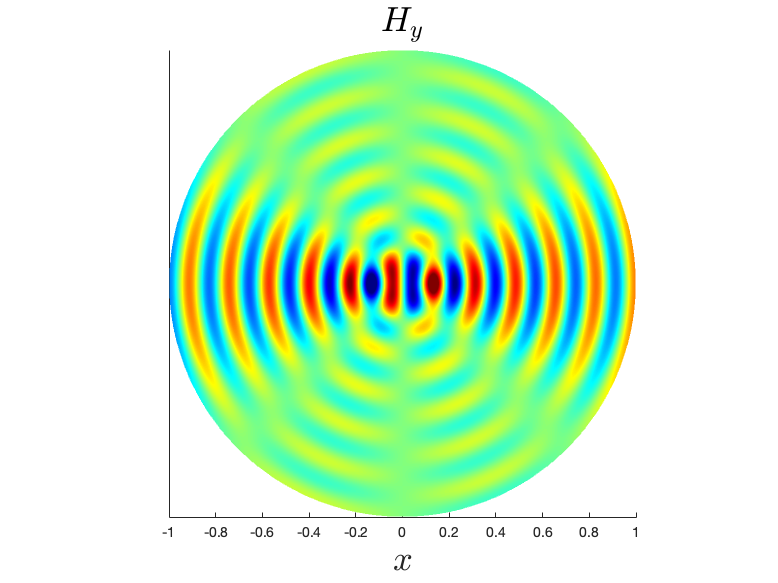}
	   	\includegraphics[width=2.5in,trim={2.5cm 0cm 1.75cm 0cm},clip]{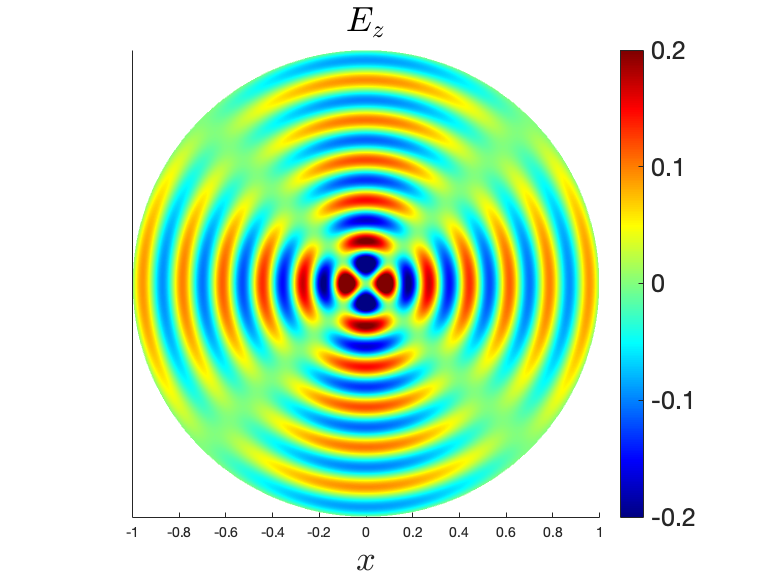}
	\end{adjustbox} 
       \caption{The components $H_x$, $H_y$ and $E_z$ at the final time $t_f=1$ for a circular cavity problem using the fifth-order Hermite-Taylor correction function method and $h=\frac{1}{460}$.}
       \label{fig:pblm_4_caption}
\end{figure}

\subsubsection*{Interface Problems with Analytic Solutions}

Let us now consider Maxwell's interface problems. 
First, 
	we consider a dielectric cylinder in free-space exposed to an excitation wave. 
The geometry of the computational domain consists of two concentric circles enclosed in $\Omega_c = [-1,1]\times[-1,1]$.
The first circle centered at $(0,0)$ has a radius of 0.8 and is an embedded boundary $\Gamma_1$ of $\Omega^+$.
The second circle also centered at $(0,0)$, 
	but with a radius $r_0=0.6$, which
	represents the interface $\Gamma_2$ between the subdomains 
	and enclosed the subdomain $\Omega^-$.
The time interval is set to $I=[0,1]$.
The initial and boundary conditions are chosen such that the solution in cylindrical coordinates is given by the real part of 
\begin{equation}
\begin{aligned}
H_{\theta}(r,\theta,t) =&\,\, \left\{ 
  \begin{array}{l l}
    -\frac{\mathfrak{i}\,k^-}{\omega\,\mu^-}\displaystyle\sum_{n=-\infty}^{\infty} C_n^{\text{tot}}\,J_n^\prime(k^-\,r)\,e^{\mathfrak{i}\,(n\,\theta+\omega\,t)}, & \text{if} \quad r\leq r_0, \\ 
    -\frac{\mathfrak{i}\,k^+}{\omega\,\mu^+}\displaystyle\sum_{n=-\infty}^{\infty} (\mathfrak{i}^{-n}\,J_n^\prime(k^+\,r) + C_n^{\text{scat}}\,H_n^{{(2)}^\prime}(k^+\,r))\,e^{\mathfrak{i}\,(n\,\theta+\omega\,t)}, & \text{if} \quad r>r_0,
  \end{array} \right.\\
H_r(r,\theta,t) =&\,\, \left\{ 
  \begin{array}{l l}
    -\frac{1}{\omega\,\mu^-\,r}\displaystyle\sum_{n=-\infty}^{\infty} n\,C_n^{\text{tot}}\,J_n(k^-\,r)\,e^{\mathfrak{i}\,(n\,\theta+\omega\,t)}, & \text{if} \quad r\leq r_0, \\ 
    -\frac{1}{\omega\,\mu^+\,r}\displaystyle\sum_{n=-\infty}^{\infty} n\,(\mathfrak{i}^{-n}\,J_n(k^+\,r) + C_n^{\text{scat}}\,H_n^{(2)}(k^+\,r))\,e^{\mathfrak{i}\,(n\,\theta+\omega\,t)}, & \text{if} \quad r>r_0,
  \end{array} \right. \\
E_z(r,\theta,t) =&\,\, \left\{ 
  \begin{array}{l l}
    \displaystyle\sum_{n=-\infty}^{\infty} C_n^{\text{tot}}\,J_n(k^-\,r)\,e^{\mathfrak{i}\,(n\,\theta+\omega\,t)}, & \text{if} \quad r\leq r_0, \\ 
    \displaystyle\sum_{n=-\infty}^{\infty} (\mathfrak{i}^{-n}\,J_n(k^+\,r) + C_n^{\text{scat}}\,H_n^{(2)}(k^+\,r))\,e^{\mathfrak{i}\,(n\,\theta+\omega\,t)}, &  \text{if} \quad r>r_0,
  \end{array} \right.
  \end{aligned}
\end{equation}
\normalsize
with 
\begin{equation}
\begin{aligned}
C_n^{\text{tot}} =&\,\, \mathfrak{i}^{-n}\,\frac{\tfrac{k^+}{\mu^+}\,(J_n^\prime(k^+\,r_0)\,H_n^{(2)}(k^+\,r_0)-H_n^{{(2)}^\prime}(k^+\,r_0)\,J_n(k^+\,r_0))}{\frac{k^-}{\mu^-}\,J_n^\prime(k^-\,r_0)\,H_n^{(2)}(k^+\,r_0)-\tfrac{k^+}{\mu^+}\,H_n^{{(2)}^\prime}(k^+\,r_0)\,J_n(k^-\,r_0)},\\
C_n^{\text{scat}}=&\,\, \mathfrak{i}^{-n}\,\frac{\tfrac{k^+}{\mu^+}\,J_n^\prime(k^+\,r_0)\,J_n(k^-\,r_0)-\tfrac{k^-}{\mu^-}\,J_n^\prime(k^-\,r_0)\,J_n(k^+\,r_0)}{\frac{k^-}{\mu^-}\,J_n^\prime(k^-\,r_0)\,H_n^{(2)}(k^+\,r_0)-\tfrac{k^+}{\mu^+}\,H_n^{{(2)}^\prime}(k^+\,r_0)\,J_n(k^-\,r_0)}.
\end{aligned}
\end{equation}
Here $\omega = 2\,\pi$,
	$k^\circ=\omega\,\sqrt{\mu^\circ\,\epsilon^\circ}$, 
	$J_n$ is the $n$-order Bessel function of first kind and $H_n^{(2)}$ is the $n$-order 
	Hankel function of second kind and the imaginary number is $\mathfrak{i}$ \cite{Taflove1993,Cai2003}. 
We set $\mu^+= 1$, 
	$\epsilon^+ = 1$,
	$\mu^- = 2$ and $\epsilon^-=2.25$.
Hence, 
	$\Omega^-$ is a magnetic dielectric material. 
In this situation,
	$H_x$ and $H_y$ are discontinuous at the interface and $E_z$ is continuous.
	
{\color{black}As shown in Fig.~\ref{fig:magnetic_dielectric_conv}, 
	we observe the expected $2\,m+1$ and $2\,m$ rates of convergence for respectively the electromagnetic fields and the divergence of the magnetic field, 
	even in the presence of discontinuities.}
\begin{figure}   
	\centering
	\begin{adjustbox}{max width=1.1\textwidth,center}
		\includegraphics[width=2.5in,trim={0.0cm 0cm 1.75cm 0cm},clip]{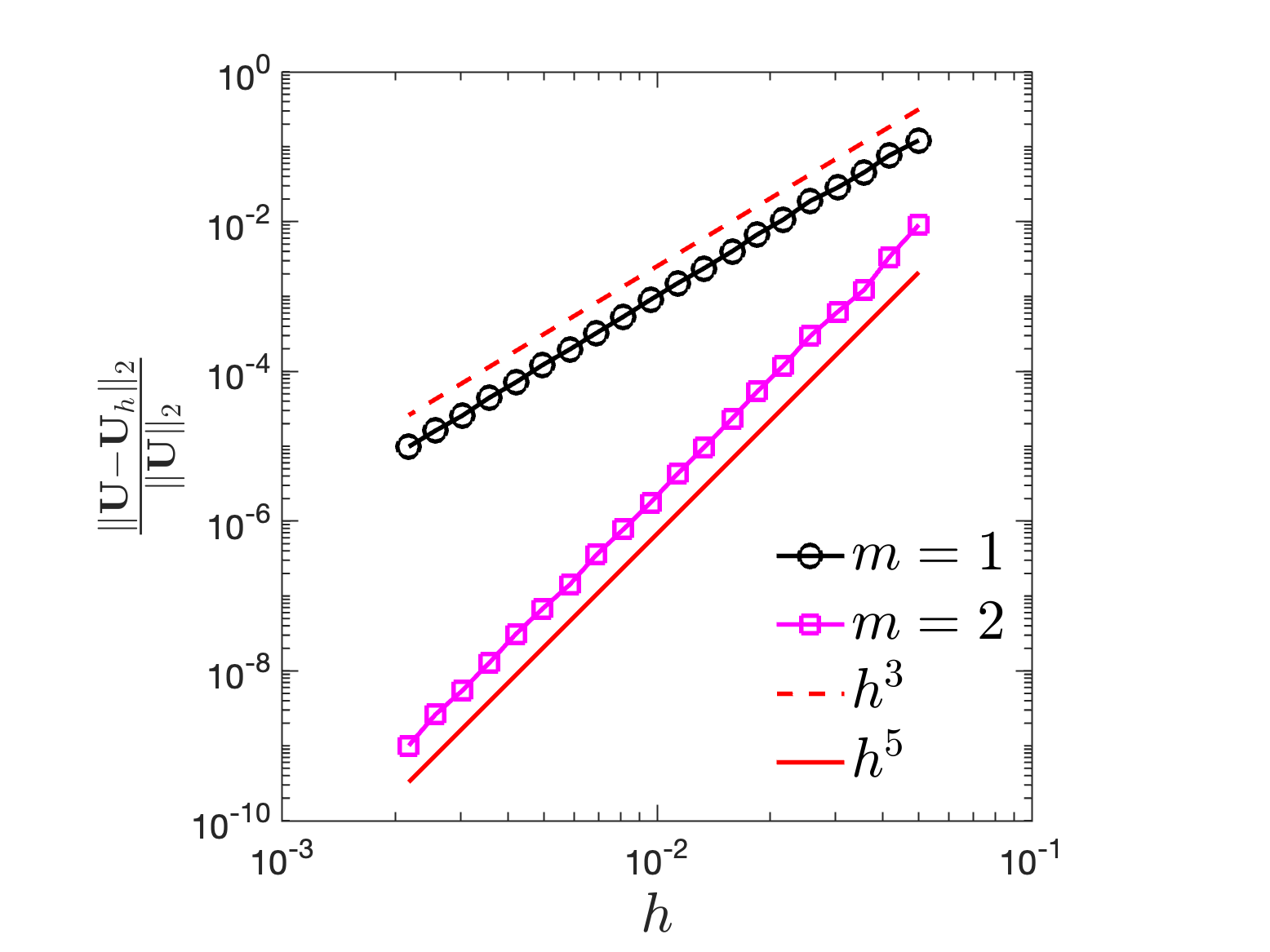}\hspace{-15pt}
		\includegraphics[width=2.5in,trim={0.0cm 0cm 1.75cm 0cm},clip]{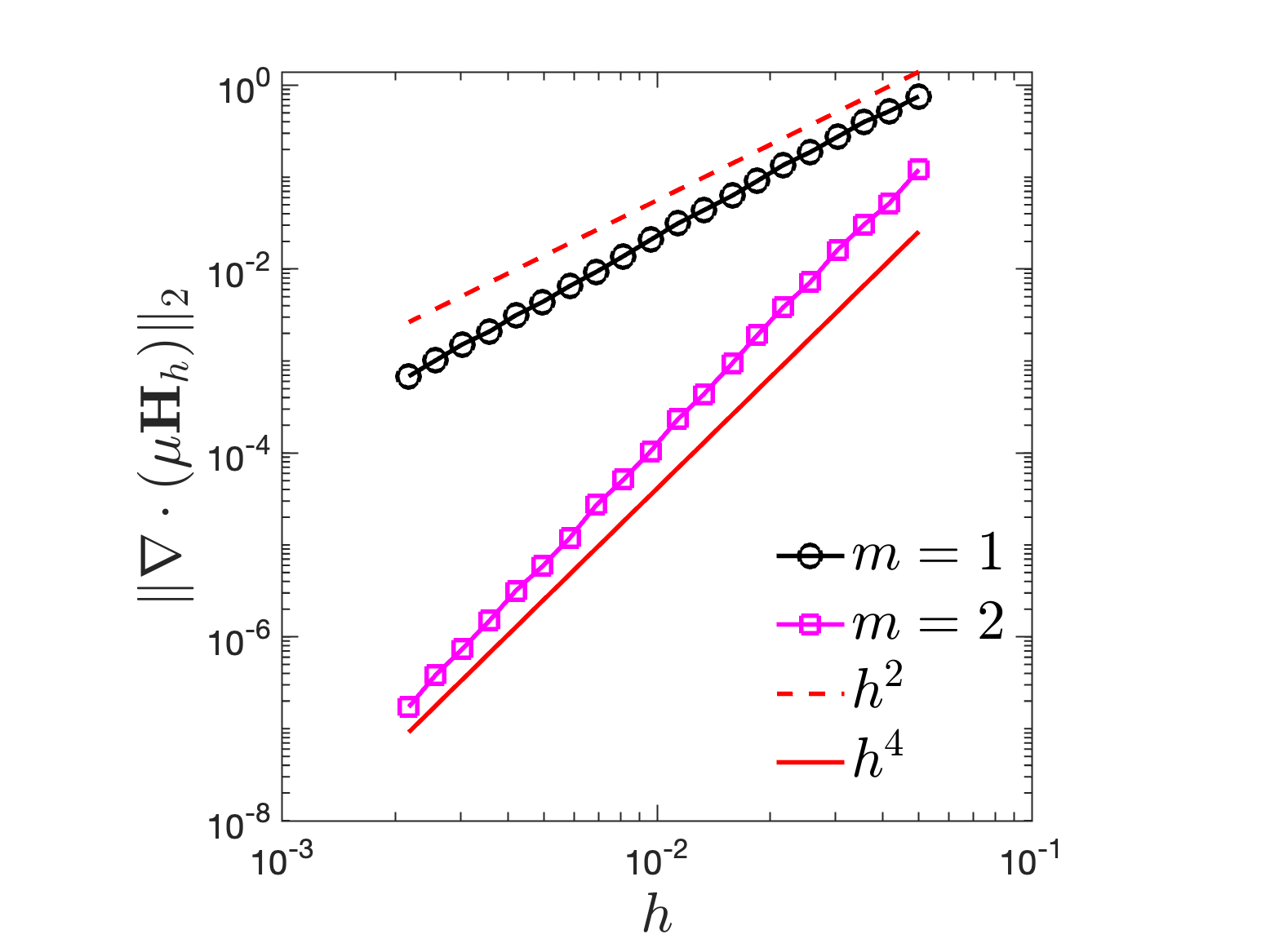}
	\end{adjustbox} 
       \caption{Convergence plots for the scattering of a magnetic dielectric cylinder problem using the third and fifth order Hermite-Taylor correction function methods. 
       The left and right plots show the convergence of the electromagnetic fields and the divergence of the magnetic fields.  Here $\mathbf{U} = [H_x, H_y, E_z]^T$.}
       \label{fig:magnetic_dielectric_conv}
\end{figure}
The approximations of electromagnetic fields at the final time are illustrated in Fig.~\ref{fig:magnetic_dielectric_caption}.
\begin{figure}   
	\centering
	\begin{adjustbox}{max width=1.0\textwidth,center}
		\includegraphics[width=2.5in,trim={2.5cm 0cm 1.75cm 0cm},clip]{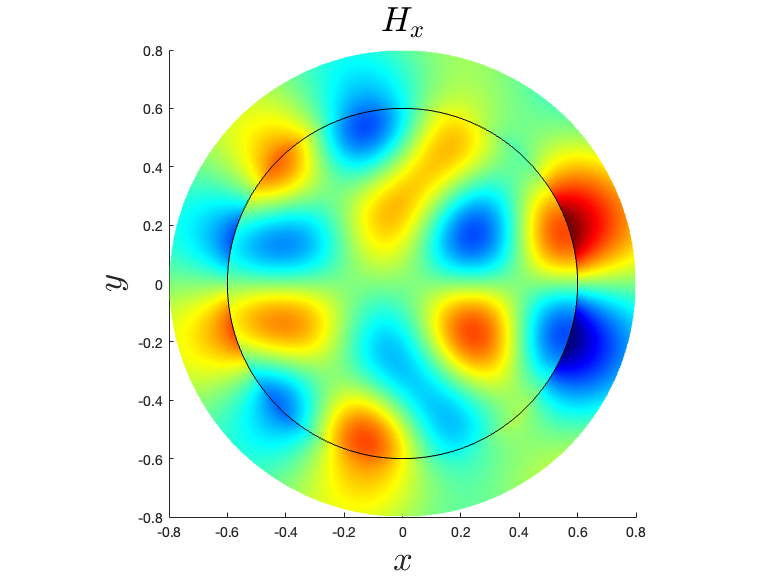}
		\includegraphics[width=2.5in,trim={2.5cm 0cm 1.75cm 0cm},clip]{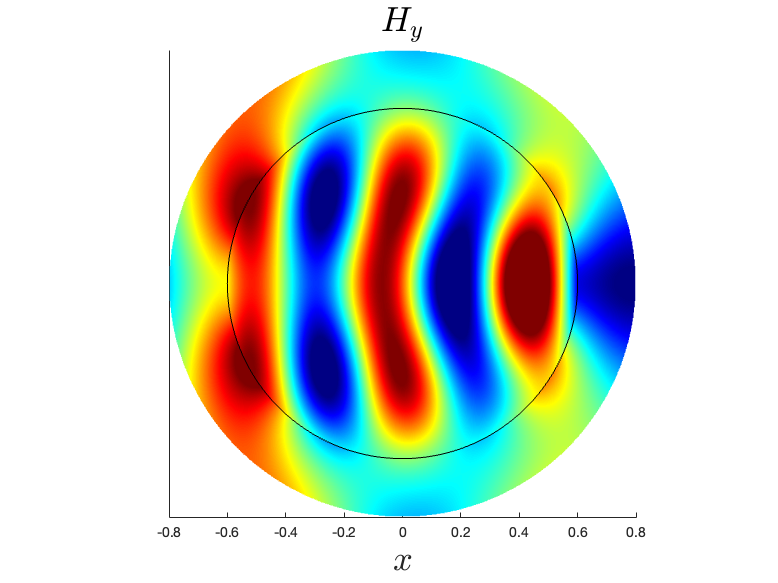}
	   	\includegraphics[width=2.5in,trim={2.5cm 0cm 1.75cm 0cm},clip]{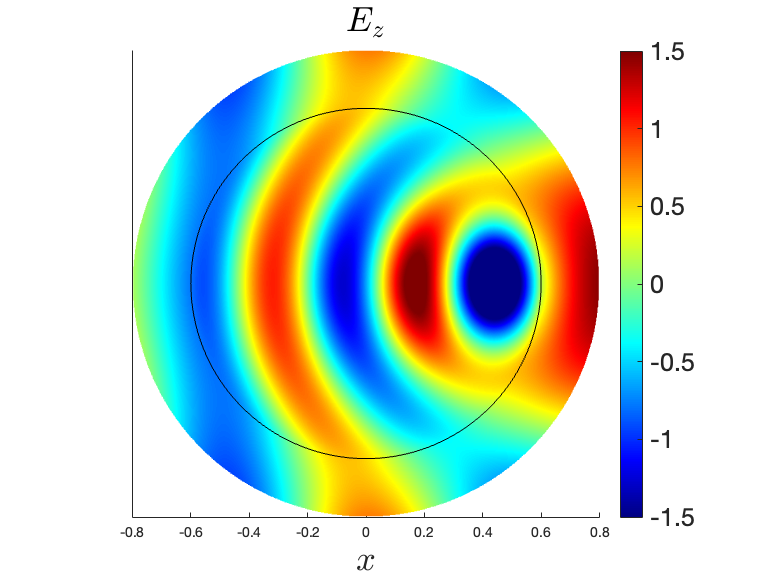}
	\end{adjustbox} 
       \caption{The components $H_x$, $H_y$ and $E_z$ at the final time $t_f=1$ for the scattering of a magnetic dielectric problem using the fifth-order Hermite-Taylor correction function method and $h=\frac{1}{460}$. The interface is represented by the black line.}
       \label{fig:magnetic_dielectric_caption}
\end{figure}
Fig.~\ref{fig:cylinder_magnetic_dielectric_slice_caption} illustrates the magnetic field components at $y=0.2$ along $x$. 
\begin{figure}   
	\centering
	\begin{adjustbox}{max width=1.0\textwidth,center}
		\includegraphics[width=2.5in,trim={0cm 0cm 1.25cm 0cm},clip]{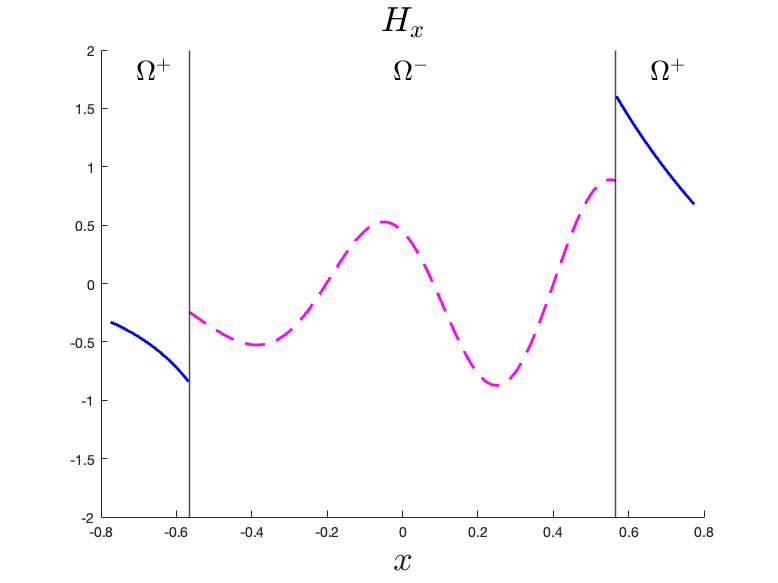}
		\includegraphics[width=2.5in,trim={0cm 0cm 1.25cm 0cm},clip]{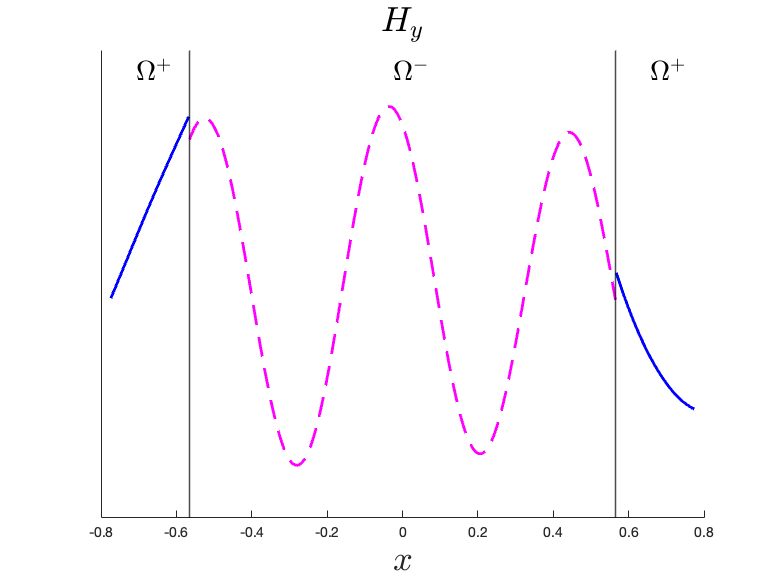}
	\end{adjustbox} 
       \caption{The components $H_x$ and $H_y$ at $y=0.2$ and the final time $t_f=1$ for the scattering of a magnetic dielectric problem using the fifth-order Hermite-Taylor correction function method and $h=\frac{1}{460}$. The interface is represented by the vertical black line. The numerical solution in $\Omega^+$ and $\Omega^-$ are respectively represented by the blue line and the dashed magenta line.}
       \label{fig:cylinder_magnetic_dielectric_slice_caption}
\end{figure}
The discontinuities at the interface are well captured by the Hermite-Taylor correction function method.

\subsubsection*{Interface Problems with Reference Solutions}

We consider the computational domain illustrated in Fig.~\ref{fig:geo_pblm_5} with $\Omega_c = [0,1]\times[0,1]$ 
	and a time interval $I=[0,1]$.  
We set $\mu^+= 1$, 
	$\epsilon^+ = 1$,
	$\mu^- = 2$ and $\epsilon^-=2.25$ 
	so $H_x$ and $H_y$ could be discontinuous at the interface.
The boundary condition on $\Gamma_1$ is given by 
\begin{equation}
	E_z(t) = e^{-\frac{(t-0.3)^2}{2\,\sigma^2}},
\end{equation}
	while interface conditions \eqref{eq:interface_cdns} are enforced on $\Gamma_2$.
Here $\sigma = 0.02$.
The initial conditions are given by trivial electromagnetic fields.
To our knowledge, 
	there is no analytical solution to this problem.
{\color{black} Hence,
	we compute a reference solution $\mathbold{U}^*$ using a Richardson extrapolation procedure pointwise with $h=\tfrac{1}{800}$ and $h=\tfrac{1}{1600}$, 
	and the fifth-order Hermite-Taylor correction function method.
This leads to reference solution with at least a sixth-order accuracy.}
The reference solution is illustrated in Fig.~\ref{fig:magnetic_dielectric_complex_geo_caption}.
\begin{figure}   
	\centering
	\begin{adjustbox}{max width=1.0\textwidth,center}
		\includegraphics[width=2.5in,trim={2.5cm 0cm 1.75cm 0cm},clip]{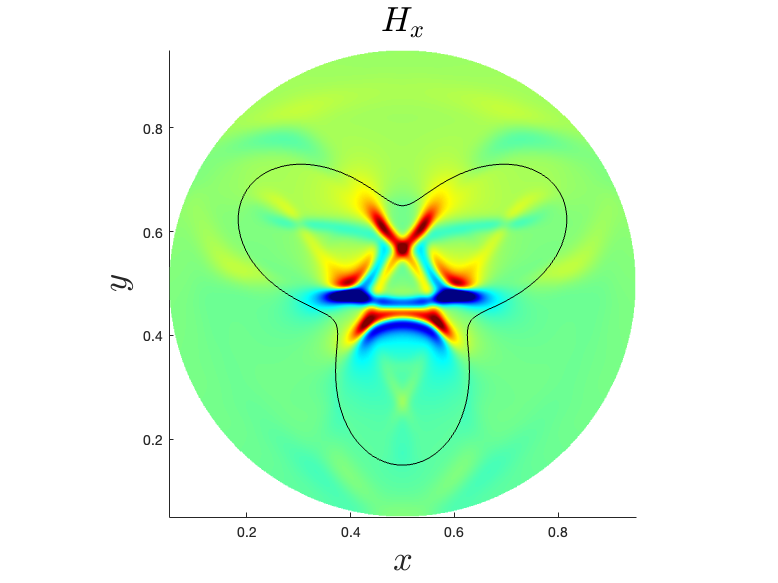}
		\includegraphics[width=2.5in,trim={2.5cm 0cm 1.75cm 0cm},clip]{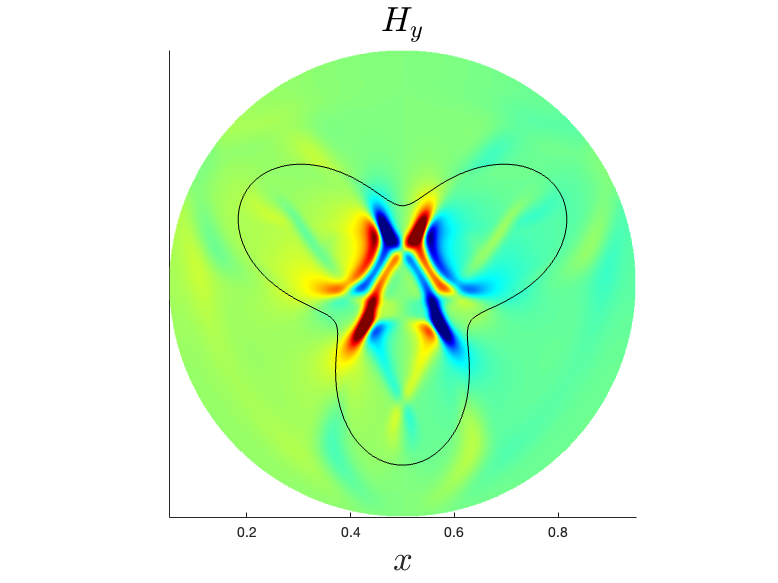}
	   	\includegraphics[width=2.5in,trim={2.5cm 0cm 1.75cm 0cm},clip]{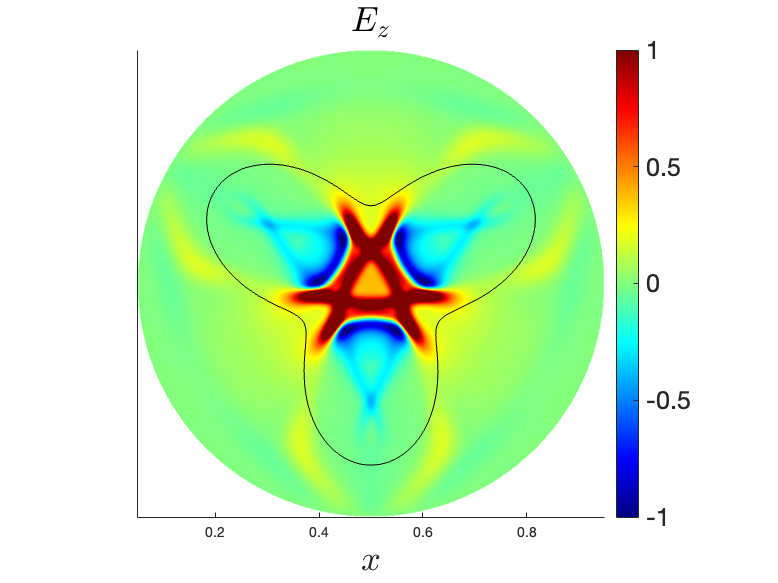}
	\end{adjustbox} 
       \caption{The components $H_x$, $H_y$ and $E_z$ at the final time $t_f=1$ for an interface problem with the geometry illustrated in Fig.~\ref{fig:geo_pblm_5} using the fifth-order Hermite-Taylor correction function method and $h=\frac{1}{1600}$. We consider a Gaussian pulse in time as the boundary condition for $E_z$. The interface is represented by the black line.}
       \label{fig:magnetic_dielectric_complex_geo_caption}
\end{figure}
Afterward,
	we estimate the relative errors by comparing the reference solution to approximations coming from meshes 
	with $h\in\{\tfrac{1}{50},\tfrac{1}{100},\tfrac{1}{200},\tfrac{1}{400},\tfrac{1}{800}\}$.
Note that all primal nodes for these meshes are also part of the reference solution mesh. 
The relative error in the $L^2$-norm is computed at the final time. 
The left plot of Fig.~\ref{fig:magnetic_dielectric_self_conv} illustrates that we obtain the expected $2\,m+1$ rates of convergence.
{\color{black} 
The $2\,m$ rates of convergence of the divergence of the magnetic field are illustrated in the left plot of Fig.~\ref{fig:magnetic_dielectric_self_conv_divH}.}
\begin{figure}   
	\centering
	\begin{adjustbox}{max width=1.0\textwidth,center}
		\includegraphics[width=2.5in,trim={0.0cm 0cm 1.75cm 0cm},clip]{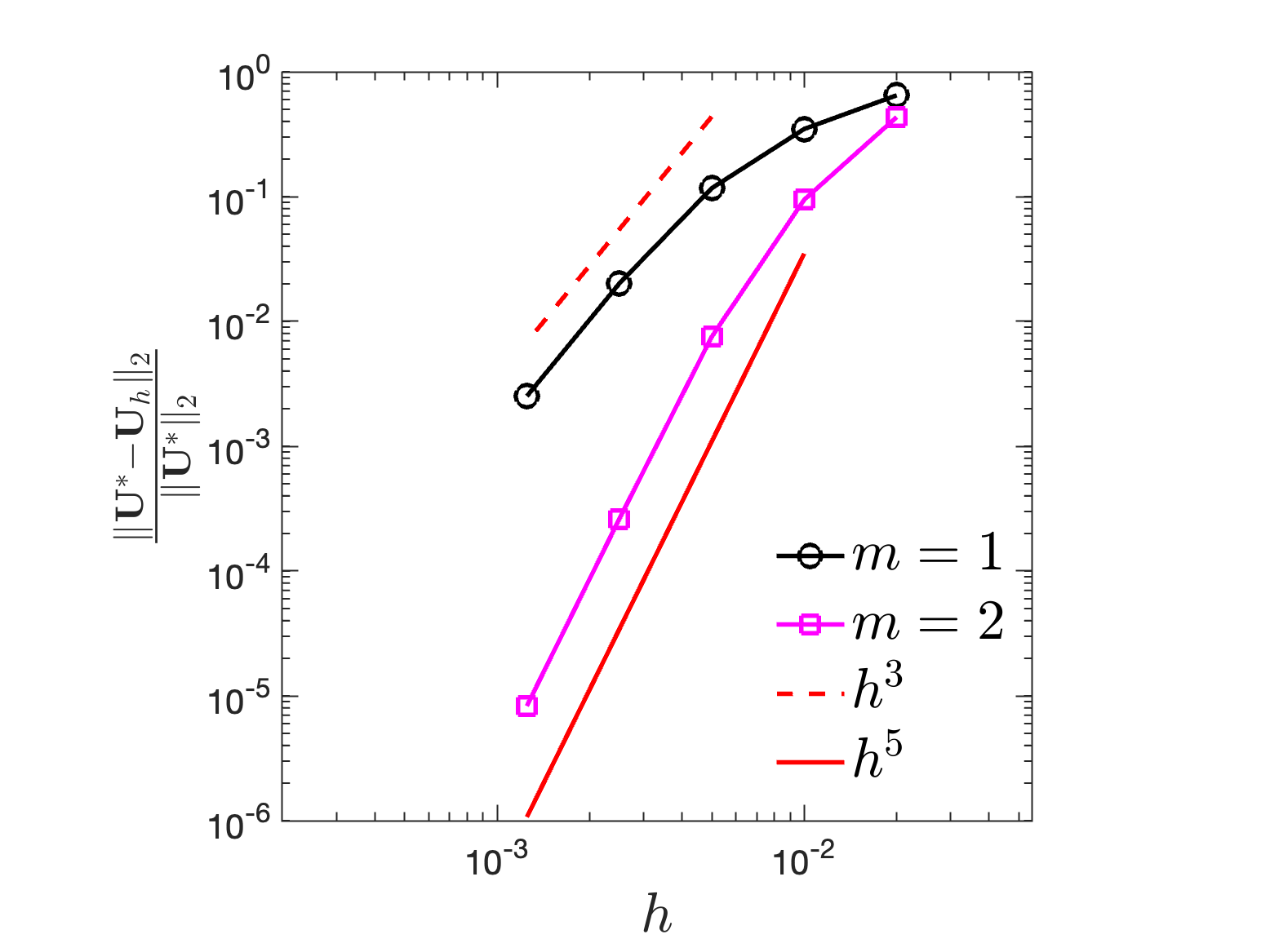}
		\hspace{-15pt}
		\includegraphics[width=2.5in,trim={0.0cm 0cm 1.75cm 0cm},clip]{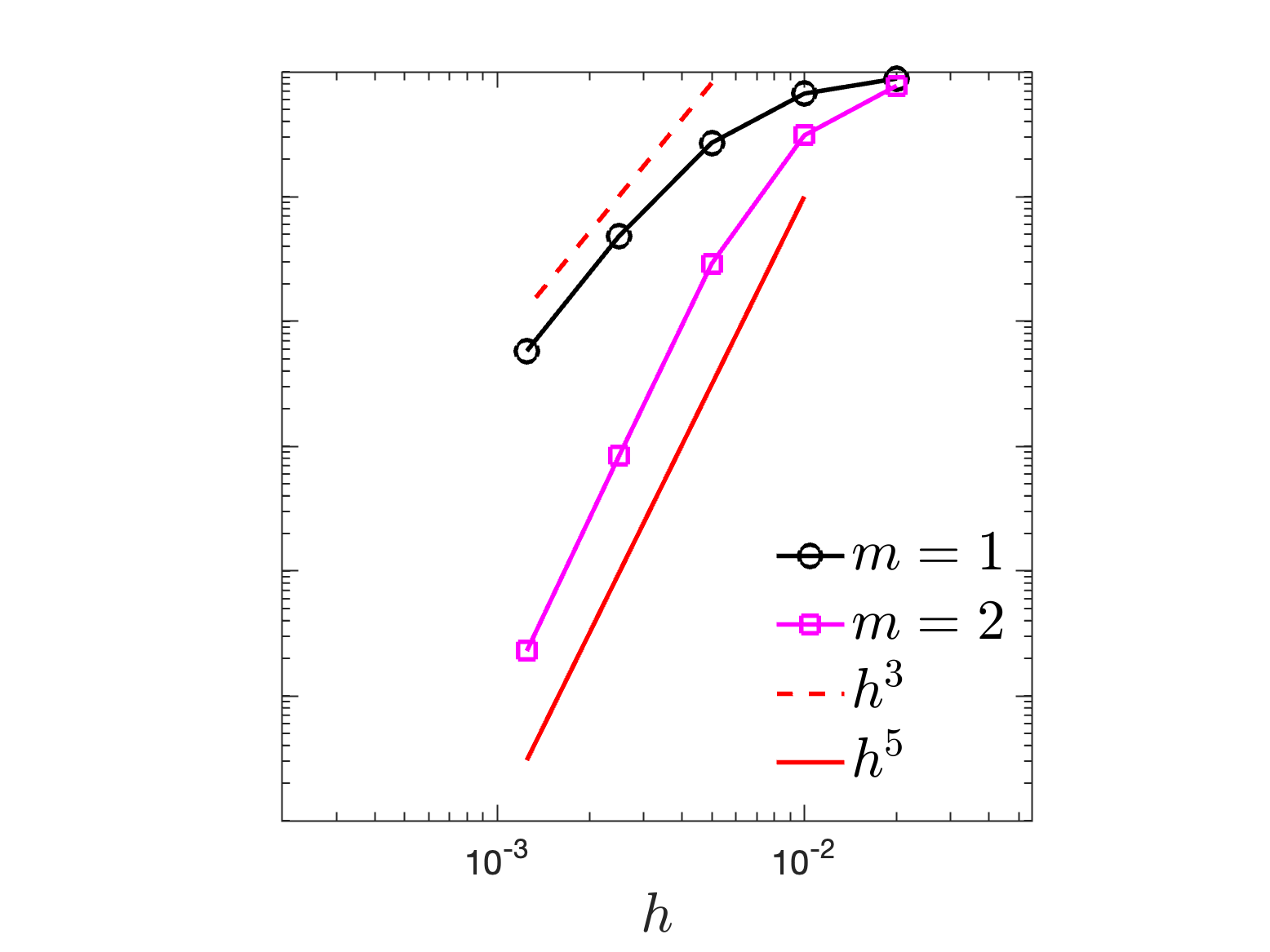}
	\end{adjustbox} 
       \caption{Self convergence plots for interface problems with the geometry illustrated in Fig.~\ref{fig:geo_pblm_5} using the third and fifth order Hermite-Taylor correction function methods. For the left plot, we consider a Gaussian pulse in time as the boundary condition for $E_z$. For the right plot, we consider a Gaussian pulse as an initial condition for $E_z$. Here $\mathbf{U} = [H_x, H_y, E_z]^T$.}
       \label{fig:magnetic_dielectric_self_conv}
\end{figure}
\begin{figure}   
	\centering
	\begin{adjustbox}{max width=1.0\textwidth,center}
		\includegraphics[width=2.5in,trim={0.0cm 0cm 1.75cm 0cm},clip]{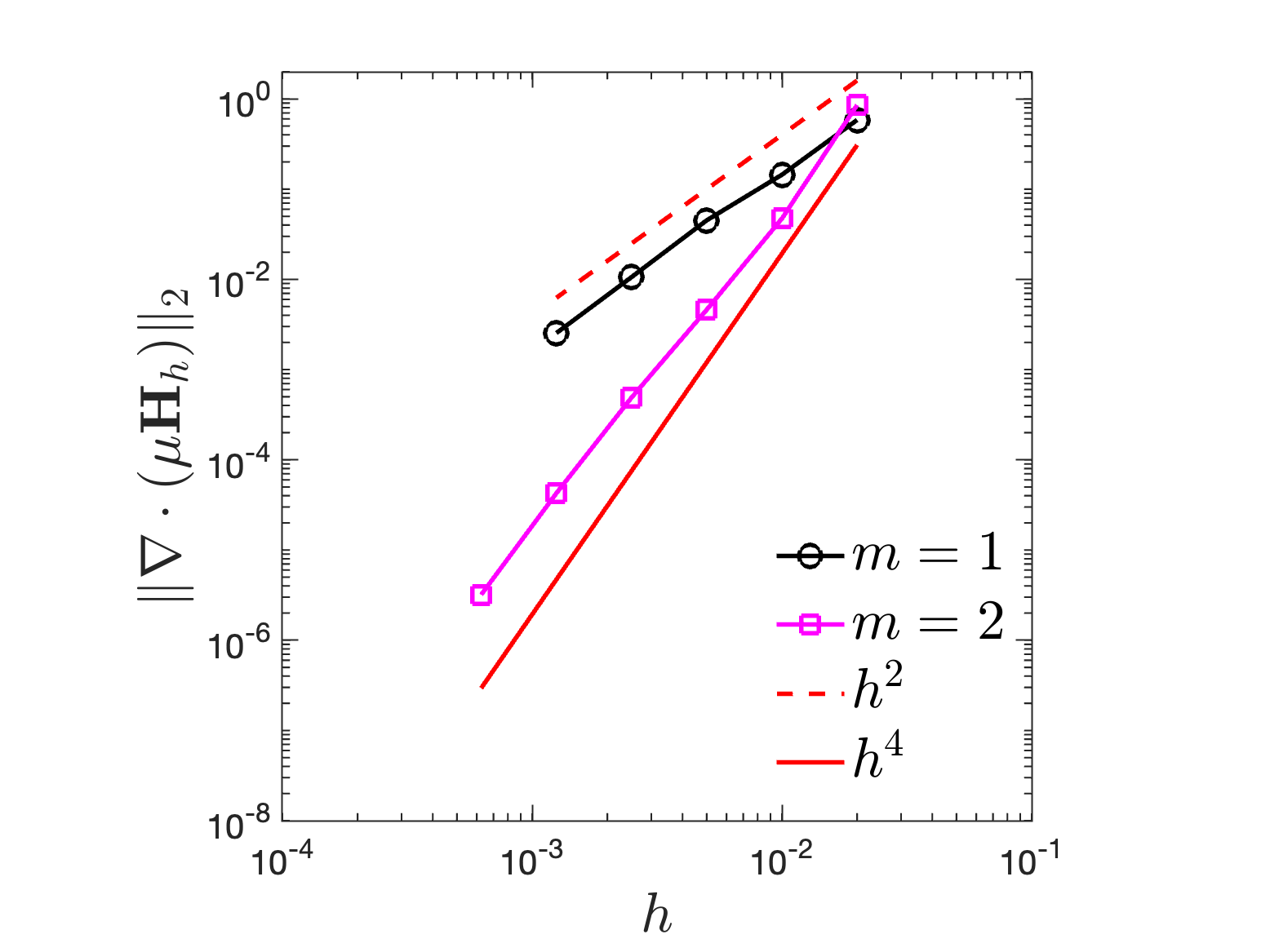}
		\hspace{-15pt}
		\includegraphics[width=2.5in,trim={0.0cm 0cm 1.75cm 0cm},clip]{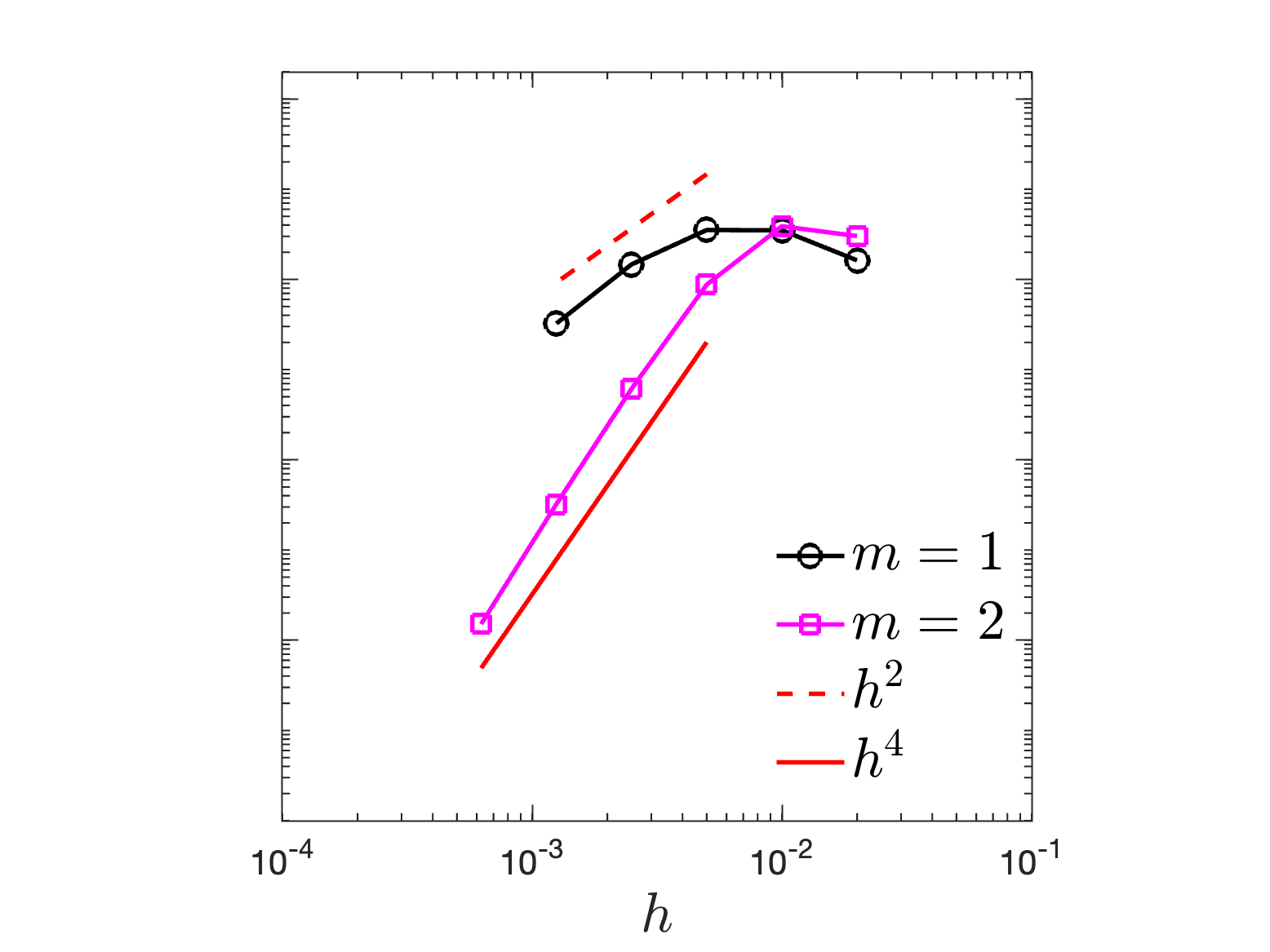}
	\end{adjustbox} 
       \caption{Self convergence plots of the divergence of the magnetic field for interface problems with the geometry illustrated in Fig.~\ref{fig:geo_pblm_5} using the third and fifth order Hermite-Taylor correction function methods. For the left plot, we consider a Gaussian pulse in time as the boundary condition for $E_z$. For the right plot, we consider a Gaussian pulse as an initial condition for $E_z$.}
       \label{fig:magnetic_dielectric_self_conv_divH}
\end{figure}

We now consider the same geometry where we enforce PEC boundary conditions on $\Gamma_1$. 
The initial conditions are $H_x = H_y = 0$ and 
\begin{equation} \label{eq:gaussian_pulse}
	E_z(x,y) = e^{\frac{-(x-0.5)^2-(y-0.5)^2}{2\,\sigma^2}}.
\end{equation}
Here $\sigma = 0.01$.
In this situation, 
	we are not aware of an analytical solution.
{\color{black} We then perform self convergence studies using the reference solution, 
	illustrated in Fig.~\ref{fig:magnetic_dielectric_complex_geo_gaussian_pulse_caption}, 
	that was computed using a Richardson procedure with $h=\tfrac{1}{800}$ and $h=\tfrac{1}{1600}$, 
    and the fifth-order Hermite-Taylor correction function method.}
\begin{figure}   
	\centering
	\begin{adjustbox}{max width=1.0\textwidth,center}
		\includegraphics[width=2.5in,trim={2.5cm 0cm 1.75cm 0cm},clip]{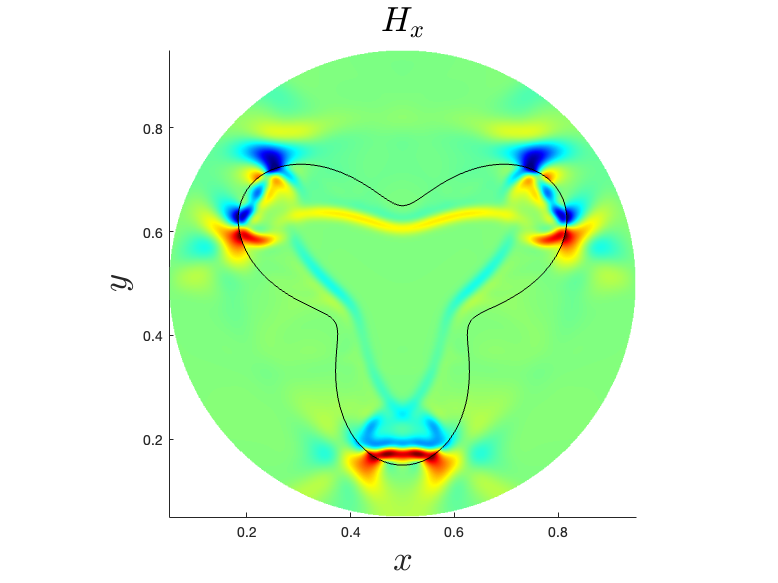}
		\includegraphics[width=2.5in,trim={2.5cm 0cm 1.75cm 0cm},clip]{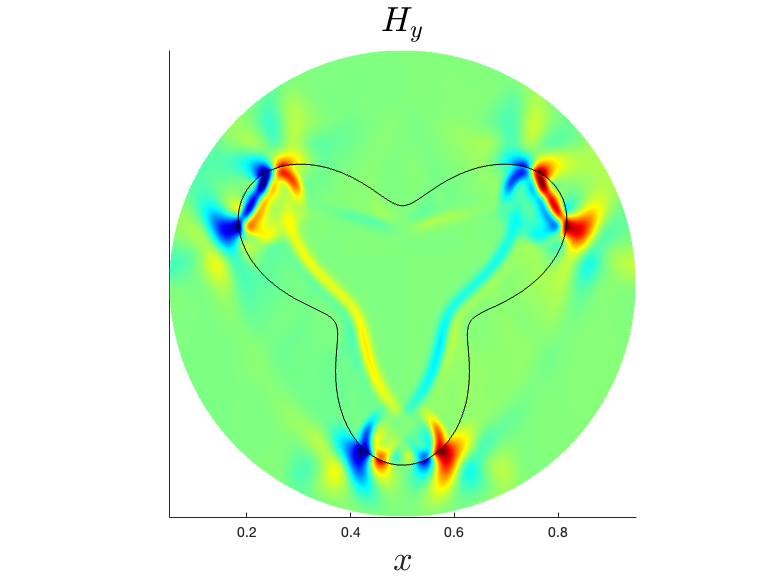}
	   	\includegraphics[width=2.5in,trim={2.5cm 0cm 1.75cm 0cm},clip]{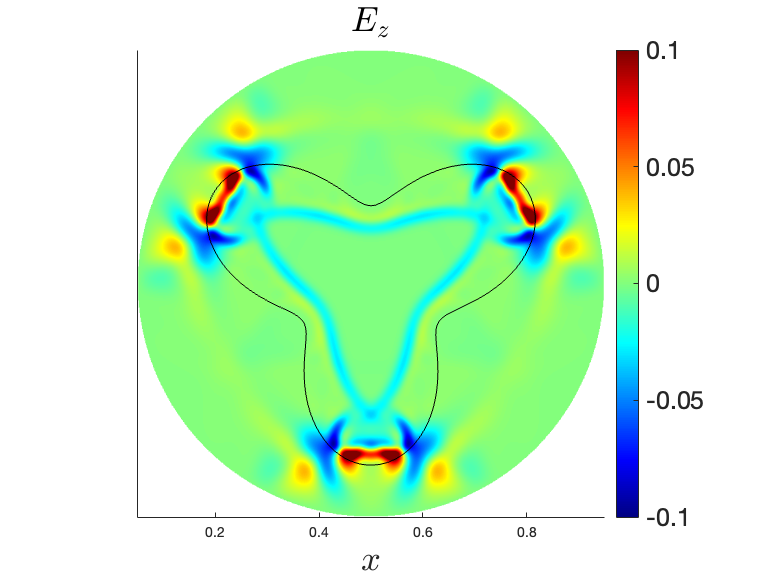}
	\end{adjustbox} 
       \caption{The components $H_x$, $H_y$ and $E_z$ at the final time $t_f=1$ for an interface problem with the geometry illustrated in Fig.~\ref{fig:geo_pblm_5} using the fifth-order Hermite-Taylor correction function method and $h=\frac{1}{1600}$. We consider a Gaussian pulse as an initial condition for $E_z$. The interface is represented by the black line.}
       \label{fig:magnetic_dielectric_complex_geo_gaussian_pulse_caption}
\end{figure}
The right plot of Fig.~\ref{fig:magnetic_dielectric_self_conv} illustrates the self convergence plots in the $L^2$-norm.
We obtain the expected $2\,m+1$ rates of convergence.
{\color{black} 
The right plot of Fig.~\ref{fig:magnetic_dielectric_self_conv_divH} shows the $2\,m$ rates of convergence of the divergence of the magnetic field.}

\subsubsection{Increasing the Order of the Interior Method}

As a final example,
    we increase the order of the interior method, 
    here the Hermite-Taylor method. This could be advantageous in cases where waves propagate many wavelengths in the interior domain.
This is possible since we can compute the derivatives through order $k$ when 
    correction function polynomials of degree $k$ are used, 
    allowing a Hermite-Taylor method with $m\leq k$.
Although it is not used here, 
	note that, 
	for localized pulses,
	the $p$-adaptive algorithm introduced in \cite{Chen2012} could be used for the interior method, 
	far away from the boundary. Then, although the extent to which the order can be increased adjacent to the boundary is limited, there would be no fundamental limits to the order achieved far from the boundaries.
We denote a Hermite-Taylor correction function method that uses derivatives through order $m$ and 
    correction function polynomials of degree $k$ as a $(k,m)$-method, 
    leading to a $(k+1)$-order method near the boundary and 
    $(2m+1)$-order method in the interior.
Here we consider the situation when $k/2 \leq m \leq k$ with $k=2-4$.
We set $N_d = k$.

We first perform long time simulations with $h=1/20$ using the same setup as 
    for the embedded boundary problems.
According to the results in subsection~\ref{sec:stability_2d},
    the $(2,1)$ and $(4,2)$ methods are numerically stable for 
    respectively a CFL constant smaller than $1$ and $0.8$.
For the remaining methods, 
    we observe that the $(2,2)$, $(3,2)$, $(3,3)$ and $(4,3)$ methods are stable under a CFL constant smaller than $0.9$, $1$, $0.5$ and $0.6$, 
    while the $(4,4)$-method is unstable.

We now consider the computational domain $\Omega_c=[0,1]\times[0,1]$. 
The embedded boundary $\Gamma$ is a circle with a radius of $0.45$ and 
    centered at $(0.5,0.5)$ that encloses the physical domain $\Omega$. 
We enforce PEC boundary conditions on the boundary $\Gamma$ and 
    consider the initial conditions $H_x=H_y=0$ and $E_z$ is given by \eqref{eq:gaussian_pulse}.
The physical parameters are $\mu=1$ and $\epsilon=1$.
The CFL constant are $0.9$, $0.85$, $0.9$, $0.45$, $0.75$ and $0.55$ for respectively $(2,1)$, $(2,2)$, $(3,2)$, $(3,3)$, $(4,2)$ and $(4,3)$ methods. 
We investigate the energy conservation on the time interval $I = [0,20]$ with $h=\Delta x = \Delta y = 1/200$. 
The energy at time $t$ is 
\begin{equation}
E(t) = \int\limits_\Omega \epsilon\|\mathbold{E}(\mathbold{x},t)\|^2 + \mu \| \mathbold{H}(\mathbold{x},t) \|^2 \,\mathrm{d}\mathbold{x}.
\end{equation}
Fig.~\ref{fig:energy} illustrates the evolution in time of the ratio $E(t)/E(0)$ for different $(k,m)$ methods. 
Based on these results, 
	the energy dissipates as the time increases, 
	demonstrating the stability of these methods. 	
As the degree $k$ of the correction functions increases,
	the energy is better conserved as expected.  
The decay in the energy due to the interior method for $m=2-3$ is negligible.
In all the cases, 
	the decay in the energy is mainly due to the correction function method when the waves interact with the boundary,
	representing by a sharp drop of the energy in Fig.~\ref{fig:energy}. 
\begin{figure}   
	\centering
	\begin{adjustbox}{max width=1.0\textwidth,center}
		\includegraphics[width=2.5in,trim={1.8cm 0cm 1.75cm 0cm},clip]{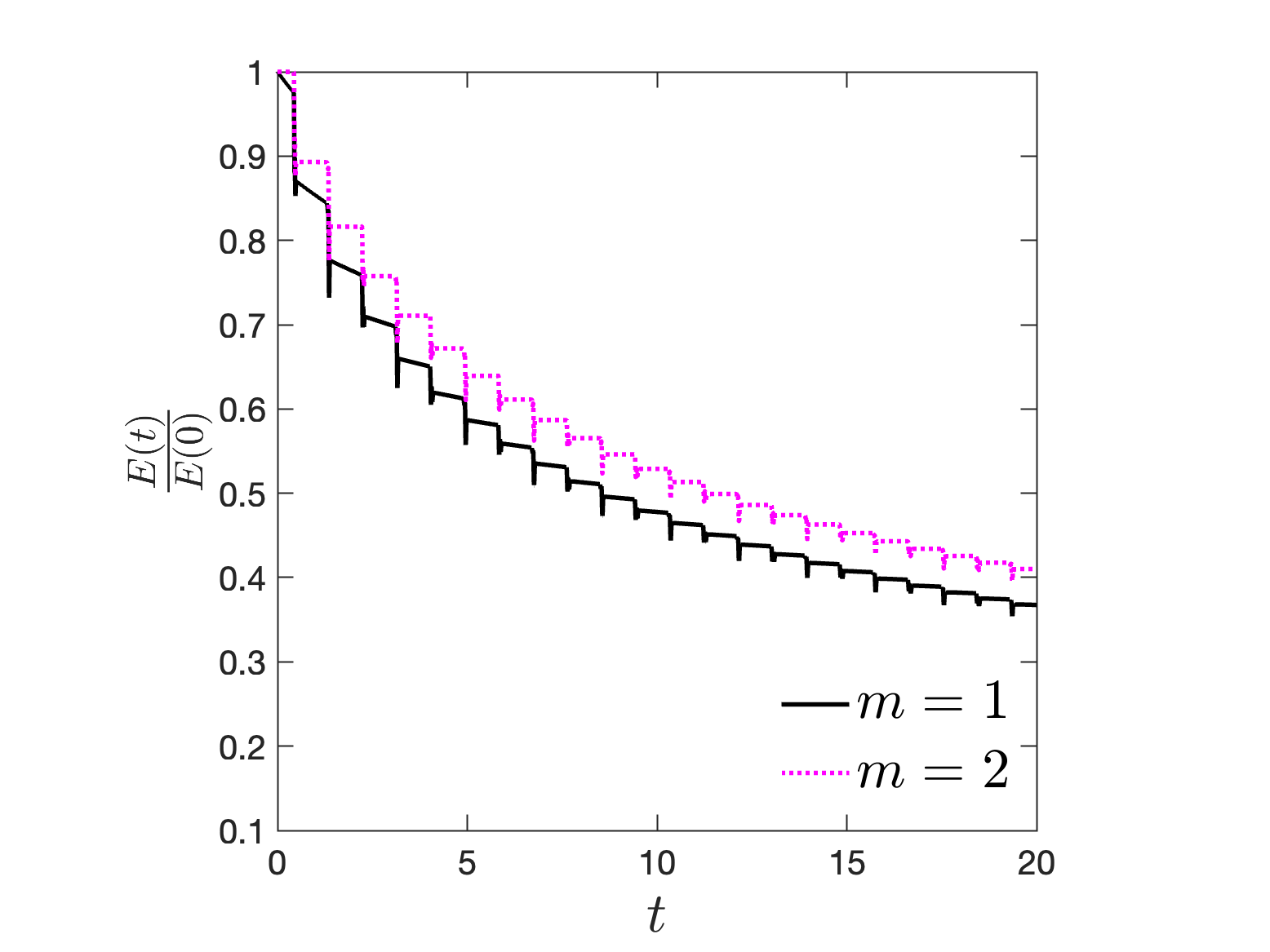}
		\includegraphics[width=2.5in,trim={1.8cm 0cm 1.75cm 0cm},clip]{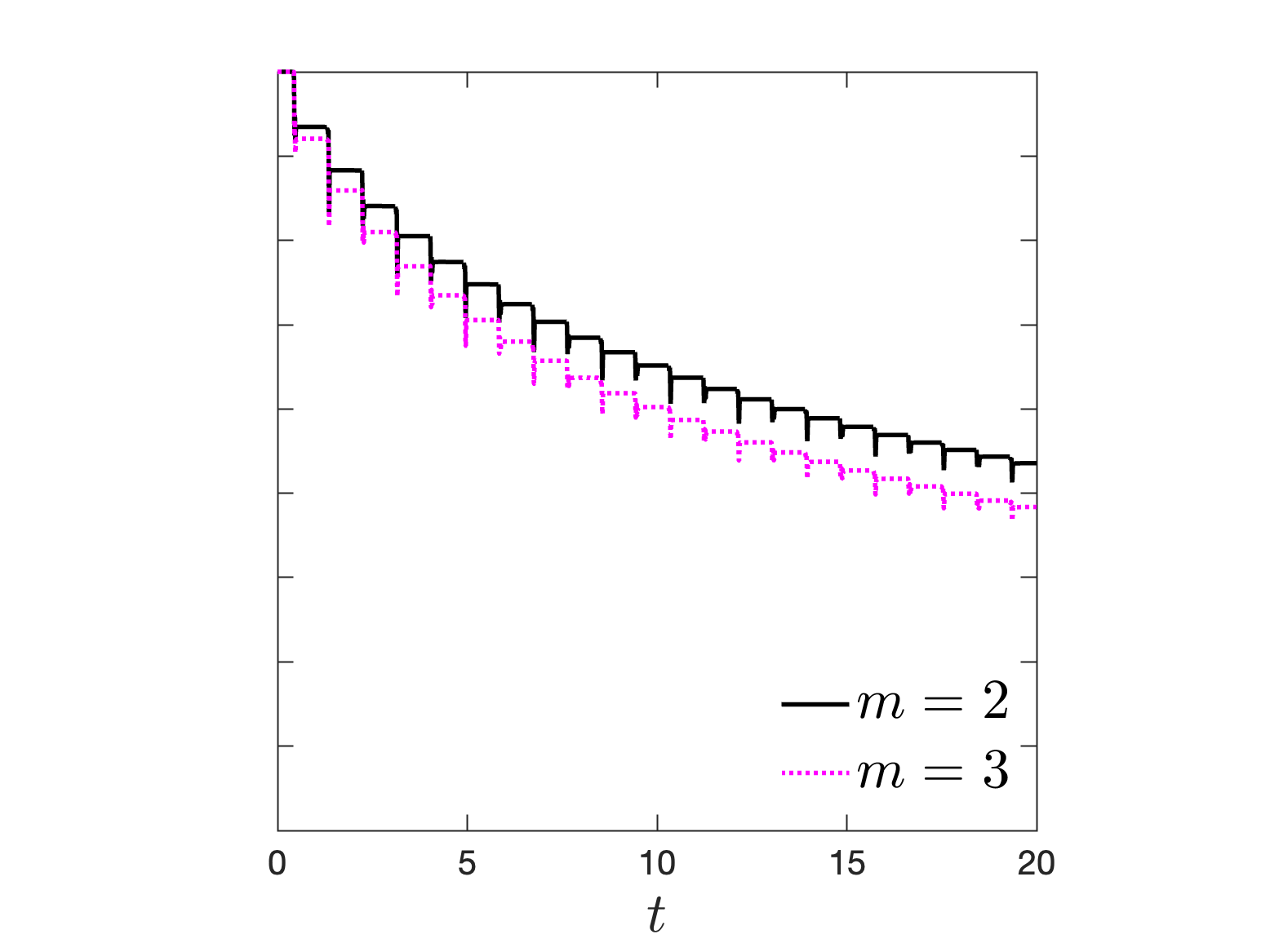}
	   	\includegraphics[width=2.5in,trim={1.8cm 0cm 1.75cm 0cm},clip]{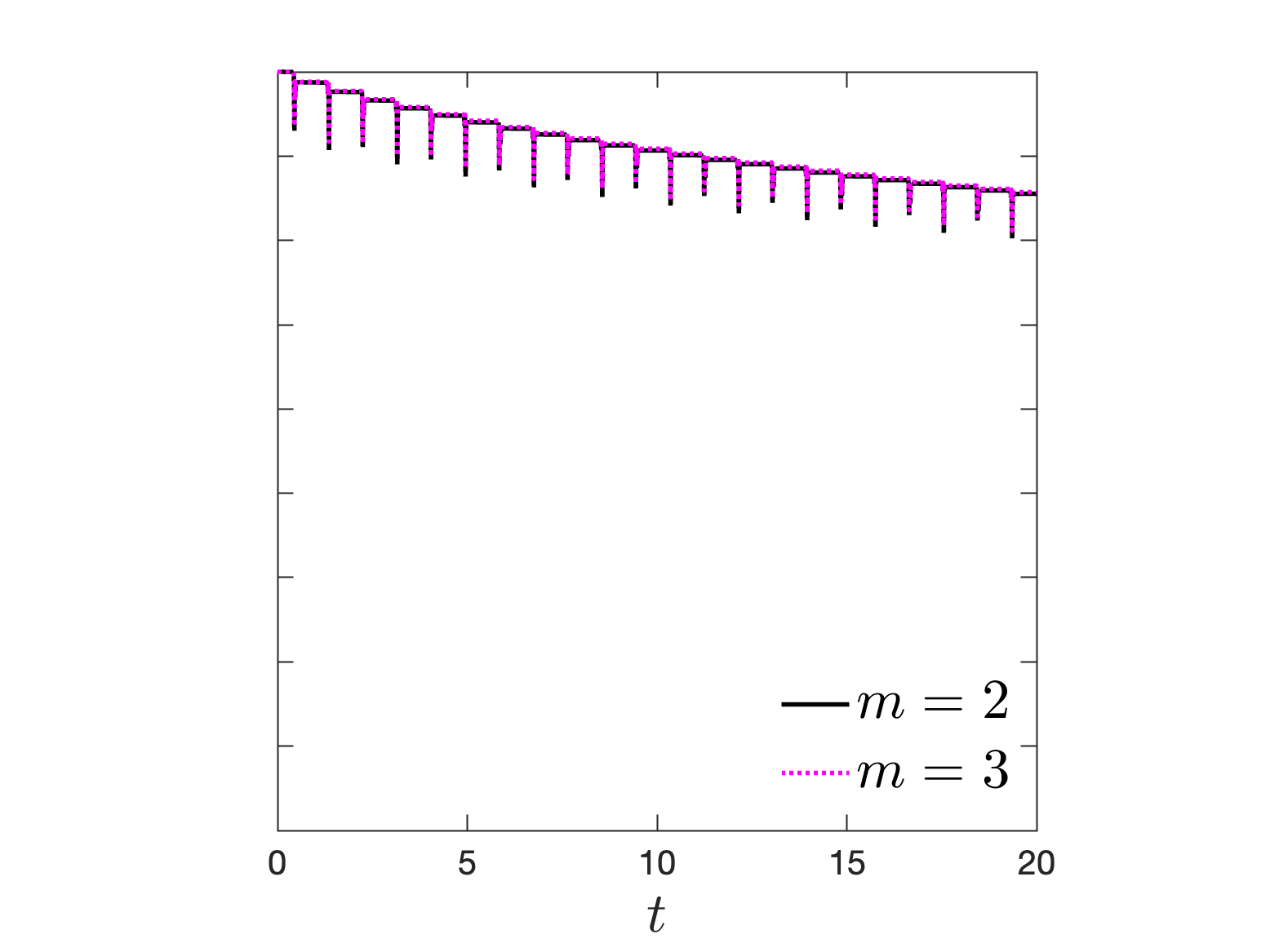}
	\end{adjustbox} 
       \caption{The evolution of the energy as a function of time for different $(k,m)$ methods. The left, middle and right plots are for respectively $k=2$, $k=3$ and $k=4$.}
       \label{fig:energy}
\end{figure}

\section{Conclusion}

In this work, 
	we have proposed a novel Hermite-Taylor correction function method to handle embedded boundary and interface 
	conditions for Maxwell's equations.
We take advantage of the correction function method to update the numerical solution at the nodes 
	where the Hermite-Taylor method cannot be applied.
To do so,
	we minimize a functional that is a square measure of the residual of Maxwell's equations, 
	the boundary or interface conditions,
	and the polynomials approximating the electromagnetic fields coming from the Hermite-Taylor method.
The stability condition of the Hermite-Taylor correction function method is improved by enforcing the time derivatives,
    converted into spatial derivatives, 
    of the boundary and interface conditions.
The approximations of the electromagnetic fields and their required derivatives are then updated using 
	the correction functions resulting from the minimization procedure.
{\color{black}
For the embedded boundary problem in one space dimension, 
    we were able to achieve up to a ninth-order rate of convergence and the methods are stable under a CFL constant between 0.7 and 1.
In two space dimensions, 
    numerical examples suggest that the third and fifth order Hermite-Taylor correction function methods 
	are stable under a CFL constant of 1 and 0.8 respectively.
The range of $m$ that can be used and therefore the order of the overall method are limited by the large condition number of the CF matrices coming from the minimization problem when $m>2$ in the multi-dimensional case.}
Note that we were able to increase the order of the interior method, 
	here the Hermite-Taylor method, 
	up to seven without significantly impacting the CFL constant.
The accuracy of the Hermite-Taylor correction function method was verified.
The proposed method achieves high order accuracy even with interface problems with discontinuous solutions. 
Finally, 
	this method can be easily adapted to other first order hyperbolic problems.
	
\section*{Declarations}
\section*{Funding} 
This work was supported in part by NSF Grants DMS-2208164,  DMS-2210286 and DMS-2012296. Any opinions, findings, and conclusions or recommendations expressed in this material are those of the authors and do not necessarily reflect the views of the NSF.
\section*{Conflicts of interest/Competing interests}
On behalf of all authors, the corresponding author states that there is no conflict of interest.




\bibliographystyle{elsarticle-num}
\bibliography{references}







\end{document}